\documentclass[a4paper,reqno,11pt]{article}

\usepackage[english]{babel}
\usepackage{epsfig}
\usepackage{amsfonts}
\usepackage{amsmath}
\usepackage{amssymb}
\usepackage{amsthm}
\usepackage{mathrsfs}
\usepackage{color,amscd}
\usepackage{currfile}
\usepackage{hyperref}
\usepackage{euscript}
\usepackage{mathtools}
\usepackage{tikz}
\usetikzlibrary{arrows.meta}
\usepackage{tikz-cd}
\usepackage[all]{xy}

\let\oldcolor\color
\renewcommand{\color}[1]{\oldcolor{#1}}  

\usepackage{amsthm}

\newtheorem{theorem}{Theorem}[section]
\newtheorem{lemma}[theorem]{Lemma}
\newtheorem{proposition}[theorem]{Proposition}
\newtheorem{corollary}[theorem]{Corollary}

\newtheorem{property}[theorem]{Property}
\newtheorem{assumption}[theorem]{Assumption}

\theoremstyle{definition}
\newtheorem{definition}[theorem]{Definition}
\newtheorem{remark}[theorem]{Remark}
\newtheorem{example}[theorem]{Example}
\newtheorem{examples}[theorem]{Examples}

\newcommand{\C}{\mbb{C}}
\newcommand{\I}{\mc{I}}
\newcommand{\J}{\mc{J}}
\newcommand{\N}{\mbb{N}}

\newcommand{\R}{\mbb{R}}
\newcommand{\Z}{\mbb{Z}}

\newcommand{\hh}{\mbb{H}}

\newcommand{\oo}{\mbb{O}}
\newcommand{\OO}{\Omega}

\newcommand{\mb}{\mathbf}
\newcommand{\mbb}{\mathbb}
\newcommand{\mc}{\mathcal}
\newcommand{\mi}{\mathit}

\newcommand{\mr}{\mathrm}

\newcommand{\mscr}{\mathscr}

\newcommand{\pr}{\prime}
\newcommand{\sss}{\scriptscriptstyle}

\newcommand{\compl}{\complement}

\newcommand{\sq}{\mbb{S}}

\newcommand{\bs}{{\tiny $\blacksquare$}}
\newcommand{\prd}{\prod}

\newcommand{\ui}{i}
\newcommand{\pa}{\mc{P}}
\newcommand{\dsim}{\bigtriangleup}
\newcommand{\dibar}{\overline\partial}

\def\dd#1#2{\dfrac{\partial#1}{\partial#2}} 
\def\ds#1#2{\frac{\partial#1}{\partial#2}}

\newcommand\IM{\operatorname{Im}}
\newcommand\RE{\operatorname{Re}}

\newcommand{\cS}{\mathbb{S}}
\newcommand\vs[1]{{#1}_s^\circ}
\newcommand{\tenso}{\odot}

\newcommand\HH{\mathbb H}

\newcommand{\SD}{\EuScript{D}}

\def\dd#1#2{\dfrac{\partial#1}{\partial#2}}

\title{\bf Slice regular functions in several variables}

\author{
\textsc{Riccardo Ghiloni and Alessandro Perotti}}

\date{
{\small
\textit{Department of Mathematics, University of Trento, I-38123, Povo-Trento, Italy}} 
\\
{\small \texttt{ghiloni@science.unitn.it} $\qquad$ \texttt{perotti@science.unitn.it}}
}

\begin{document}


\maketitle



\begin{abstract}
In this paper, we lay the foundations of the theory of slice regular functions in several variables ranging in any real alternative $^*$-algebra, including quaternions, octonions and Clifford algebras. This theory is an extension of the classical theory of holomorphic functions in several complex variables.

\vspace{.5em}

\noindent \emph{2010 MSC:} Primary 30G35; Secondary 32A30; 17D05


\noindent \emph{Keywords:}  Slice regular functions; Functions of  hypercomplex variables; Cauchy integral formula; Quaternions; Clifford algebras; Octonions; Real alternative algebras
\end{abstract}


\tableofcontents


\section{Introduction}

The theory of slice regular functions of one variable in a real alternative $^*$-algebra is now well developed. It was introduced firstly for functions of one quaternionic variable by Gentili and Struppa in \cite{GeSt2006CR,GeSt2007Adv} and then extended to octonions in \cite{GeStRocky} and to Clifford algebras in \cite{CoSaSt2009Israel}. In \cite{GhPe_Trends,AIM2011}, a new approach to slice regularity, based on the concept of stem function, allowed to extend the theory to any real alternative $^*$-algebra $A$ of finite dimension.

The original definition \cite{GeSt2006CR,GeSt2007Adv} of slice regularity for a quaternion-valued function $f$, defined on an open domain $\OO$ of the algebra $\HH$ of quaternions, requires that, for every imaginary unit $J\in\HH$, the restriction of $f$ to the complex line generated by $J$ is holomorphic with respect to the complex structure defined by left multiplication by $J$. The approach taken in \cite{AIM2011,GhPe_Trends} allows to embed the class of slice regular functions into a larger class, that of slice functions, on which no holomorphicity condition is assumed. We refer to the monograph \cite{GeStoSt2013} for a survey of slice analysis in one quaternionic variable and to the papers \cite{AlgebraSliceFunctions,singular,DivisionAlgebras} for a recent account of the theory on real alternative $^*$-algebras. 

In the present paper we propose a generalization of slice analysis to several variables in a real alternative $^*$-algebra $A$.  Our function theory includes, in particular, the class of polynomials in several (ordered) variables with right coefficients in $A$. Our approach is based on the concept of stem functions of several variables and on the introduction  of a family of commuting complex structures on the real vector space $\R^{2^n}$.
For $A=\mathbb H$, several variables have been investigated also by Colombo, Sabadini and Struppa \cite{CoSaSt2013Indiana}. Their approach via stem functions is similar to ours, but the definition of regularity is different, as we will see in Section \ref{sec:Slice regular functions}. 
For $A=\mathbb O$, the algebra of octonions, a  slice functions theory of several variables has been proposed by Ren and Yang \cite{Ren_yang_2020}. The major difference with our theory is that the authors define slice functions 
on a class of non-open subsets of the space $\mathbb O^n$, where the octonionic variables associate and commute. In \cite{Ren_yang_2020} the  authors state as challenging the possibility to establish the theory on 
open subsets of $\mathbb O^n$. 
For $A=\C$, as one may expect, the slice function theory in $n$ variables reduces to the classic theory of several complex variables on domains $D$ of $\C^n$, with the unique restriction that $D$ must be assumed invariant with respect to complex conjugation in every variable $z_1,\ldots,z_n$.

Some of the results proved here were presented in \cite{GhPe_ICNPAA} for the case $A=\R_m$, the real Clifford algebra of signature $(0,m)$, and in \cite{MoscowSeveral} for the general case of real alternative $^*$-algebras.

We describe the structure of the paper. In the Introduction we give some preliminaries and recall the main definitions of the one variable slice function theory. Then we present without proofs the principal results in two quaternionic variables, where the exposition is simpler but sufficiently representative of the theory, at least for the associative case.  In Section~\ref{sec:Slice functions}, we introduce the \emph{stem functions} of several variables in $A$ and define the induced \emph{slice functions}. We prove a representation formula and the identity principle. We generalize to $n$ variables the concepts of \emph{spherical  value} and \emph{spherical derivatives} and give the relation between sliceness in $n$ variables and sliceness in one variable. We then prove the smoothness properties of slice functions. We study also the multiplicative structures on slice functions induced by pointwise products of stem functions, and we investigate some special real subalgebras of slice functions. Section~\ref{sec:Slice regular functions} is dedicated to slice regularity. After giving the definition of a family of commuting complex structures on $A\otimes \R^{2^n}$, the concept of \emph{slice regular function} of several variables is introduced. All polynomials (with ordered monomials) turn out to be slice regular functions. 
We study the real dimension of the zero set of polynomials in the quaternionic and octonionic cases and give some results about the zero set of polynomials with Clifford coefficients. In particular, we prove that these zero sets are nonempty for nonconstant polynomials.
We show that slice regularity in several variables has an interpretation, by means of the spherical value and spherical derivatives, in terms of slice regularity in one variable. We investigate Leibniz's rule and we prove the stability of slice regularity under the so-called slice tensor product of slice functions. We show the relation between slice regularity and expansions in (ordered) power series on products of open balls in $A$ centered in the origin. Finally, we define a \emph{slice Cauchy kernel} associated to any given slice regular function, and obtain a Cauchy integral formula. In the associative case, we are able to define a universal slice Cauchy kernel, and express it in terms of pointwise operations in $A$.

\subsection{Preliminaries}

Let $A$ be a \emph{real algebra with unity $1\ne0$}. Assume that $A$ is \emph{alternative}, i.e. $x^2y=x(xy)$ and $(yx)x=yx^2$ for all $x,y\in A$. A theorem of E.\ Artin asserts that the subalgebra generated by any two elements of $A$ is associative. The real multiples of $1$ in $A$ are identified with the field $\R$ of real numbers. Assume that $A$ is a *-algebra, i.e., it is equipped with a real linear anti-involution $A\to A$, $x\mapsto x^c$, such that $(xy)^c=y^cx^c$ for all $x,y\in A$ and  $x^c=x$ for $x$ real. Let $t(x):=x+x^c\in A$ be the \emph{trace} of $x$ and $n(x):=xx^c\in A$  the \emph{(squared) norm} of $x$. 
Let
\[
\cS_A:=\{J\in A : t(x)=0,\ n(x)=1\}
\]
be the `sphere' of the imaginary units of $A$ compatible with the *-algebra structure of $A$. Assuming $\cS_A\ne\emptyset$, one can consider the \emph{quadratic cone} of $A$, defined as the subset of $A$
\[
Q_A:=\bigcup_{J\in \cS_A}\C_J,
\]
where $\C_J=\langle 1,J\rangle$ is the complex `slice' of $A$ generated by $1,J$ as a vector subspace or, equivalently, by $J$ as a subalgebra. It holds $\C_J\cap\C_K=\R$ for each $J,K\in\cS_A$ with $J\ne\pm K$. The quadratic cone is a real cone invariant w.r.t.\ translations along the real axis. 
Observe that $t$ and $n$ are real-valued on $Q_A$ and that $Q_A=A$ if and only if $A$ is isomorphic as a real $^*$-algebra to one of the division algebras $\C,\HH,\mathbb O$ with the standard conjugations (see \cite[Prop.~1]{AIM2011}).
Moreover, it holds
\[
Q_A=\R\cup\{x\in A\setminus\R: t(x)\in\R,n(x)\in\R,4n(x)>t(x)^2\}.
\]

Each element $x$ of $Q_A$ can be written as $x=\RE(x)+\IM(x)$, with $\RE(x)=\frac{x+x^c}2$, $\IM(x)=\frac{x-x^c}2=\beta J$, where $\beta=\sqrt{n(\IM(x))}\geq0$ and $J\in\cS_A$. The choice of $\beta\geq0$ and $J\in\cS_A$ is unique if $x\not\in\R$.

We refer to \cite[\S2]{AIM2011} and \cite[\S1]{AlgebraSliceFunctions} for more details and examples about real alternative $^*$-algebras and their quadratic cones.


\subsection{The one variable slice function theory}

The \emph{slice functions} on $A$ are the functions which are compatible with the slice character of the quadratic cone.
More precisely, let $D$ be a subset of $\C$ be a set that is invariant w.r.t.\ complex conjugation. 
Let $A\otimes_{\R}\C$ be the complexified algebra, whose elements $w$ are of the form $w=a+\ui b$ with $a,b\in A$ and $\ui^2=-1$. In $A\otimes_{\R}\C$ we consider the complex conjugation mapping $w=a+\ui b$ to $\overline w=a-\ui b$  for all $a,b\in A$.
If a function $F: D \to A\otimes_{\R}\C$ satisfies  $F(\overline z)=\overline{F(z)}$ for every $z\in D$, then $F$  is called a \emph{stem function} on $D$. For every $J\in\cS_A$, we define the real $^*$-algebra isomorphism $\phi_J:\C\to\C_J$ by setting
\begin{equation}\label{eq:phiJ}
\text{$\phi_J(\alpha+i\beta):=\alpha+J\beta\;$ for all $\alpha,\beta\in\R$.}
\end{equation}
Let $\OO_D$ be the \emph{circular} subset of the quadratic cone defined by 
\[
\OO_D=\bigcup_{J\in\cS_A}\phi_J(D)=\{\alpha+J\beta\in
A : \alpha,\beta\in\R, \alpha+\ui\beta\in D,J\in\cS_A\}.
\]
The stem function $F=F_1+\ui F_2:D \to A\otimes_{\R}\C$  induces the \emph{(left) slice function} $f=\I(F):\OO_D \to A$ in the following way: if $x=\alpha+J\beta =\phi_J(z)\in \OO_D\cap \C_J$, then  
\[ f(x)=F_1(z)+JF_2(z),\]
where $z=\alpha+\ui\beta$. 

Suppose that $D$ is open. Left multiplication by $i$ defines a complex structure on $A\otimes_{\R}\C$. The slice function $f=\I(F):\OO_D \to A$ is called \emph{(left) slice regular} if $F$ is holomorphic. 
For example, polynomial functions $f(x)=\sum_{j=0}^d x^ja_j$ with right coefficients belonging to $A$ are slice regular on the quadratic cone. 

To any slice function $f=\I(F):\OO_D \to A$, one can associate the function $\vs f:\OO_D \to A$, called \emph{spherical value} of $f$, and the function $f'_s:\OO_D \setminus \R \to A$, called  \emph{spherical derivative} of $f$, defined as
\[
\vs f(x):=\frac{1}{2}(f(x)+f(x^c))
\quad \text{and} \quad
f'_s(x):=\frac{1}{2}\IM(x)^{-1}(f(x)-f(x^c)).
\] 
If $x=\alpha+\beta J\in\OO_D$ and $z=\alpha+\ui\beta\in D$, then $\vs f(x)=F_1(z)$ and $f'_s(x)=\beta^{-1} F_2(z)$. Therefore $\vs f$ and $f'_s$ are slice functions, constant on every set $\cS_x:=\alpha+\beta\,\cS_A$. They are slice-regular only if $f$ is locally constant.
Moreover, the formula
\[
f(x)=\vs f(x)+\IM(x)f'_s(x)
\]
holds for all $x\in\OO_D\setminus \R$. As we will see later, the concepts of spherical value and spherical derivative in one variable will have a central role to get a characterization of slice regularity in several variables in terms of separate one variable regularity.

We refer the reader to \cite[\S3,4]{AIM2011} for more properties of slice functions and slice regularity in one variable.

\subsection{Slice regular functions on $\hh^2$}
Before presenting the full theory in the general case of $n$ variables in a real alternative $^*$-algebra $A$, in this subsection we summarize the main results in the simpler case of two quaternionic variables. We refer to the following sections for full proofs.

\paragraph{Slice functions on $\hh^2$.}
Let $D$ be a non-empty subset of $\mathbb C^2$, invariant w.r.t.\ complex conjugation in each variable $z_1,z_2$. Let $\OO_D$ be the \emph{circular} open subset of $\HH^2$ associated to  $D\subset\C^2$, defined as
\[
\OO_D:=\{(x_1,x_2)\in \HH^2 : x_h=\alpha_h+J_h\beta_h, \text{ with
$J_h\in\cS_\HH$ for $h=1,2$}, (\alpha_1+i\beta_1,\alpha_2+i\beta_2)\in D\}.
\]
Let $\{e_\emptyset,e_1,e_2,e_{12}\}$ denote a fixed basis of $\R^4$. If $\pa(2)$ denotes the set of all subsets of $\{1,2\}$, we can write any element $x$ of the real vector space $\HH\otimes\R^4$ as $x=\sum_{K\in\pa(2)}e_Ka_K$, where each $a_K$ belongs to $\HH$, and $e_{\{1\}}=e_1$, $e_{\{2\}}=e_2$ and $e_{\{1,2\}}=e_{12}$.

\begin{definition}
 A function $F:D \rightarrow \HH \otimes \R^4$, with $F=e_\emptyset F_\emptyset+e_1F_1+e_2F_2+e_{12}F_{12}$ and $F_K:D\to\HH$ for each $K\in\pa(2)$, is called a \emph{stem function} if  the components $F_\emptyset, F_1, F_2, F_{12}$ are, respectively, even-even, odd-even, even-odd, odd-odd w.r.t.\ the pair $(\beta_1,\beta_2)$, where $z_1=\alpha_1+i\beta_1$ and $z_2=\alpha_2+i\beta_2$ with $\alpha_1,\alpha_2,\beta_1,\beta_2\in\R$. The \emph{(left) slice function} $f=\I(F):\OO_D \rightarrow \HH$ induced by $F$ is the function obtained by setting, for each $x=(x_1,x_2)=(\alpha_1+J_1\beta_1,\alpha_2+J_2\beta_2)$,
\[
\textstyle
f(x):=F_\emptyset(z_1,z_2)+J_1F_1(z_1,z_2)+J_2F_1(z_1,z_2)+J_1J_2F_{12}(z_1,z_2)
\]
where $(z_1,z_2)=(\alpha_1+i\beta_1,\alpha_2+i\beta_2)\in D$.
\end{definition}

\paragraph{Representation formula on $\HH^2$.}
The values of a slice function can be recovered by its values on a four-dimensional slice of $\OO_D$. Let $a^c$ denote the conjugate of a quaternion $a\in\HH$.

\begin{proposition}\label{pro:rep2}
Let $f:\OO_D\to\HH$ be a slice function and let $y=(y_1,y_2)=(\alpha_1+I_1\beta_1,\alpha_2+I_2\beta_2)\in\OO_D$. Then for every $x=(x_1,x_2)=(\alpha_1+J_1\beta_1,\alpha_2+J_2\beta_2)\in\OO_D$ it holds:
\begin{align*}
f(x)=&\,\frac14\big(
f(y_1,y_2)+f(y_1^c,y_2)+f(y_1,y_2^c)+f(y_1^c,y_2^c)+\\
&+J_1I_1\left( -f(y_1,y_2)+ f(y_1^c,y_2)- f(y_1,y_2^c)+ f(y_1^c,y_2^c)\right)+\\
&+J_2I_2\left(-f(y_1,y_2)- f(y_1^c,y_2)+ f(y_1,y_2^c)+ f(y_1^c,y_2^c)\right)+\\
&+J_1J_2I_2I_1\left(f(y_1,y_2)- f(y_1^c,y_2)- f(y_1,y_2^c)+ f(y_1^c,y_2^c)\right)
\big).
\end{align*}
\end{proposition}

\begin{corollary}[Identity principle]\label{cor:ip2}
Let $f,g:\OO_D \to A$ be slice functions and let $I_1,I_2 \in \sq_\HH$ such that $f=g$ on $\OO_D \cap(\C_{I_1}\times\C_{I_2})$. Then $f=g$ on the whole $\OO_D$. 
\end{corollary}


\paragraph{Smoothness.}
Suppose that $D$ is open in $\C^2$.

\begin{proposition} \label{thm:Cr2}
For every slice function $f=\I(F):\OO_D\to\HH$, it holds:
\begin{itemize}
\item[$(\mr{i})$] If $F$ is continuous on $D$, then
$f$ is continuous on $\OO_D$.
\item[$(\mr{ii})$] Let $k\in\N\setminus\{0\}$. If $F$
is of class $\mscr{C}^{4k+3}$ on $D$, then $f$ is of class
$\mscr{C}^{k}$ on $\OO_D$.
\item[$(\mr{iii})$] Let $k\in\{\infty,\omega\}$. If $F$
is of class $\mscr{C}^k$ on $D$, then $f$ is of class
$\mscr{C}^k$ on $\OO_D$.
\end{itemize}
\end{proposition}

\paragraph{Multiplicative structure on slice functions.}

Every product on $\HH\otimes\R^4$ induces a product on slice functions, and hence a structure of real algebra on the set of slice functions. We consider the product on $\HH\otimes\R^4$ constructed as follows. First, we equip $\R^4$ with the unique (commutative and associative) multiplicative structure which makes the real linear isomorphism $\R^4\to\C\otimes\C$, sending $e_\emptyset$ to the unity $1=1\otimes1$, $e_1$ to $i\otimes 1$, $e_2$ to $1\otimes i$ and $e_{12}$ to $i\otimes i$, a real algebra isomorphism. In other words, $e_\emptyset=1$ is the unity of $\R^4$, $e_1^2=e_2^2=-1$ and $e_1e_2=e_2e_1=e_{12}$. Then, we extend this product to $\HH\otimes\R^4$ by setting $(a\otimes v)\cdot(b\otimes w)=(ab)\otimes(vw)$ for all $a,b\in\HH$ and $v,w\in\R^4$. In this way, we can identify the real algebra $\HH\otimes\R^4$ with $\HH\otimes(\C\otimes\C)$. 

\begin{definition} \label{def:slice-product2}
Let $f,g:\OO_D \to \HH$ be slice functions with $f=\I(F)$ and $g=\I(G)$. We define the \emph{{(tensor) slice product} $f\cdot g:\OO_D\to \HH$} of $f$ and $g$ by $f\cdot g:=\I(FG)$, where $FG$ is the pointwise product defined by $(FG)(z)=F(z)G(z)$ in $\HH\otimes(\C\otimes\C)$ for all $z\in D$.
\end{definition}

For example, the slice product of the coordinate functions $x_1:\HH^2\to\HH$ and $x_2:\HH^2\to\HH$, with $x_h=\alpha_h+J_h\beta_h=\I(\alpha_h+e_h\beta_h)$ for $h=1,2$, is the slice function $x_1\cdot x_2:\HH^2\to\HH$ given by
\begin{align*}
x_1\cdot x_2&=\I((\alpha_1+e_1\beta_1)(\alpha_2+e_2\beta_2))=\I(\alpha_1\alpha_2+e_1\alpha_2\beta_1+e_2\alpha_1\beta_2+e_{12}\beta_1\beta_2)=\\
&=\alpha_1\alpha_2+J_1\alpha_2\beta_1+J_2\alpha_1\beta_2+J_1J_2\beta_1\beta_2.
\end{align*}
In this case, the slice product $x_1\cdot x_2$ coincides with the pointwise product $x_1x_2$. Moreover, $x_1\cdot x_2=x_2\cdot x_1$. In general, if $a,b\in\HH$, the slice product of $x_1a=\I(\alpha_1a+e_1(\beta_1a))$ and $x_2b=\I(\alpha_2b+e_2(\beta_2b))$ is the slice function $(x_1a)\cdot (x_2b)=x_1x_2ab$, while $(x_2b)\cdot(x_1a)=x_1x_2ba$. Note that the pointwise product $x_2x_1$ is not even a slice function, see Remark \ref{rem:order}.

The isomorphism from the real algebra of stem functions $F:D\to\HH\otimes(\C\otimes\C)$ with the pointwise product to the real algebra of slice functions $f:\OO_D\to\HH$ with the slice product can be expressed by the commutativity of the following diagrams for all $J_1,J_2\in \cS_\HH$:
\begin{center}
\begin{tikzpicture}

\node (A) at (0.05,0.04) [] {$\C^2\supset$}; 

\node (B) at (.8,0) [] {$D$}; 

\node (C) at (4,0) [] {$\HH\otimes(\C\otimes\C)$}; 

\node (A2) at (-0.06,-1.94) [] {$\HH^2\supset$}; 

\node (B2) at (.8,-2) [] {$\OO_D$}; 

\node (C2) at (4,-2) [] {$\HH$}; 

\draw [] (A)  (B);

\draw [->] (B) --node[above, ]{$F$} (C);

\draw [->] (B) --node[left, ]{$\phi_{J_1}\times\phi_{J_2}$} (B2);
\draw [-> ] (C) --node[right, ]{$\Phi_{J_1,J_2}$} (C2);

\draw [] (A2)  (B2);

\draw [->] (B2) --node[above, ]{$f$} (C2);

\node at (2.4,-1) [] {$\circlearrowleft$};
\end{tikzpicture}
\end{center}
Here $\phi_{J_1}\times\phi_{J_2}:D\to\OO_D$ denotes the product map $(\phi_{J_1}\times\phi_{J_2})(z_1,z_2):=(\phi_{J_1}(z_1),\phi_{J_2}(z_2))$, where $\phi_J$ is the map defined in \eqref{eq:phiJ}, and $\Phi_{J_1,J_2}:\HH\otimes(\C\otimes\C)\to\HH$ is the real linear map defined by $\Phi_{J_1,J_2}(a\otimes(z_1\otimes z_2)):=\phi_{J_1}(z_1)\phi_{J_2}(z_2)a$ for all $a\in\HH$ and $z_1,z_2\in\C$.

\paragraph{Slice regular functions.}
Let $\J_1$ and $\J_2$ be the commuting complex structures on $\R^4\simeq\C\otimes\C$ induced, respectively, by the standard structures of the two copies of $\C$. Explicitly, $\J_h(e_h)=-1$ for $h=1,2$ and $\J_1(e_2)=\J_2(e_1)=e_{12}$. We extend these structures to $\HH\otimes\R^4$ by setting $\J_h(a\otimes v)=a\otimes\J_h(v)$ for all $a\in\HH$ and $v\in\R^4$.

Throughout the remaining part of this section, we assume that $D$ is open in $\C^2$.

\begin{definition} \label{def:partial_h,dibar_h2}
Let $F:D \to \HH \otimes \R^4$ be a stem function of class $\mscr{C}^1$. For each $h=1,2$, we denote $\partial_h$ and $\dibar_h$ the \emph{Cauchy-Riemann operators} w.r.t.\ the complex structures $i$ on $D$ and  $\J_h$ on $\HH\otimes \R^{4}$, i.e.
\[
\partial_hF=\frac{1}{2}\left(\dd{F}{\alpha_h}-\J_h\left(\dd{F}{\beta_h}\right)\right)
\;\;\text{ and }\;\;
\quad \dibar_hF=\frac{1}{2}\left(\dd{F}{\alpha_h}+\J_h\left(\dd{F}{\beta_h}\right)\right),
\]
where $\alpha_h+i \beta_h:D\to\C$ is the $h^{\mr{th}}$-coordinate function of $D$. Let $f=\I(F):\OO_D\to\HH$ and let $h=1,2$. We define the \emph{slice partial derivatives} of $f$ as the following slice function on $\OO_D$:
\[
\dd{f}{x_h}:=\I(\partial_hF)
\;\;\text{ and }\;\;
\dd{f}{x_h^c}:=\I(\dibar_hF).\, \text{  \bs}
\]
\end{definition}

We denote $\mc{S}^1(\OO_D,\HH)$ the set of all slice functions induced by stem functions of class~$\mscr{C}^1$.

\begin{proposition}[Leibniz's rule]\label{prop:leibniz2}
For each slice functions $f,g\in\mc{S}^1(\OO_D,\HH)$ and $h=1,2$, it holds:
\begin{align*}
\frac{\partial}{\partial x_h}(f\cdot g)&=\frac{\partial f}{\partial x_h}\cdot g+f\cdot \frac{\partial g}{\partial x_h},\\
\frac{\partial}{\partial x_h^c}(f\cdot g)&=\frac{\partial f}{\partial x_h^c}\cdot g+f\cdot \frac{\partial g}{\partial x_h^c}.
\end{align*}
\end{proposition}

\begin{definition}
Let $F:D\to\HH\otimes\R^4$ be a stem function of class $\mscr{C}^1$ and let $f=\I(F):\OO_D \rightarrow \HH$ be the induced slice function. $F$ is called \emph{holomorphic} if $\dibar_1F=\dibar_2F=0$ on $D$. If $F$ is holomorphic, then we say that $f=\I(F)$ is a \emph{slice regular function}. \bs
\end{definition}

Thanks to Proposition \ref{thm:Cr2}, every slice regular function is real analytic on $\OO_D$.

\begin{example}
All polynomial functions $f:\HH^2\to\HH$ of the form $f(x)=\sum_{(\ell_1,\ell_2)\in L}x_1^{\ell_1} x_2^{\ell_2} a_{\ell_1,\ell_2}$, for some finite subset $L$ of $\N^2$ and $a_{\ell_1,\ell_2}\in\HH$, are slice regular. More generally, the sum of a convergent power series $\sum_{(\ell_1,\ell_2)\in\N^2}x_1^{\ell_1} x_2^{\ell_2} a_{\ell_1,\ell_2}$ is slice regular on a product of two open balls of $\HH$ centered at the origin.
\end{example}

For each $J\in\cS_\HH$ and for each slice function $f:\OO_D\to\HH$, we define $\OO_D(J):=\OO_D\cap(\C_J\times\C_J)$ and we denote  $f_J:\OO_D(J)\to\HH$ the restriction of $f$ to $\OO_D(J)$.

\begin{proposition} \label{prop:slice-regularity2}
Let $f\in\mc{S}^1(\OO_D,\HH)$. The following assertions are equivalent:
\begin{itemize}
  \item[$(\mr{i})$] $f$ is slice regular.
  \item[$(\mr{ii})$] $\displaystyle\frac{\partial f}{\partial x_h^c}=0$ on $\OO_D$ for $h=1,2$.
  \item[$(\mr{iii})$] There exists $J \in \cS_\HH$ such that $f_J:\OO_D(J)\to\HH$ is holomorphic w.r.t.\ the complex structures on $\OO_D(J)$ and on $\HH$ defined by the left multiplication by $J$; that is,
\begin{equation}\label{eq:holom-f_I2}
\dd{f_J}{\alpha_h}(z)+J\dd{f_J}{\beta_h}(z)=0\;\text{ for all $z\in\OO_D(J)$ and for all $h=1,2$,}
\end{equation}
where $z=(\alpha_1+J\beta_1,\alpha_2+J\beta_2)\in\C_J\times\C_J$.
  \item[$(\mr{iv})$] For each $J \in \cS_\HH$, $f_J$ is holomorphic in the sense of \eqref{eq:holom-f_I2}.
  \item[$(\mr{v})$] $f$ is slice regular w.r.t.\ $x_1$, and the spherical value and spherical derivative of $f$ w.r.t.\ $x_1$ are slice regular w.r.t.\ $x_2$. 
 \end{itemize}
\end{proposition}

As a simple illustration of condition $(\mr{e})$, consider the polynomial $f(x_1,x_2)=x_1x_2$. For every fixed $x_2\in\HH$, $g(x_1):=f(x_1,x_2)=x_1x_2$ is slice regular w.r.t.\ the variable $x_1$. Moreover, the spherical value $\vs g(x_1)=x_2\RE(x_1)$ and the spherical derivative $g'_s(x_1)=x_2$ of $g$ w.r.t.\ $x_1$ are slice regular w.r.t.\ the variable $x_2$ for every fixed $x_1\in\HH$. 

\begin{proposition}
The zero set of any nonconstant polynomial function  $f:\HH^2\to\HH$ is a nonempty real algebraic subset of $\R^8=\HH^2$, whose real dimension can assume precisely the three values $4$, $5$ and $6$.
\end{proposition}

Leibniz's rule (Proposition \ref{prop:leibniz2}) and Proposition \ref{prop:slice-regularity2} imply that the slice product preserves slice regularity.  
In particular, every slice product $f(x)=f_1(x_1)\cdot f_2(x_2)$, with $f_h(x_h)$ slice regular w.r.t.\ the variable $x_h$ for $h=1,2$,  is slice regular. 

Since the complex structures $\J_1$ and $\J_2$ commute, every pair of Cauchy-Riemann operators commute. In particular, for every slice regular function $f$, also the slice partial derivatives $\ds{f}{x_1}$ and $\ds{f}{x_2}$ are slice regular.

\paragraph{Cauchy integral formula for slice regular functions.}
Let $J\in\cS_\HH$ be fixed. Recall that $\phi_J:\C\to\C_J$ denotes the real $^*$-algebra isomorphism $\phi_J(\alpha+i\beta):=\alpha+J\beta$. Let $E'_1$ and $E'_2$ be bounded non-empty open subsets of $\C$ invariant under complex conjugation and with boundaries of class $\mscr{C}^1$. Define $E_h:=\phi_J(E_h')$ for $h=1,2$, and $E:=E_1\times E_2\subset\C_J\times\C_J$ with distinguished boundary $\partial^*E:=(\partial E_1)\times (\partial E_2)$. Let $\OO(E)=\OO_{E'_1\times E'_2}$ be the circular open subset of $\HH^2$ such that $\OO(E)\cap(\C_J\times\C_J)=E$.

\begin{definition}
We define the \emph{slice Cauchy kernel for $E$} as the function $C:\OO(E)\times\partial^*E\to\HH$ given by
\begin{align*}
C(x,y):=C_\HH(x_1,y_1)\cdot_x C_\HH(x_2,y_2),
\end{align*}
where each $C_\HH(x_h,y_h)=\Delta_{y_h}(x_h)^{-1}(y_h^c-x_h)=(x_h^2-x_ht(y_h)+n(y_h))^{-1}(y_h^c-x_h)$ is the usual slice Cauchy kernel in one quaternionic variable, 
and the slice product is performed w.r.t.\ $x=(x_1,x_2)\in\OO(E)$ for each $y=(y_1,y_2)\in\partial^*E$. \bs
\end{definition}

It is worth noting that $C(\cdot,y)$ is slice regular on $\OO(E)$ for each fixed $y\in\partial^*E$.

For $h=1,2$, let $\xi_h:T_h\to\partial E'_h$ be piecewise $\mscr{C}^1$ parametrizations of $\partial E'_h$ and let $T:=T_1\times T_2$. Given two continuous functions $p,q:\partial^*E\to\HH$, we define
\[
\int_{\partial^*E}p(y)\,dy\,q(y):=\int_Tp(\xi_1(t_1),\xi_2(t_2))\,\dot{\xi}_1(t_1)\,\dot{\xi}_2(t_2)\,q(\xi_1(t_1),\xi_2(t_2))\,dt_1dt_2,
\]
where $\dot{\xi}_h:T_h\to\HH$ denotes the a.e.\ defined derivatives of $\xi_h$.

\begin{theorem}
Let $f:\OO_D\to \HH$ be a slice regular function. Suppose that the closure of $\OO(E)$ in $\HH^2$ is contained in $\OO_D$. Then
\[
f(x)=(2\pi)^{-2}\int_{\partial^*E}C(x,y)J^{-2}\,dy\, f(y)\quad\text{ for all $x\in\OO(E)$,}
\]
and the slice Cauchy kernel $C$ can be expressed in terms of pointwise operations as follows:
\begin{align*}
 C(x,y)=&\,\Delta_{y_1}(x_1)^{-1}x_1\Delta_{y_2}(x_2)^{-1}x_2-\Delta_{y_1}(x_1)^{-1}\Delta_{y_2}(x_2)^{-1}x_2y_1^c+\\
 &-\Delta_{y_1}(x_1)^{-1}x_1\Delta_{y_2}(x_2)^{-1}y_2^c+\Delta_{y_1}(x_1)^{-1}\Delta_{y_2}(x_2)^{-1}y_1^cy_2^c=\\
=&\,\Delta_{y_1}(x_1)^{-1}\big(x_1\Delta_{y_2}(x_2)^{-1}x_2-\Delta_{y_2}(x_2)^{-1}x_2y_1^c-x_1\Delta_{y_2}(x_2)^{-1}y_2^c+\Delta_{y_2}(x_2)^{-1}y_1^cy_2^c\big)
\end{align*}
for all $(x,y)\in\OO(E)\times\partial^*E$.
\end{theorem}


\section{Slice functions}\label{sec:Slice functions}


\subsection{Basic definitions}
% ...

If $S$ is a set, then we denote $|S|$ the cardinality of $S$ and $\mc{P}(S)$ the set of all subsets of~$S$. Let $\N$ be the set of all non-negative integers and let $\N^*:=\N\setminus\{0\}$. For simplicity, given any $n\in\N^*$, we use the symbol $\pa(n)$ instead of $\mc{P}(\{1,\ldots,n\})$. Given $z=(z_1,\ldots,z_n) \in \C^n$ and $h \in \{1,\ldots,n\}$, we define
\[
\overline{z}^h:=(z_1,\ldots,z_{h-1},\overline{z_h},z_{h+1},\ldots,z_n),
\]
where $\overline{z_h}$ denotes the usual conjugation of the complex number $z_h$. Given a subset $D$ of $\C^n$, we say that $D$ is invariant under complex conjugations if $\overline{z}^h \in D$ for all $z \in D$ and for all $h \in \{1,\ldots,n\}$.

\begin{assumption}\label{assumption1}
Throughout the paper, $A$ is a real alternative algebra of finite dimension with unity $1\neq0$, we equip with its natural Euclidean topology and structure of real analytic manifold, as a finite dimensional real vector space.

We denote $n$ a positive integer and $\{e_K\}_{K \in \mc{P}(n)}$ a fixed basis of the real vector space $\R^{2^n}$. We identify $\R$ with a real vector subspace of $\R^{2^n}$ via the real linear embedding sending $1\in\R$ into $e_\emptyset\in\R^{2^n}$, and we write $e_{\emptyset}=1$. For simplicity, we set $e_k:=e_{\{k\}}$ for all $k\in\{1,\ldots,n\}$.

We assume that $\cS_A\neq\emptyset$.

We also assume that $D$ is a non-empty subset of $\C^n$ invariant under complex conjugations.
\end{assumption}

Consider the (real) tensor product $A \otimes \R^{2^n}$. Each element $x$ of $A \otimes \R^{2^n}$ can be uniquely written as $x=\sum_{K\in\pa(n)}e_Ka_K$ with $a_K\in A$. In particular, we can identify each element $a$ of $A$ with the element $e_\emptyset a=1a$ of $A\otimes\R^{2^n}$, and we write $a=1a$. As a consequence, $A$ turns out to be a real vector subspace of $A\otimes\R^{2^n}$. Note that, given any function $F:D \to A \otimes \R^{2^n}$, there exist, and are unique, functions $F_K:D\to A$ such that $F=\sum_{K \in \pa(n)}e_KF_K$. We say that $F_K$ is the \emph{$K$-component of $F$}.

\begin{definition}\label{def:stem-function}
We say that a function $F:D \to A \otimes \R^{2^n}$ with $F=\sum_{K \in \pa(n)}e_KF_K$ is a \emph{stem function} if $F_K(\overline{z}^h)=(-1)^{|K \cap \{h\}|}F_K(z)$ or, equivalently,
\begin{equation} \label{eq:stem}
F_K(\overline{z}^h)=
\left\{
 \begin{array}{ll}
F_K(z)  & \text{ if  $\,h \not\in K$}
\vspace{.3em}
\\
-F_K(z)  & \text{ if $\,h \in K$}
 \end{array}
\right.
\end{equation}
for all  $z \in D$, $K \in \pa(n)$ and $h \in \{1,\ldots,n\}$. We denote $\mr{Stem}(D,A\otimes\R^{2^n})$ the set of all stem functions from $D$ to $A\otimes\R^{2^n}$. \bs 
\end{definition}

Let $z=(z_1,\ldots,z_n) \in \C^n$ and let $H=\{h_1,\ldots,h_p\} \in \pa(n)\setminus\{\emptyset\}$ with $h_1<\ldots <h_p$. Define $\overline{z}^{H} \in \C^n$ by setting
\[
\overline{z}^{H}:=(z_1,\ldots,z_{h_1-1},\overline{z_{h_1}},z_{h_1+1},\ldots,z_{h_p-1},\overline{z_{h_p}},z_{h_p+1},\ldots,z_n).
\]
If $H=\emptyset$, then we set $\overline{z}^H:=z$. Note that $\overline{z}^{\{h\}}=\overline{z}^h$ for all $h \in \{1,\ldots,n\}$ and $\overline{z}^H \in D$ for all $z \in D$ and $H \in \pa(n)$. Moreover, if $F:D \to A \otimes\R^{2^n}$ is a function with $F=\sum_{K \in \pa(n)}e_KF_K$, then $F$ is a stem function if and only if
\begin{equation} \label{eq:stem2}
F_K(\bar z^{H})=(-1)^{|K\cap H|}F_K(z) \; \text{ for all $z\in D$ and $K,H \in \pa(n)$}.
\end{equation}

\begin{definition} \label{def:circularization}
Let $W$ be a subset of $\C^n$. We call \emph{circularization of $W$ (in $A^n$)} the set $\OO_W$ of all points $(\alpha_1+J_1\beta_1,\ldots,\alpha_n+J_n\beta_n)\in(Q_A)^n$ with $\alpha_1,\beta_1,\ldots,\alpha_n,\beta_n \in \R$ and $J_1,\ldots,J_n \in \cS_A$ such that $(\alpha_1+\ui \beta_1,\ldots,\alpha_n+\ui \beta_n) \in W$. A subset $\Theta$ of $(Q_A)^n$ is said to be \emph{circular (in $A^n$)} if $\Theta=\OO_W$ for some subset $W$ of $\C^n$.

Given any $x\in(Q_A)^n$, we denote $\cS_x$ the smallest circular subset of $(Q_A)^n$ containing $x$. \bs
\end{definition}

Note that, if $x=(x_1,\ldots,x_n)\in(Q_A)^n$, then $\cS_x=\cS_{x_1}\times\cdots\times\cS_{x_n}$, where $\cS_{x_h}=\alpha_h+\beta_h\cS_A$ if $x_h=\alpha_h+J_h\beta_h$ for some $\alpha_h,\beta_h\in\R$ and $J_h\in\cS_A$. Given another point $y=(y_1,\ldots,y_n)$ in $(Q_A)^n$, we have that $\cS_x=\cS_y$ if and only if $t(x_h)=t(y_h)$ and $n(x_h)=n(y_h)$ for all $h\in\{1,\ldots,n\}$.

Let us introduce a notation, which is very useful especially in the non-associative case. 

\begin{definition} \label{def:[]}
Given any $m\in\N^*$ and any sequence $u=(u_1,\ldots,u_m)$ of elements of $A$, we define the \emph{ordered product $[u]=[u_1,\ldots,u_m]$ of $u$} by setting $[u]=[u_1,\ldots,u_m]:=u_1$ if $m=1$ and 
\[
[u]=[u_1,\ldots,u_m]:=u_1(u_2(\cdots(u_{m-1}u_m)\ldots))
\]
if $m\geq2$.
Given any $v\in A$, we write $[u,v]$ to indicate $[u_1,\ldots,u_m,v]$.

Moreover, we set $[\emptyset]:=1$ and $[\emptyset,v]:=v$. 

Let $H\in\pa(m)$. If $H=\emptyset$, then we define $u_H:=\emptyset$; hence $[u_H]=1$ and $[u_H,v]=v$. If $H\neq\emptyset$, then we write $H=\{h_1,\ldots,h_p\}$ with $h_1<\ldots <h_p$, and we define $u_H:=(u_{h_1},\ldots,u_{h_p})$; as a consequence, we have:
\[
[u_H]=u_{h_1}(u_{h_2}(\cdots(u_{h_{p-1}}u_{h_p})\ldots))
\]
and
\[
[u_H,v]=u_{h_1}(u_{h_2}(\cdots(u_{h_{p-1}}(u_{h_p}v))\ldots)).
\]
We use also the symbols $(u_h)_{h=1}^m$ to denote $u$, and $(u_h)_{h\in H}$ to denote $u_H$. 

Suppose now that, for each $h\in\{1,\ldots,m\}$, $u_h$ is invertible in $A$, and denote $u_h^{-1}$ its inverse $(u_h)^{-1}$ in $A$. In this case, if $H=\emptyset$, then we define $u_H^{-1}:=\emptyset$; hence $[u_H^{-1}]=1$ and $[u_H^{-1},v]=v$. If $H\neq\emptyset$ and $H=\{h_1,\ldots,h_p\}$ with $h_1<\ldots <h_p$, then we define $u_H^{-1}:=(u_{h_p}^{-1},u_{h_{p-1}}^{-1},\ldots,u_{h_1}^{-1})$; as a consequence, we have:
\[
[u_H^{-1}]=u_{h_p}^{-1}(u_{h_{p-1}}^{-1}(\cdots(u_{h_2}^{-1}u_{h_1}^{-1})\ldots))
\]
and
\[
[u_H^{-1},v]=u_{h_p}^{-1}(u_{h_{p-1}}^{-1}(\cdots(u_{h_2}^{-1}(u_{h_1}^{-1}v))\ldots)).
\]
We use also the symbols $((u_h)_{h=1}^m)^{-1}$ to denote $(u_{m+1-h}^{-1})_{h=1}^m$, and $((u_h)_{h\in H})^{-1}$ to denote $u_H^{-1}$. \bs
\end{definition}

Given any $v,w\in A$, any $H\in\pa(m)$, and any $u=(u_1,\ldots,u_m)\in A^m$ with $m\geq 1$ and $u_1,\ldots,u_m$ invertible in $A$, it is immediate to verify that
\begin{equation}\label{eq:uHv=w}
[u_H,v]=w\, \text{ if and only if }\,v=[u_H^{-1},w].
\end{equation}
The elements $[u_H]$ and $[u_H^{-1}]$ are invertible in $A$, see for instance Lemma 1.5(2) of \cite{AlgebraSliceFunctions}. However, in general, if $A$ is not associative, then $[u_H^{-1}]\neq[u_H]^{-1}$. On the contrary, if $A$ is associative, then $[u_H^{-1}]=[u_H]^{-1}$ and
$[u_H]=u_{h_1}u_{h_2}\cdots u_{h_p}$.

We are in position to introduce the notion of slice function in several variables.

\begin{definition} \label{def:slice-function}
Given a function $f:\OO_D\to A$, we say that $f$ is a \emph{(left) slice function} if there exists a stem function $F:D \to A \otimes\R^{2^n}$ with $F=\sum_{K \in \pa(n)}e_KF_K$ such that
\[
\textstyle
f(x):=\sum_{K \in \pa(n)}[J_K,F_K(z)]
\]
for all $x=(\alpha_1+J_1\beta_1,\ldots,\alpha_n+J_n\beta_n) \in \OO_D$, where $\alpha_1,\beta_1,\ldots,\alpha_n,\beta_n \in \R$, $z=(\alpha_1+i\beta_1,\ldots,\alpha_n+i\beta_n)\in D$ and $J=(J_1,\ldots,J_n) \in(\cS_A)^n$; hence $J_K=(J_{k_1},\ldots,J_{k_p})$ if $K=\{k_1,\ldots,k_p\}\in\pa(n)\setminus\{\emptyset\}$ with $k_1<\ldots<k_p$, and $J_\emptyset=\emptyset$.

If this is the case, we say that $f$ is \emph{induced by $F$}, and we write $f=\I(F)$. We denote $\mc{S}(\OO_D,A)$ the set of all slice functions from $\OO_D$ to $A$, and $\I:\mr{Stem}(D,A\otimes\R^{2^n})\to\mc{S}(\OO_D,A)$ the map sending each stem function $F$ into the corresponding slice function $\I(F)$. \bs
\end{definition}

The preceding definition is well-posed. Let $x=(x_1,\ldots,x_n)$ be a point of $\OO_D$. For each $h\in\{1,\ldots,n\}$, there exist $\alpha_h,\beta_h\in\R$ with $\beta_h\geq0$ and $J_h\in\cS_A$ such that $x_h=\alpha_h+J_h\beta_h$. If $x_h\in\R$, i.e., $\beta_h=0$, then $\alpha_h$ is uniquely determined by $x$; on the contrary, $J_h$ can be chosen arbitrarily in $\cS_A$. If $x_h\not\in\R$, i.e. $\beta_h>0$, then $\alpha_h$, $\beta_h$ and $J_h$ are uniquely determined by $x_h$, and $x_h$ has the following two representations:
\[
\alpha_h+J_h\beta_h=x_h=\alpha_h+(-J_h)(-\beta_h).
\]
Set $z:=(\alpha_1+i\beta_1,\ldots,\alpha_n+i\beta_n)\in D$. Let $L$ be the set of all $h\in\{1,\ldots,n\}$ such that $x_h\not\in\R$. Thanks to \eqref{eq:stem}, we know that $F_K(z)=0$ if $K\not\subset L$. For each $H\in\pa(n)$ with $\emptyset\neq H\subset L$, it is possible to write $x$ as follows:
\[
x=(\alpha_1+J_1\beta_1,\ldots,\alpha_n+J_n\beta_n)
\]
and $x=x'_H$, where
\[
x'_H=(\alpha_1+(\epsilon_1 J_1)(\epsilon_1\beta_1),\ldots,\alpha_n+(\epsilon_n J_n)(\epsilon_n\beta_n))
\]
with $\epsilon_h=-1$ if $h\in H$, $\epsilon_h=1$ if $h\in L\setminus H$, and $J_h$ can be chosen arbitrarily in $\cS_A$ if $h\not\in L$. By \eqref{eq:stem2}, we have
\begin{align*}
f(x)&\textstyle=\sum_{K \in \pa(n)}[J_K,F_K(z)]=\sum_{K \in \pa(n),K\subset L}[J_K,F_K(z)]=\\
&\textstyle=\sum_{K \in \pa(n),K\subset L}[J_K,(-1)^{|K\cap H|}F_K(\overline{z}^H)]=\\
&\textstyle=\sum_{K \in \pa(n)}[((-1)^{|H\cap\{k\}|}J_k)_{k\in K},F_K(\overline{z}^H)]=f(x'_H).
\end{align*}
It follows that Definition \ref{def:slice-function} is well-posed, as claimed. In Proposition \ref{prop:representation} below, we will show that a slice function is induced by a unique stem function.

The real algebra $A$ we are working with is assumed to be alternative. In particular, it is power-associative. Consequently, if $a\in A$ and $m\in\N^*$ then the power $a^m$ is a well-defined element of $A$, independently from the system of parentheses we use to compute it. For convention, we set $a^0:=1$ for all $a\in A$.

\begin{remark}\label{rem:vector}
The usual pointwise defined operations of addition and multiplication by real scalars define structures of real vector spaces on the sets $\mr{Stem}(D,A\otimes\R^{2^n})$ and $\mc{S}(\OO_D,A)$, which makes the map $\I:\mr{Stem}(D,A\otimes\R^{2^n})\to\mc{S}(\OO_D,A)$ a real linear map. Given any $F=\sum_{K\in\pa(n)}e_KF_K,G=\sum_{K\in\pa(n)}e_KG_K\in\mr{Stem}(D,A\otimes\R^{2^n})$, $f,g\in\mc{S}(\OO_D,A)$ and $r\in\R$, we set $(F+G)(z):=\sum_{K\in\pa(n)}e_K(F_K(z)+G_K(z))$, $(Fr)(z):=\sum_{K\in\pa(n)}e_K(F_K(z)r)$, $(f+g)(x):=f(x)+g(x)$ and $(fr)(x):=f(x)r$, where $F_K(z)+G_K(z)$, $F_K(z)r$, $f(x)+g(x)$ and $f(x)r$ are additions and scalar multiplications in $A$. Evidently, $\I(F+G)=\I(F)+\I(G)$ and $\I(Fr)=\I(F)r$. Actually, the map $\I$ is an isomorphism of real vector spaces, see Corollary \ref{cor:I} below. \bs
\end{remark}

\begin{definition}\label{def:polynomials}
Given $x=(x_1,\ldots,x_n)\in A^n$, $\ell=(\ell_1,\ldots,\ell_n)\in\N^n$ and $a\in A$, we denote $x^\ell a$ the element $[(x_h^{\ell_h})_{h=1}^n,a]$ of $A$. A function $P:(Q_A)^n\to A$ is \emph{monomial} if there exist $\ell\in\N^n$ and $a\in A$ such that $P(x)=x^\ell a$ for all $x\in(Q_A)^n$. The function $P:(Q_A)^n\to A$ is \emph{polynomial} if it is the finite sum of monomial functions.

The restriction of a monomial (respectively polynomial) function on $\OO_D$ is said to be \emph{monomial} (respectively \emph{polynomial}) \emph{on $\OO_D$.} \bs
\end{definition}

We conclude the present section with an important result, which asserts that the class of slice functions includes the one of polynomial functions. First, we need a definition.

\begin{definition}\label{def:pkqk}
 For each $k\in\N$, we denote $p_k$ and $q_k$ the real polynomials in $\R[X,Y]$ such that $(\alpha+i\beta)^k=p_k(\alpha,\beta)+iq_k(\alpha,\beta)$ for all $\alpha,\beta\in\R$. \bs
\end{definition}

\begin{proposition}\label{prop:polynomials}
Each polynomial function on $\OO_D$ is a slice function. More precisely, given any $\ell=(\ell_1,\ldots,\ell_n)\in\N^n$ and $a\in A$, the function $F^{(\ell)}:D\to A\otimes\R^{2^n}$, defined by
\begin{equation}\label{eq:F-ell}
\textstyle
F^{(\ell)}(z)=\sum_{K\in\pa(n)}e_K\big(\big(\prod_{h\in\{1,\ldots,n\}\setminus K}p_{\ell_h}(\alpha_h,\beta_h)\big)\big(\prod_{h\in K}q_{\ell_h}(\alpha_h,\beta_h)\big)a\big),
\end{equation}
for all $z=(\alpha_1+i\beta_1,\ldots,\alpha_n+i\beta_n)\in D$, is a stem function inducing the monomial function $x^\ell a$ on $\OO_D$.
\end{proposition}
\begin{proof}
Let $\ell=(\ell_1,\ldots,\ell_n)\in\N^n$. It suffices to show that the function $F^{(\ell)}:D\to A$, defined in \eqref{eq:F-ell}, is a stem function and $\I(F^{(\ell)})(x)=x^\ell a_\ell$. Let $z=(z_1,\ldots,z_n)=(\alpha_1+i\beta_1,\ldots,\alpha_n+i\beta_n)\in D$, let $h\in\{1,\ldots,n\}$ and, given any $K\in\pa(n)$, let $F_K^{(\ell)}:D\to A$ be the function
\[\textstyle
F_K^{(\ell)}(z):=\big(\prod_{h\in\{1,\ldots,n\}\setminus K}p_{\ell_h}(\alpha_h,\beta_h)\big)\big(\prod_{h\in K}q_{\ell_h}(\alpha_h,\beta_h)\big)a.
\]
Note that, for each $z_h=\alpha_h+i\beta_h\in\C$, it holds:
\[
p_{\ell_h}(\alpha_h,\beta_h)-iq_{\ell_h}(\alpha_h,\beta_h)=\overline{z_h^{\ell_h}}=\overline{z_h}^{\ell_h}=p_{\ell_h}(\alpha_h,-\beta_h)+iq_{\ell_h}(\alpha_h,-\beta_h)
\]
and hence $p_{\ell_h}(\alpha_h,-\beta_h)=p_{\ell_h}(\alpha_h,\beta_h)$ and $q_{\ell_h}(\alpha_h,-\beta_h)=-q_{\ell_h}(\alpha_h,\beta_h)$. Consequently, we have $F_K^{(\ell)}(\overline{z}^h)=-F_K^{(\ell)}(z)$ if $h\in K$ and $F_K^{(\ell)}(\overline{z}^h)=F_K^{(\ell)}(z)$ if $h\in\{1,\ldots,n\}\setminus K$. This proves that $F^{(\ell)}=\sum_{K\in\pa(n)}e_KF^{(\ell)}_K$ is a stem function. Moreover, if $x=(\alpha_1+J_1\beta_1,\ldots,\alpha_n+J_n\beta_n) \in \OO_D$ for some $J_1,\ldots,J_n \in \cS_A$, then
\begin{align*}
\I(F^{(\ell)})(x)&\textstyle=\sum_{K\in\pa(n)}[(J_h)_{h\in K},F_K^{(\ell)}(z)]=\\
&\textstyle=\sum_{K\in\pa(n)}\big(\prod_{h\in\{1,\ldots,n\}\setminus K}p_{\ell_h}(\alpha_h,\beta_h)\big)[(J_hq_{\ell_h}(\alpha_h,\beta_h))_{h\in K},a]=\\
&\textstyle=[(p_{\ell_h}(\alpha_h,\beta_h)+J_hq_{\ell_h}(\alpha_h,\beta_h))_{h=1}^n,a]=[(x_h^{\ell_h})_{h=1}^n,a]=x^\ell a.
\end{align*}
The proof is complete.
\end{proof}

\begin{definition}
We say that a stem function $F:\C^n\to A\otimes\R^{2^n}$ is \emph{monomial}, or \emph{polynomial}, if $\I(F):(Q_A)^n\to A$ is. The restriction of a monomial (respectively polynomial) stem function on $D$ is said to be \emph{monomial} (respectively \emph{polynomial}) \emph{on $D$.} \bs
\end{definition}
Thanks to Proposition \ref{prop:representation} below, a stem function $F:D\to A\otimes\R^{2^n}$ is monomial on $D$ if and only if it has form~\eqref{eq:F-ell}.

\subsection{Representation formulas}

We need an elementary, but useful, combinatorial lemma.

\begin{lemma} \label{lem:combinatorial}
For each $H,L \in \pa(n)$, it holds:
\[
\textstyle
\sum_{K \in \pa(n)}(-1)^{|H \cap K|+|K \cap L|}=2^n\delta_{H,L}\,,
\]
where $\delta_{H,L}=1$ if $H=L$ and $\delta_{H,L}=0$ otherwise.
\end{lemma}
\begin{proof} Since $|H \cap K|+|K \cap L|=|K \cap (H \setminus L)|+|K \cap (L \setminus H)|+2|K \cap H \cap L|$ for all $K \in \pa(n)$, it suffices to prove that the sum $s(H,L):=\sum_{K \in \pa(n)}(-1)^{|K \cap (H \setminus L)|+|K \cap (L \setminus H)|}$ is equal to $2^n$ if $H=L$ and is null otherwise. This is evident in the case in which $L=H$. Let $L \neq H$. If $H \subset L$, then we have:
\begin{eqnarray*}
\textstyle
s(H,L)
\!\!\!\! &=& \!\!\!\!
\textstyle
\sum_{K \in \pa(n)}(-1)^{|K \cap (L \setminus H)|}=\sum_{S_1 \in \pa(L \setminus H), \, S_2 \in \pa(\compl (L \setminus H))}(-1)^{|S_1|}=\\
&=& \!\!\!\!
\textstyle
2^{n-|L \setminus H|}\sum_{S_1 \in \pa(L \setminus H)}(-1)^{|S_1|}=2^{n-|L \setminus H|}\sum_{h=0}^{|L \setminus H|}{|L \setminus H| \choose h}(-1)^h=\\
&=& \!\!\!\!
\textstyle
2^{n-|L \setminus H|}(1+(-1))^{|L \setminus H|}=0,
\end{eqnarray*}
where $\compl (L \setminus H)$ is the complement of $L \setminus H$ in $\{1,\ldots,n\}$. Similarly, one proves that $s(H,L)=0$ if $L \subset H$. Finally, suppose $H \not\subset L$ and $L \not\subset H$. We have:
\begin{eqnarray*}
\textstyle
s(H,L)
\!\!\!\! &=& \!\!\!\!
\textstyle
\sum_{K \in \pa(n)}(-1)^{|K \cap (L \setminus H)|+|K \cap (H \setminus L)|}=\\
&=& \!\!\!\!
\textstyle
\sum_{S_1 \in \pa(L \setminus H), \, S_2 \in \pa(H \setminus L), \, S_3\in\pa(\compl (L \dsim H))}(-1)^{|S_1|}(-1)^{|S_2|}=\\
&=& \!\!\!\!
\textstyle
2^{n-|L \dsim H|}\big(\sum_{S_1 \in \pa(L \setminus H)}(-1)^{|S_1|}\big)\big(\sum_{S_2 \in \pa(H \setminus L)}(-1)^{|S_2|}\big)=\\
&=& \!\!\!\!
\textstyle
2^{n-|L \dsim H|}(1+(-1))^{|L \setminus H|}(1+(-1))^{|H \setminus L|}=0,
\end{eqnarray*}
where $L \dsim H$ is the usual symmetric difference between $L$ and $H$.
\end{proof}

Let $x=(x_1,\ldots,x_n)\in A^n$ and let $H \in \pa(n)$. If $H=\emptyset$, then we set $x^{c,H}:=x$. Suppose that $H\neq\emptyset$ and write $H=\{h_1,\ldots,h_p\}$ with $h_1<\ldots<h_p$. We denote $x^{\, c,H}$ the element of $A^n$ defined as follows:
\[
x^{\, c,H}:=(x_1,\ldots,x_{h_1-1},x_{h_1}^{\, c},x_{h_1+1},\ldots,x_{h_p-1},x_{h_p}^{\, c},x_{h_p+1},\ldots,x_n).
\]

We have the following representation formulas.

\begin{proposition}[Representation formula] \label{prop:representation}
Let $f:\OO_D \to A$ be a slice function and let $y \in \OO_D$. Write $y$ as follows:
\[
y=(\alpha_1+I_1\beta_1,\ldots,\alpha_n+I_n\beta_n),
\]
where $\alpha_h,\beta_h \in \R$ and $I_h \in \sq_A$ for all $h \in \{1,\ldots,n\}$. Then it holds:
\begin{equation} \label{eq:f}
\textstyle
f(x)=2^{-n}\sum_{K,H \in \pa(n)}(-1)^{|K \cap H|}[J_K,[I_K^{-1}, f(y^{\, c,H})]],
\end{equation}
where $x=(\alpha_1+J_1\beta_1,\ldots,\alpha_n+J_n\beta_n)$ for some $J=(J_1,\ldots,J_n)\in(\cS_A)^n$, and $I:=(I_1,\ldots,I_n)$.

Furthermore, if $F=\sum_{K \in \pa(n)}e_KF_K$ is a stem function inducing $f$, then we have:
\begin{align}\label{eq:components-F}
F_K(z)&\textstyle=2^{-n}[I_K^{-1},\sum_{H \in \pa(n)}(-1)^{|K \cap H|}f(y^{\, c,H})]=\nonumber\\
&\textstyle=2^{-n}\sum_{H \in \pa(n)}(-1)^{|K \cap H|}[I_K^{-1},f(y^{\, c,H})]
\end{align}
for all $K \in \pa(n)$, where $z=(\alpha_1+\ui \beta_1,\ldots,\alpha_n+\ui \beta_n) \in D$. In particular, each slice function $f$ is induced by a unique stem function $F$. 
\end{proposition}
\begin{proof}
Let $x=(\alpha_1+J_1\beta_1,\ldots,\alpha_n+J_n\beta_n)$. Define $z:=(\alpha_1+\ui\beta_1,\ldots,\alpha_n+\ui\beta_n) \in D$. Let $K \in \pa(n)$. Thanks to (\ref{eq:stem2}) and to Lemma \ref{lem:combinatorial}, we obtain:
\begin{eqnarray*}
\textstyle
\sum_{H \in \pa(n)}(-1)^{|K \cap H|}f(y^{\, c,H})
\!\!\!\! &=& \!\!\!\!
\textstyle
\sum_{H \in \pa(n)}(-1)^{|K \cap H|}\big(\sum_{L \in \pa(n)}[I_L,F_L(\overline{z}^H)]\big)=\\
&=& \!\!\!\!
\textstyle
\sum_{H \in \pa(n)}(-1)^{|K \cap H|}\big(\sum_{L \in \pa(n)}(-1)^{|H \cap L|}[I_L,F_L(z)]\big)=\\
&=& \!\!\!\!
\textstyle
\sum_{L \in \pa(n)}[I_L,F_L(z)]\big(\sum_{H \in \pa(n)}(-1)^{|K \cap H|+|H \cap L|}\big)=\\
&=& \!\!\!\!
\textstyle 2^n[I_K,F_K(z)].
\end{eqnarray*}
Bearing in mind \eqref{eq:uHv=w}, we deduce \eqref{eq:components-F}. Consequently,
\begin{eqnarray*}
f(x)
\!\!\!\! &=& \!\!\!\!
\textstyle
\sum_{K \in \pa(n)}[J_K,2^{-n} \sum_{H \in \pa(n)}(-1)^{|K \cap H|}[I_K^{-1},f(y^{\, c,H})]]=\\
&=& \!\!\!\!
\textstyle
2^{-n}\sum_{K,H \in \pa(n)}(-1)^{|K \cap H|}[J_K,[I_K^{-1},f(y^{\, c,H})]],
\end{eqnarray*}
as desired.
\end{proof}

As an immediate corollary, we obtain:

\begin{corollary}[Identity principle]\label{cor:ip}
Let $f,g:\OO_D \to A$ be slice functions and let $I_1,\ldots,I_n \in \sq_A$ such that $f=g$ on $\OO_D \cap(\C_{I_1}\times\ldots\times\C_{I_n})$. Then $f=g$ on the whole $\OO_D$. 
\end{corollary}

\begin{remark}\label{rem:order}
Let $n\geq2$ and let $f:\hh^n\to\hh$ be the function $f(x_1,\ldots,x_n):=x_2x_1$, i.e. the pointwise product between the coordinate functions $x_2$ and $x_1$. The function $f$ is not slice. Otherwise, being $x_2x_1=x_1x_2$ on $(\C_i)^n$, Proposition \ref{prop:polynomials} and Corollary \ref{cor:ip} would imply that $x_2x_1=x_1x_2$ on the whole $\hh^n$, i.e. the algebra of quaternions is commutative, which is false. \bs     
\end{remark}

Bearing in mind Remark \ref{rem:vector}, we have:

\begin{corollary}\label{cor:I}
The map $\I:\mr{Stem}(D,A\otimes\R^{2^n})\to\mc{S}(\OO_D,A)$, sending stem functions $F$ into the corresponding slice functions $\I(F)$, is a bijection, and hence a real linear isomorphism.
\end{corollary}

Another consequence is the following intrinsic characterization of sliceness.

\begin{corollary}[Sliceness criterion] 
\label{cor:sliceness-intrinsic}
Let $f:\OO_D \to A$ be a function. Then $f$ is a slice function if and only if there exist $I=(I_1,\ldots,I_n)\in(\cS_A)^n$ with the following property: 
\begin{equation} \label{eq:sliceness-intrinsic}
\textstyle
f(x)=2^{-n}\sum_{K,H \in \pa(n)}(-1)^{|K \cap H|}[J_K,[I_K^{-1},f(y^{\, c,H})]]
\end{equation}
for all $y=(\alpha_1+I_1\beta_1,\ldots,\alpha_n+I_n\beta_n) \in \OO_D$ and for all $x=(\alpha_1+J_1\beta_1,\ldots,\alpha_n+J_n\beta_n) \in \OO_D$, where $\alpha_1,\beta_1,\ldots,\alpha_n,\beta_n \in \R$ and $J=(J_1,\ldots,J_n)\in(\cS_A)^n$.
\end{corollary}
\begin{proof}
If $f$ is slice, then \eqref{eq:sliceness-intrinsic} follows from \eqref{eq:f}.

Suppose that \eqref{eq:sliceness-intrinsic} holds for some $I=(I_1,\ldots,I_n)\in(\cS_A)^n$. We will show that $f=\I(F)$ for some stem function $F:D \to A \otimes \R^{2^n}$. For each $K \in \pa(n)$, define the function $F_K:D \to A$ by setting
\[
\textstyle
F_K(z):=2^{-n}\sum_{H \in \pa(n)}(-1)^{|K \cap H|}[I_K^{-1},f(y^{\, c,H})],
\]
where $y=(\alpha_1+I_1\beta_1,\ldots,\alpha_n+I_n\beta_n)$ if $z=(\alpha_1+\ui \beta_1,\ldots,\alpha_n+\ui \beta_n)\in D$.

Fix $K \in \pa(n)$, $h \in \{1,\ldots,n\}$ and $z=(\alpha_1+\ui \beta_1,\ldots,\alpha_n+\ui \beta_n) \in D$. Set $y:=(\alpha_1+I_1\beta_1,\ldots,\alpha_n+I_n\beta_n)$. Note that 
$|K \cap H|=|K \cap (H \dsim \{h\})|-(-1)^{|H \cap \{h\}|}|K \cap \{h\}|$ for all $H \in \pa(n)$. Moreover, the map $\Psi_h:\pa(n) \to \pa(n)$, sending $H$ into $H \dsim \{h\}$, is a bijection. Bearing in mind the last two elementary facts and \eqref{eq:uHv=w}, we have that
\begin{eqnarray*}
\textstyle
2^n[I_K,F_K(\overline{z}^h)]
\!\!\!\! &=& \!\!\!\!
\textstyle
\sum_{H \in \pa(n)}(-1)^{|K \cap H|}f(y^{\, c,H \dsim \{h\}})=\\
&=& \!\!\!\!
\textstyle
\sum_{H \in \pa(n)}(-1)^{|K \cap (H \dsim \{h\})|+|K \cap \{h\}|}f(y^{\, c,H \dsim \{h\}})=\\
&=& \!\!\!\!
\textstyle
(-1)^{|K \cap \{h\}|}\sum_{H \in \pa(n)}(-1)^{|K \cap \Psi_h(H)|}f(y^{\, c,\Psi_h(H)})=
\\
&=& \!\!\!\!
\textstyle
(-1)^{|K \cap \{h\}|}\sum_{H \in \pa(n)}(-1)^{|K \cap H|}f(y^{\, c,H})=
\\
&=& \!\!\!\!
\textstyle
2^n(-1)^{|K \cap \{h\}|}[I_K,F_K(z)];
\end{eqnarray*}
consequently, $[I_K,F_K(\overline{z}^h)]=(-1)^{|K \cap \{h\}|}[I_K,F_K(z)]$. Using \eqref{eq:uHv=w} again, we deduce:
\[
F_K(\overline{z}^h)=[I_K^{-1},[I_K,F_K(\overline{z}^h)]]=(-1)^{|K \cap \{h\}|}[I_K^{-1},[I_K,F_K(z)]]=(-1)^{|K \cap \{h\}|}F_K(z).
\]
In other words, the function $F:D \to A \otimes \R^{2^n}$, defined by $F:=\sum_{K \in \pa(n)}e_KF_K$, is a stem function. Formula \eqref{eq:sliceness-intrinsic} now ensures that $f=\I(F)$.
\end{proof}

A consequence of the last result is as follows.

\begin{corollary}\label{cor:fl}
Let $\{f_l:\OO_D\to A\}_{l\in\N}$ be a sequence of slice functions, which pointwise converges to a function $f:\OO_D\to A$. Then $f$ is a slice function.
\end{corollary}
\begin{proof}
Let $x,y,I,J$ be as in the statement of Corollary \ref{cor:sliceness-intrinsic}. Since each $f_l$ is a slice function, equation \eqref{eq:sliceness-intrinsic} holds for each $f_l$. Note that $\{f_l(x)\}_{l\in\N}$ converges to $f(x)$ and $\{[J_K,[I_K^{-1},f_l(y^{\, c,H})]]\}_{l\in\N}$ converges to $[J_K,[I_K^{-1},f(y^{\, c,H})]]$ for all $K,H\in\pa(n)$. It follows that equation \eqref{eq:sliceness-intrinsic} holds also for $f$. Using Corollary \ref{cor:sliceness-intrinsic} again, we deduce that $f$ is a slice function.        
\end{proof}


\subsection{Spherical value and spherical derivatives}

Let $F=\sum_{K \in \pa(n)}e_KF_K$ be a stem function and let $f:\OO_D \to A$ be the slice function $\I(F)$.

\begin{definition}\label{def:vs}
We call \emph{spherical value of $f$} the slice function $\vs f:\OO_D\to A$ induced by the $A$-valued stem function $F_\emptyset:D\to A$, that is $\vs f:=\I(F_\emptyset)$. \bs
\end{definition}

From \eqref{eq:components-F} it follows that 
\begin{equation}\label{eq:vs}
\textstyle
\vs f(x)=2^{-n}\sum_{H\in\pa(n)}f(x^{c,H})
\end{equation}
for all $x\in\OO_D$. For each $K\in\pa(n)\setminus\{\emptyset\}$,we define:
\begin{itemize}
 \item $D_K:=\bigcap_{k\in K}\{(z_1,\ldots,z_n)\in D:z_k\not\in\R\}$; we assume to be non-empty.
 \item $F_K^*:D_K\to A$ by $F_K^*(z):=\beta_K^{-1}F_K(z)$, where $z=(\alpha_1+\ui \beta_1,\ldots,\alpha_n+\ui \beta_n)\in D_K$ and $\beta_K:=\prd_{k\in K}\beta_k$.
 \item $\R_K:=\bigcup_{k\in K}\{(x_1,\ldots,x_n)\in A^n:x_k\in\R\}$.
\end{itemize}

Note that $\R_K$ is closed in $A$, $D_K$ is invariant under all the complex conjugations of $\C^n$, and the circularization of $D_K$ in $A^n$ is equal to $\OO_D\setminus\R_K$, i.e.
\[
\OO_{D_K}=\OO_D\setminus\R_K.
\]
Furthermore, it is immediate to verify that the function $F_K^*$ is a $A$-valued stem function on $D_K$.

\begin{definition}\label{def:ds}
For each $K\in\pa(n)\setminus\{\emptyset\}$, we call \emph{spherical $K$-derivative of $f$} the slice function $f'_{s,K}:\OO_D\setminus\R_K\to A$ induced by $F_K^*$, that is $f'_{s,K}:=\I(F_K^*)$. \bs
\end{definition}

Bearing in mind \eqref{eq:components-F} and the equality $\IM(x)=\frac{x-x^c}{2}$, given any $K\in\pa(n)\setminus\{\emptyset\}$, we have 
\begin{align*}
\textstyle
f'_{s,K}(x)&\textstyle=2^{-n}[((\beta_kJ_k)_{k\in K})^{-1},\sum_{H\in\pa(n)}(-1)^{|K\cap H|}f(x^{c,H})]=\nonumber\\
&\textstyle=2^{-n}[((\IM(x_k))_{k\in K})^{-1},\sum_{H\in\pa(n)}(-1)^{|K\cap H|}f(x^{c,H})]
\end{align*}
for all $x=(x_1,\ldots,x_n)\in\OO_D\setminus\R_K$. As a consequence, if for each $x=(x_1,\ldots,x_n)\in A^n$ and $K\in\pa(n)\setminus\{\emptyset\}$ we set
\begin{equation}\label{eq:ImK}
\textstyle
\IM_K(x):=(\IM(x_k))_{k\in K},
\end{equation}
then we have
\begin{equation}\label{eq:f'sk}
\textstyle
f'_{s,K}(x)=2^{-n}[(\IM_K(x))^{-1},\sum_{H\in\pa(n)}(-1)^{|K\cap H|}f(x^{c,H})]
\end{equation}
for all $x\in\OO_D\setminus\R_K$. Note that the latter equality can be rewritten as follows:
\begin{equation}\label{eq:f'sk-2}
\textstyle
f'_{s,K}(x)=2^{|K|-n}[((x_k-(x_k)^c)_{k\in K})^{-1},\sum_{H\in\pa(n)}(-1)^{|K\cap H|}f(x^{c,H})].
\end{equation}

In order to simplify the notation, we set $D_\bullet:=D_{\{1,\ldots,n\}}$ and $\R_\bullet:=\R_{\{1,\ldots,n\}}$. Consequently,
\begin{align}
&D_\bullet\textstyle=\bigcap_{k=1}^n\{(z_1,\ldots,z_n)\in D:z_k\not\in\R\}\neq\emptyset,\\
&\R_\bullet\textstyle=\bigcup_{k=1}^n\{(x_1,\ldots,x_n)\in A^n:x_k\in\R\},\\
&\OO_{D_\bullet}\textstyle=\OO_D\setminus\R_\bullet=\bigcap_{k=1}^n\{(x_1,\ldots,x_n)\in\OO_D:\IM(x_k)\neq0\}\neq\emptyset.
\end{align}

Moreover, we set
\begin{equation}\label{eq:R-empty}
\R_\emptyset:=\emptyset.
\end{equation}

\begin{remark}
According to Definitions \ref{def:vs} and \ref{def:ds}, and \eqref{eq:R-empty}, we can also say that the spherical value of $f$ is the \emph{spherical $\emptyset$-derivative of $f$}, that is $f'_{s,\emptyset}:=\vs f$. \bs
\end{remark}

\begin{proposition}
Let $f:\OO_D\to A$ be a slice function. The following assertions hold.
\begin{itemize}
 \item[$(\mr{i})$] For each $x\in\OO_D$, the spherical value $\vs f$ is constant on $\cS_x$. For each $K\in\pa(n)\setminus\{\emptyset\}$ and for each $x\in\OO_D\setminus\R_K$, the spherical $K$-derivative $f'_{s,K}$ is constant on $\cS_x$. More precisely, if $f=\I(F)$, then
\begin{equation}
\text{$\vs f(x)=F_\emptyset(z)\,$ for all $x\in\OO_D$}
\end{equation}
and
\begin{equation}
\text{$f'_{s,K}(x)=\beta_K^{-1}F_K(z)\,$ for all $x\in\OO_D\setminus\R_K$,}
\end{equation}
where $z:=(\alpha_1+i\beta_1,\ldots,\alpha_n+i\beta_n)\in D$ if $x=(\alpha_1+J_1\beta_1,\ldots,\alpha_n+J_n\beta_n)\in\OO_D$.
 \item[$(\mr{ii})$] If $x=(x_1,\ldots,x_n)\in\OO_D\setminus\R_L$ for some $L\in\pa(n)$ and $x_h\in\R$ for all $h\in\{1,\ldots,n\}\setminus L$, then it holds the formula
\begin{equation}\label{eq:spherical_representation_L}
\textstyle
f(x)=\vs f(x)+\sum_{K\in\pa(n)\setminus\{\emptyset\},K\subset L}[\IM_K(x),f'_{s,K}(x)].
\end{equation}
In particular, for all $x\in\OO_D\setminus\R_\bullet$, we have:
\begin{equation}\label{eq:spherical_representation}
\textstyle
f(x)=\vs f(x)+\sum_{K\in\pa(n)\setminus\{\emptyset\}}[\IM_K(x),f'_{s,K}(x)].
\end{equation}

 \item[$(\mr{iii})$] If $x=(x_1,\ldots,x_n)\in\OO_D\setminus\R_L$ for some $L\in\pa(n)$ and $x_h\in\R$ for all $h\in\{1,\ldots,n\}\setminus L$, then $f$ is constant on $\cS_x$ if and only if $f'_{s,K}(x)=0$ for all $K\in\pa(n)\setminus\{\emptyset\}$ with $K\subset L$. In this case, $f$ takes the value $\vs f(x)$ on $\cS_x$.
 
In particular, for each $x\in\OO_D\setminus\R_\bullet$, $f$ is constant on $\cS_x$ if and only if $f'_{s,K}(x)=0$ for all $K\in\pa(n)\setminus\{\emptyset\}$.
\end{itemize}
\end{proposition}
\begin{proof}
Point $(\mr{i})$ follows immediately from the fact that the stem functions inducing $\vs f$ and the $f'_{s,K}$'s are $A$-valued.

Let $x=(\alpha_1+J_1\beta_1,\ldots,\alpha_n+J_n\beta_n)\in\OO_D\setminus\R_L$ for some $L\in\pa(n)$ and $\beta_h=0$ for all $h\in\{1,\ldots,n\}\setminus L$, and let $z=(\alpha_1+i\beta_1,\ldots,\alpha_n+i\beta_n)\in D$. Denote $F=\sum_{K\in\pa(n)}e_KF_k$ the stem function inducing $f$. By \eqref{eq:stem}, if $K\in\pa(n)$ with $K\not\subset L$
, then $F_K(z)=0$. Consequently, $f(x)=\sum_{K\in\pa(n),K\subset L}[J_K,F_K(z)]$, where $J=(J_1,\ldots,J_n)$. On the other hand, by the very definitions of spherical value and derivatives, we deduce:
\begin{align*}
f(x)&\textstyle=\sum_{K\in\pa(n),K\subset L}[J_K,F_K(z)]=\vs f(x)+\sum_{K\in\pa(n)\setminus\{\emptyset\},K\subset L}[\IM_K(x),\beta_K^{-1}F_K(z)]=\\
&\textstyle=\vs f(x)+\sum_{K\in\pa(n)\setminus\{\emptyset\},K\subset L}[\IM_K(x),f'_{s,K}(x)].
\end{align*}
This proves \eqref{eq:spherical_representation_L}, which reduces to \eqref{eq:spherical_representation} when $L=\{1,\ldots,n\}$.

Let us prove $(\mr{iii})$. If $f'_{s,K}(x)=0$ for each $K\in\pa(n)\setminus\{\emptyset\}$ with $K\subset L$, then $(\mr{i})$ and \eqref{eq:spherical_representation_L} imply at once that such $f'_{s,K}$'s vanish on the whole $\cS_x$ and $f$ is constantly equal to $\vs f(x)$ on $\cS_x$. Finally, suppose that $f$ is constantly equal on $\cS_x$ to some $a\in A$. Choose $K\in\pa(n)\setminus\{\emptyset\}$ with $K\subset L$. Since $x\in\OO_D\setminus\R_K$ and $x^{c,H}\in\cS_x$ for all $H\in\pa(n)$, \eqref{eq:f'sk} and Lemma \ref{lem:combinatorial} ensure that
\begin{align*}
f'_{s,K}(x)&\textstyle=2^{-n}[(\IM_K(x))^{-1},\sum_{H\in\pa(n)}(-1)^{|K\cap H|}a]=\\
&\textstyle=2^{-n}\big(\sum_{H\in\pa(n)}(-1)^{|K\cap H|}\big)[(\IM_K(x))^{-1},a]=\\
&=\delta_{K,\emptyset}[(\IM_K(x))^{-1},a]=0.
\end{align*}
The proof is complete.
\end{proof}

We now show that there exists a relation between the spherical value and derivatives of $f$ and their one-variable analogues introduced in \cite[Definition 6]{AIM2011}. Let $z=(z_1,\ldots,z_n)\in D$ and let $h\in\{1,\ldots,n\}$. Denote $D_h(z)$ the subset of $\C$ defined by
\begin{equation}\label{def:Dhz}
D_h(z):=\{w\in\C:(z_1,\ldots,z_{h-1},w,z_{h+1},\ldots,z_n)\in D\}.
\end{equation}
Since $D$ is invariant under all the complex conjugations of $\C^n$, it follows immediately that $D_h(z)$ is invariant under the complex conjugation of $\C$. Note that $D_h(z)\neq\emptyset$, because it contains $z_h$. Moreover, $D_h(z)$ is open in $\C$ if $D$ is open in $\C^n$.

Let $x=(x_1,\ldots,x_n)\in \OO_D$. Denote $\OO_{D,h}(x)$ the subset of $Q_A$ defined by
\begin{equation}\label{def:OmegaDhx}
\OO_{D,h}(x):=\{a\in A:(x_1,\ldots,x_{h-1},a,x_{h+1},\ldots,x_n)\in\OO_D\}.
\end{equation}

Suppose that $x\in\OO_{\{z\}}$. Let us show that $\OO_{D_h(z)}=\OO_{D,h}(x)$. First, note that, if we write $z=(\alpha_1+i\beta_1,\ldots,\alpha_n+i\beta_n)$ with $\alpha_1,\ldots,\alpha_n,\beta_1,\ldots,\beta_n\in\R$, then $x_\ell=\alpha_\ell+J_\ell\beta_\ell$ for each $\ell\in\{1,\ldots,n\}$ and for some $J_\ell\in\cS_A$. Let $a\in\OO_{D_h(z)}$. Write $a=\alpha+J\beta$ with $\alpha,\beta\in\R$ and $J\in\cS_A$. By definition of $\OO_{D_h(z)}$ and $D_h(z)$, we have that $\alpha+i\beta\in D_h(z)$ and
\[
(\alpha_1+i\beta_1,\ldots,\alpha_{h-1}+i\beta_{h-1},\alpha+i\beta,\alpha_{h+1}+i\beta_{h+1},\ldots,\alpha_n+i\beta_n)\in D,
\]
respectively. The definitions of $\OO_D$ and $\OO_{D,h}(x)$ imply that
$(x_1,\ldots,x_{h-1},a,x_{h+1},\ldots,x_n)\in\OO_D$ and $a\in\OO_{D,h}(x)$, respectively. Vice versa, if $a\in A$ with $(x_1,\ldots,x_{h-1},a,x_{h+1},\ldots,x_n)\in\OO_D$, then there exists $z'=(z'_1,\ldots,z'_n)\in D$ with $z'_\ell=\alpha'_\ell+i\beta'_\ell$ and $J'_\ell\in\cS_A$ for each $\ell\in\{1,\ldots,n\}$ such that $\alpha'_h+J'_h\beta'_h=a$, and $\alpha_\ell+J_\ell\beta_\ell=x_\ell=\alpha'_\ell+J'_\ell\beta'_\ell$ for each $\ell\in\{1,\ldots,n\}\setminus\{h\}$. Let $\ell\neq h$. Note that, if $\beta_\ell=0$, then $\beta'_\ell=0$ as well, and $\alpha_\ell=x_\ell=\alpha'_h$. If $\beta_\ell\neq0$, then either $(\beta'_\ell,J'_\ell)=(\beta_\ell,J_\ell)$ or $(\beta'_\ell,J'_\ell)=(-\beta_\ell,-J_\ell)$. Define $H:=\{\ell\in\{1,\ldots,n\}\setminus\{h\}:\beta_\ell\neq0,(\beta'_\ell,J'_\ell)=(-\beta_\ell,-J_\ell)\}$. Since $D$ is invariant under all complex conjugations of $\C^n$, it follows that
\[
(z_1,\ldots,z_{h-1},\alpha'_h+i\beta'_h,z_{h+1},\ldots,z_n)=
\overline{\,z'\,}^{H}\in D.
\]
Hence $a=\alpha'_h+J'_h\beta'_h\in\OO_{D_h(z)}$. We have just proven that
\begin{equation}\label{eq:Dhz}
\OO_{D_h(z)}=\OO_{D,h}(x)
\end{equation}
for all $z=(z_1,\ldots,z_n)\in D$, $x=(x_1,\ldots,x_n)\in\OO_{\{z\}}\subset\OO_D$ and $h\in\{1,\ldots,n\}$.

\begin{definition}\label{def:gy}
Let $g:\OO_D\to A$ be a function and let $h\in\{1,\ldots,n\}$. We say that $g$ is a \textit{slice function w.r.t.\ $\mr{x}_h$} if, for each $y=(y_1,\ldots,y_n)\in\OO_D$, the restriction function $g_h^{\sss(y)}:\OO_{D,h}(y)\to A$, defined by
\[
g_h^{\sss(y)}(x_h):=g(y_1,\ldots,y_{h-1},x_h,y_{h+1},\ldots,y_n),
\]
is a slice function. \bs 
\end{definition}

Let $g:\OO_D\to A$ be a slice function w.r.t.\ $\mr{x}_h$, let $y\in\OO_D$, let $\vs{(g_h^{\sss(y)})}:\OO_{D,y}(y)\to A$ and $(g_h^{\sss(y)})'_s:\OO_{D,h}(y)\setminus\R\to A$ be the usual spherical value and spherical derivative of the one variable slice function $g_h^{\sss(y)}$, respectively. If $z$ is a point of $D$ such that $y\in\OO_{\{z\}}$ and $g_h^{\sss(y)}$ is induced by the stem function $G_1+\ui G_2:D_h(z)\to A\otimes\R^2$, then 
$\vs{(g_h^{\sss(y)})}(x_h)=G_1(w)$ for all $x_h=\alpha_h+J_h\beta_h\in\OO_{D_h(z)}$, where $w:=\alpha_h+i\beta_h\in D_h(z)$, and $(g_h^{\sss(y)})'_s(x_h)=\beta_h^{-1} G_2(w)$ if $\beta_h\neq0$. As a consequence, we have that $\vs{(g_h^{\sss(y)})}(x_h)=\frac{1}{2}(g_h^{\sss(y)}(x_h)+g_h^{\sss(y)}((x_h)^c))$ and $(g_h^{\sss(y)})'_s(x_h)=\frac{1}{2}(\IM(x_h))^{-1}(g_h^{\sss(y)}(x_h)-g_h^{\sss(y)}((x_h)^c))$.

Assume that $g$ is a slice function w.r.t.\ $\mr{x}_h$. Then, for each $e\in\{0,1\}$, we define the function $\SD_{\mr{x}_h}^0g:\OO_D\to A$ and $\SD_{\mr{x}_h}^1g:\OO_D\setminus\R_{\{h\}}\to A$ by setting
\begin{equation}\label{eq:ge0}
\SD_{\mr{x}_h}^0g(x):=\vs{(g_h^{\sss(x)})}(x_h)\, \text{ for all $x=(x_1,\ldots,x_n)\in\OO_D$}
\end{equation}
and
\begin{equation}\label{eq:ge1}
\SD_{\mr{x}_h}^1g(x):=(g_h^{\sss(x)})'_s(x_h)\, \text{ for all $x=(x_1,\ldots,x_n)\in\OO_D\setminus\R_{\{h\}}$.}
\end{equation}

Given any $K\in\pa(n)$ and $h\in\{1,\ldots,n\}$, define $K_h:=K\cap\{1,\ldots,h\}$.

\begin{proposition}\label{prop:slicewrtx1}
Assume that $n\geq2$. Let $f:\OO_D\to A$ be a slice function, let $K\in\pa(n)$ and let $\varepsilon:\{1,\ldots,n\}\to\{0,1\}$ be the characteristic function of $K$. Then $f$ is a slice function w.r.t.\ $\mr{x}_1$ and, for each $h\in\{2,\ldots,n\}$, the function $\SD^{\varepsilon(h-1)}_{\mr{x}_{h-1}}\cdots \SD^{\varepsilon(1)}_{\mr{x}_1}f:\OO_D\setminus\R_{K_{h-1}}\to A$, obtained iterating \eqref{eq:ge0} and \eqref{eq:ge1} as follows $\SD^{\varepsilon(h-1)}_{\mr{x}_{h-1}}\cdots\SD^{\varepsilon(1)}_{\mr{x}_1}f:=\SD^{\varepsilon(h-1)}_{\mr{x}_{h-1}}(\cdots(\SD^{\varepsilon(2)}_{\mr{x}_2}(\SD^{\varepsilon(1)}_{\mr{x}_1}f))\cdots)$, is a well-defined slice function w.r.t.\ $\mr{x}_h$. Moreover, it holds:
\[
\SD^{\varepsilon(n)}_{\mr{x}_n}\cdots \SD^{\varepsilon(1)}_{\mr{x}_1}f:=\SD^{\varepsilon(n)}_{\mr{x}_n}(\SD^{\varepsilon(n-1)}_{\mr{x}_{n-1}}\cdots \SD^{\varepsilon(1)}_{\mr{x}_1}f)=
\begin{cases}
\vs f&\text{ if }K=\emptyset,\vspace{.3em}\\
f'_{s,K}&\text{ if }K\neq\emptyset.
\end{cases}  
\]
\end{proposition}
\begin{proof}
Let $F=\sum_{H\in\pa(n)}e_HF_H:D\to A\otimes\R^{2^n}$ be the stem function inducing $f$, let $y=(y_1,\ldots,y_n)=(\alpha_1+I_1\beta_1,\ldots,\alpha_n+I_n\beta_n)\in\OO_D$, let $I:=(I_1,\ldots,I_n)$ and let $w=(w_1,\ldots,w_n):=(\alpha_1+i\beta_1,\ldots,\alpha_n+i\beta_n)\in D$. Let us prove by induction on $h\in\{1,\ldots,n\}$ that the two following properties hold true:
\begin{itemize}
 \item[(a)] $\SD^{\varepsilon(h-1)}_{\mr{x}_{h-1}}\cdots\SD^{\varepsilon(1)}_{\mr{x}_1}f$ is a slice function w.r.t.\  $\mr{x}_h$, where for convention  $\SD^{\varepsilon(h-1)}_{\mr{x}_{h-1}}\cdots \SD^{\varepsilon(1)}_{\mr{x}_1}f:=f$ if $h=1$.
 \item[(b)] $\SD^{\varepsilon(h)}_{\mr{x}_h}\cdots \SD^{\varepsilon(1)}_{\mr{x}_1}f(y)=
(\beta_{K_h})^{-1}\sum_{H\in\pa(n),H_h=\emptyset}[I_H,F_{H\cup K_h}(z)]$, where $\beta_\emptyset:=1$.
\end{itemize}

First, we consider the case $h=1$. It holds:
\begin{equation}\label{eq:fx1}
\textstyle
f(x_1,y')=\sum_{H\in\pa(n),1\not\in H}[J_H,F_H(z_1,w')]+J_1\big(\sum_{H\in\pa(n),1\not\in H}[J_H,F_{H\cup\{1\}}(z_1,w')]\big),
\end{equation}
where $x_1=\alpha_1+J_1\beta_1\in\OO_{D,1}(y)$, $y':=(y_2,\ldots,y_n)$, $z_1:=\alpha_1+i\beta_1\in D_1(w)$, $w':=(w_2,\ldots,w_n)$ and $J=(J_1,I_2,\ldots,I_n)$. Define the functions $F_1,F_2:D_1(w)\to A$ by setting
\begin{align*}
F_1(z_1)&\textstyle:=\sum_{H\in\pa(n),1\not\in H}[J_H,F_H(z_1,w')],\\
F_2(z_1)&\textstyle:=\sum_{H\in\pa(n),1\not\in H}[J_H,F_{H\cup\{1\}}(z_1,w')].
\end{align*}
It is immediate to verify that $F_1+iF_2:D_1(w)\to A\otimes\R^2$ is a stem function. Consequently, $f$ is slice w.r.t.\ $\mr{x}_1$. Moreover, we have: 
\[
\textstyle
\SD^0_{\mr{x}_1}f(y)=\sum_{H\in\pa(n),1\not\in H}[I_H,F_H(z)]
\;\text{ and }\;
\SD^1_{\mr{x}_1}f(y)=\beta_1^{-1}\sum_{H\in\pa(n),1\not\in H}[I_H,F_{H\cup\{1\}}(z)].
\]
This proves that $f$ satisfies (a) and (b) for $h=1$.

Assume (a) and (b) are verified for some $h\in\{1,\ldots,n-1\}$. By (b), we deduce:
\begin{align}
&\SD^{\varepsilon(h)}_{\mr{x}_h}\cdots \SD^{\varepsilon(1)}_{\mr{x}_1}f(y'',x_{h+1},\hat{y})=\textstyle
(\beta_{K_{h-1}}\beta_{h,K})^{-1}\sum_{H\in\pa(n),H_{h+1}=\emptyset}[L_H,F_{H\cup K_h}(z'',z_{h+1},\hat{z})]+\nonumber\\
&\textstyle+J_{h+1}
\big((\beta_{K_{h-1}}\beta_{h,K})^{-1}\sum_{H\in\pa(n),H_{h+1}=\emptyset}[L_H, F_{H\cup K_h\cup\{h+1\}}(z'',z_{h+1},\hat{z})]\big),\label{eq:DD}
\end{align}
where $x_h=\alpha_h+J_h\beta_h\in\OO_{D,h}(y)$, $\beta_{h,K}=\beta_h$ if $h\in K$, $\beta_{h,K}=1$ if $h\not\in K$, $y''=(y_1,\ldots,y_{h})$, $\hat{y}=(y_{h+2},\ldots,y_n)$, $z_h=\alpha_h+i\beta_h\in D_h(z)$, $z''=(z_1,\ldots,z_{h})$, $\hat{z}=(z_{h+2},\ldots,z_n)$ and $L=(I_1,\ldots,I_{h},J_{h+1},I_{h+2},\ldots,I_n)$. Here $\hat{y}$ and $\hat{z}$ are omitted if $h+1=n$.
Proceeding as above, it is immediate to verify that $\SD^{\varepsilon(h)}_{\mr{x}_h}\cdots\SD^{\varepsilon(1)}_{\mr{x}_1}f$ is slice w.r.t.\ $\mr{x}_{h+1}$, i.e.\ (a) is satisfied for $h+1$. Moreover, we have:
\begin{align*}
\SD^0_{\mr{x}_{h+1}}\SD^{\varepsilon(h)}_{\mr{x}_h}\cdots\SD^{\varepsilon(1)}_{\mr{x}_1}f(y)&\textstyle=(\beta_{K_h})^{-1}\sum_{H\in\pa(n),H_{h+1}=\emptyset}[I_H, F_{H\cup K_h}(z)],\\\notag
\SD^1_{\mr{x}_{h+1}}\SD^{\varepsilon(h)}_{\mr{x}_h}\cdots\SD^{\varepsilon(1)}_{\mr{x}_1}f(y)&\textstyle=
\beta_{h+1}^{-1}\big(\beta_{K_h}\big)^{-1}\sum_{H\in\pa(n),H_{h+1}=\emptyset}[I_H,F_{H\cup K_h\cup\{h+1\}}(z)].
\end{align*}
In both cases, the last two expressions are equal to 
\[
\textstyle
(\beta_{K_{h+1}})^{-1}\sum_{H\in\pa(n),H_{h+1}=\emptyset}[I_H, F_{H\cup K_{h+1}}(z)]
\]
and the induction step works. When $h=n$, the right-hand side of (b) becomes $F_\emptyset(z)$ if $K=\emptyset$, and $\beta_K^{-1}F_K(z)$ if $K\ne\emptyset$. This completes the proof.
\end{proof}

\begin{definition}\label{def:D_epsilon}
Assume that $n\geq2$. Let $h\in\{2,\ldots,n\}$ and let $\epsilon:\{1,\ldots,h-1\}\to\{0,1\}$ be any function. Given a slice function $f:\OO_D\to A$, we define the \emph{truncated spherical $\epsilon$-derivative $\SD_\epsilon f:\OO_D\setminus\R_{\epsilon^{-1}(1)}\to A$ of $f$} by setting $\SD_\epsilon f:=\SD^{\epsilon(h-1)}_{\mr{x}_{h-1}}\cdots\SD^{\epsilon(1)}_{\mr{x}_1}f$, and we say that such a derivative has \emph{order $h-1$}. For convention, we define also the  \emph{truncated spherical $\emptyset$-derivative $\SD_\emptyset f:\OO_D\to A$ of $f$} by setting $\SD_\emptyset f:=f$, and we say that such a derivative has \emph{order $0$}. \bs
\end{definition}

Proposition \ref{prop:slicewrtx1} asserts that each $h$-order truncated spherical derivative of a slice function is a well-defined slice function w.r.t.\ $\mr{x}_h$. Moreover, a by-product of the proof of the mentioned proposition reads as follows. If $F:D\to A\otimes\R^{2^n}$ is the stem function inducing $f$, then
\begin{equation}\label{eq:SD}
\textstyle
\SD_\epsilon f(x)=(\beta_{\epsilon^{-1}(1)})^{-1}\sum_{H\in\pa(n),H_{h-1}=\emptyset}[J_H,F_{H\cup\epsilon^{-1}(1)}(z)]
\end{equation}
for all $x=(\alpha_1+J_1\beta_1,\ldots,\alpha_n+J_n\beta_n)\in\OO_D\setminus\R_{\epsilon^{-1}(1)}$ with $z:=(\alpha_1+i\beta_1,\ldots,\alpha_n+i\beta_n)$ and $J:=(J_1,\ldots,J_n)$, where $\beta_\emptyset:=1$.


\subsection{Smoothness}

By Assumption \ref{assumption1}, we are assuming that the real alternative $^*$-algebra $A$ we are working with has finite dimension and, as a finite dimensional real vector space, $A$ is equipped with the natural $\mscr{C}^\omega$~manifold structure defined by the global coordinate systems associated with its real vector bases. Here `$\,\mscr{C}^\omega\,$' means `real analytic'. For each $n\geq1$, we equip $A^n$ with the corresponding product structure of $\mscr{C}^\omega$~manifolds. We call the underlying topology on $A^n$ as Euclidean topology of $A^n$. Given any non-empty subset $S$ of $A^n$, we equip $S$ with the relative topology induced by the Euclidean one of $A^n$. We call such a topology on $S$ as Euclidean topology of $S$. If in addition $S$ is open in $A^n$, then we always assume that $S$ is equipped with the $\mscr{C}^\omega$~manifold structure induced by the one of $A^n$.

Similarly, we equip $D$ with the Euclidean topology induced by the one of $\C=\R^2$ and, in the case $D$ is open in $\C$, we always assume that $D$ is equipped with the $\mscr{C}^\omega$~manifold structure induced by the one of $\C=\R^2$. 

As usual, given two topological spaces $X$ and $Y$, we denote $\mscr{C}^0(X,Y)$ the set of all continuous maps from $X$ to $Y$. If $r\in(\N\setminus\{0\})\cup\{\infty,\omega\}$ and $X$ and $Y$ are equipped with some $\mscr{C}^r$ manifold structures (for instance, $\mscr{C}^\omega$ manifold structures), then the symbol $\mscr{C}^r(X,Y)$ indicates the set of all $\mscr{C}^r$ maps from $X$ to $Y$. In the latter case, given any non-empty subset $S$ of $X$ and a map $f:S\to Y$, we say that $f$ is a \textit{$\mscr{C}^r$ map} if there exist an open neighborhood $U$ of $S$ in $X$ and a map $g:U\to Y$ such that $g(x)=f(x)$ for all $x\in S$ and, equipping $U$ with the natural $\mscr{C}^r$ manifold structure induced by the one of $X$, $g$ belongs to $\mscr{C}^r(U,Y)$. We denote $\mscr{C}^r(S,Y)$ the set of all $\mscr{C}^r$ maps from $S$ to $Y$.     

\begin{definition}
We define 
\begin{itemize}
 \item $\mr{Stem}^0(D,A\otimes\R^{2^n})$ as the set of all continuous stem functions from $D$ to $A\otimes\R^{2^n}$, i.e. the set of all stem functions $F=\sum_{K\in\pa(n)}e_KF_K:D\to A\otimes\R^{2^n}$ such that each $F_K$ belongs to $\mscr{C}^0(D,A)$,
 \item $\mc{S}^0(\OO_D,A)$ as the set of slice functions from $\OO_D$ to $A$ induced by continuous stem functions, i.e. $\mc{S}^0(\OO_D,A):=\I(\mr{Stem}^0(D,A\otimes\R^{2^n}))$,
\end{itemize}
and, in the case $D$ is open in $\C$ and $r\in(\N\setminus\{0\})\cup\{\infty,\omega\}$,
\begin{itemize}
 \item $\mr{Stem}^r(D,A\otimes\R^{2^n})$ as the set of all $\mscr{C}^r$ stem functions from $D$ to $A\otimes\R^{2^n}$, i.e. the set of all stem functions $F=\sum_{K\in\pa(n)}e_KF_K:D\to A\otimes\R^{2^n}$ such that each $F_K$ belongs to $\mscr{C}^r(D,A)$,
  \item $\mc{S}^r(\OO_D,A)$ as the set of slice functions from $\OO_D$ to $A$ induced by $\mscr{C}^r$ stem functions, i.e. $\mc{S}^r(\OO_D,A):=\I(\mr{Stem}^r(D,A\otimes\R^{2^n}))$. \bs
\end{itemize}
\end{definition}

Equip $\N\cup\{\infty,\omega\}$ with the unique total ordering $\leq$, extending the one $\leq$ of $\N$, by requiring that $s\leq\infty$ for all $s\in\N$, and $\infty\leq\omega$. Denote $\lfloor s\rfloor$ the integer part of $s\in\R$.

\begin{theorem} \label{thm:Cr}
The following assertions hold.
\begin{itemize}
 \item[$(\mr{i})$] If $\cS_A$ is compact, then $\mc{S}^0(\OO_D,A)\subset\mscr{C}^0(\OO_D,A)$.
 \item[$(\mr{ii})$] Suppose that $D$ is open in $\C^n$. Let $r\in\N\cup\{\infty,\omega\}$ such that $r\geq 2^n-1$, and let $\mb{w}_n(r)$ be the element of $\N\cup\{\infty,\omega\}$ defined by $\mb{w}_n(r):=r$ if $r\in\{\infty,\omega\}$ and
\[
\textstyle
\mb{w}_n(r):=\left\lfloor\frac{r-2^n+1}{2^n}\right\rfloor=\left\lfloor\frac{r+1}{2^n}\right\rfloor-1
\]
if $r\in\N$ and $r\geq2^n-1$. Then it holds:
\[
\mc{S}^r(\OO_D,A)\subset\mscr{C}^{\mb{w}_n(r)}(\OO_D,A).
\]
In particular, we have $\mc{S}^\infty(\OO_D,A)\subset\mscr{C}^\infty(\OO_D,A)$ and $\mc{S}^\omega(\OO_D,A)\subset\mscr{C}^\omega(\OO_D,A)$.
\end{itemize}
\end{theorem}
\begin{proof}
Choose a real vector basis $\mc{B}=(u_1,\ldots,u_d)$ of $A$ with $u_1=1$, and denote $\pi_\R:A\to\R$ the projection of $A$ onto the first component of the coordinates induced by $\mc{B}$, i.e. the real linear function sending each $a=\sum_{h=1}^da_hu_h\in A$ into $a_1\in\R$. Define the functions $\theta,\eta,\xi:A\to\R$ and $v,w:A\to\C$ by setting
\[
\theta(a):=\pi_\R(\RE(a)), \quad \eta(a):=\pi_\R(n(\IM(x))), \quad \xi(a):=\sqrt{|\eta(a)|}
\]
and
\[
v(a):=\theta(a)+i\xi(a), \quad w(a):=\theta(a)+i\eta(a).
\]
Note that $\theta,\eta\in\mscr{C}^\omega(A,\R)$, $\xi\in\mscr{C}^0(A,\R)$, $v\in\mscr{C}^0(A,\C)$ and $w\in\mscr{C}^\omega(A,\C)$. Moreover, it holds
\begin{align}
&v(\alpha+J\beta)=\alpha+i|\beta|,\label{eq:wh1}\\
&w(\alpha+J\beta)=\alpha+i\beta^2\label{eq:wh2}
\end{align}
for all $\alpha,\beta\in\R$ and $J\in\cS_A$.  Define also the maps $v_n,w_n:A^n\to\C^n$ by setting
\[
v_n(x_1,\ldots,x_n):=(v(x_1),\ldots,v(x_n)) \;\text{ and }\; w_n(x_1,\ldots,x_n):=(w(x_1),\ldots,w(x_n)).
\]
Let $C$ be the closed subset of $A$ defined by $C:=\xi^{-1}(0)=\eta^{-1}(0)$. Since $C\cap Q_A=\R$, we have that $(A\setminus C)\cap Q_A=Q_A\setminus C=Q_A\setminus\R$, and hence $(A\setminus C)^n\cap\OO_D=\OO_D\setminus\R_\bullet$. Let $\mr{j}:A\setminus C\to A$ be the continuous map defined by
\[
\textstyle
\mr{j}(a):=\frac{1}{\xi(a)}\IM(a).
\]  
Define $\mr{J}_\emptyset:=1$ and, for each $K=\{k_1,\ldots,k_p\}\in\pa(n)\setminus\{\emptyset\}$ with $k_1<\ldots<k_p$, define the continuous map $\mr{J}_K:\OO_D\setminus\R_{K}\to(\cS_A)^{|K|}$ by
\[
\mr{J}_K(x_1,\ldots,x_n):=(\mr{j}(x_{k_1}),\ldots,\mr{j}(x_{k_p})).
\]
Denote $\mr{J}_\bullet:Q_A\setminus\R_\bullet\to(\cS_A)^n$ the continuous map $\mr{J}_{\{1,\ldots,n\}}$. Note that $\mr{J}_\bullet(\alpha_1+J_1|\beta_1|,\ldots,\alpha_n+J_n|\beta_n|)=(J_1,\ldots,J_n)$ if $(\alpha_1+i\beta_1,\ldots,\alpha_n+i\beta_n)\in D_\bullet$ and $J_1,\ldots,J_n\in\cS_A$. 

Choose $F=\sum_{K\in\pa(n)}e_KF_K\in\mr{Stem}(D,A\otimes\R^{2^n})$ and define $f:=\I(F)$.

Let us prove $(\mr{i})$.

Suppose that $F$ is continuous and $\cS_A$ is compact. For each $H\in\pa(n)$ and $\ell\in\{0,1,\ldots,n\}$, define $Q(H):=\bigcap_{h\in\{1,\ldots,n\}\setminus H}\{(x_1,\ldots,x_n)\in\OO_D:x_h\in\R\}$ and $Q(\ell):=\bigcup_{H\in\pa(n),|H|\leq\ell}Q(H)$. Note that the $Q(H)$'s and the $Q(\ell)$'s are closed subsets of $\OO_D$; moreover, $Q(0)=Q(\emptyset)=\OO_D\cap\R^n$ and $Q(n)=Q(\{1,\ldots,n\})=\OO_D$. We will prove by induction on $\ell\in\{0,1,\ldots,n\}$ that the restriction $f|_{Q(\ell)}$ of $f$ to $Q(\ell)$ is continuous. Evidently, if the latter assertion is true then $f$ is continuous, because $Q(n)=\OO_D$. The case $\ell=0$ follows immediately from the equality $f(x)=F_\emptyset(v_n(x))$ for all $x\in Q(0)=\OO_D\cap\R^n$. Suppose the assertion is true for some $\ell\in\{0,1,\ldots,n-1\}$. Note that, for each $H,L\in\pa(n)$ with $|H|=|L|=\ell+1$ and $H\neq L$, we have that $|H\cap L|\leq\ell$ and hence
\[
\textstyle
Q(H)\cap Q(L)=\bigcap_{h\in\{1,\ldots,n\}\setminus (H\cap L)}\{(x_1,\ldots,x_n)\in\OO_D:x_h\in\R\}\subset Q(\ell).
\]
It follows that $\{Q(H)\}_{H\in\pa(n),|H|=\ell{+1}}$ is a finite closed cover of $Q(\ell+1)$ and, for each $H,L\in\pa(n)$ with $|H|=|L|=\ell+1$ and $H\neq L$, the restrictions $f|_{Q(H)}$ and $f|_{Q(L)}$ are continuous on $Q(H)\cap Q(L)$ by induction. Consequently, it suffices to show that, for each fixed $H\in\pa(n)$ with $|H|=\ell+1$, $f|_{Q(H)}$ is continuous. By induction, $f|_{Q(\ell)}$ is continuous so the same is true for $f|_P$, where $P:=Q(H)\cap Q(\ell)$. The set $P$ is closed in $Q(H)$ and
\[
Q(H)\setminus P=Q(H)\setminus Q(\ell)\subset\OO_D\setminus\R_\bullet.
\]
Since $f(x)=\sum_{K\in\pa(n)}[\mr{J}_\bullet(x),F_K(v_n(x))]$ for all $x\in\OO_D\setminus\R_\bullet$, it follows that $f|_{Q(H)\setminus P}$ is continuous. Now, in order to complete the proof of $(\mr{i})$, it suffices to show that, if $\{y_m\}_{m\in\N}$ is a sequence in $Q(H)\setminus P$ converging to some point $x=(x_1,\ldots,x_n)\in P$, then the sequence $\{f(y_m)\}_{m\in\N}$ converges to $f(x)$. Consider such a sequence $\{y_m\}_{m\in\N}$ in $Q(H)\setminus P$ and $x=(x_1,\ldots,x_n)\in P$. Define $H^*:=\{h\in H:x_h\in\R\}$. Note that, by definition of $P$, $H^*\neq\emptyset$. By the even-odd properties of the $F_K$'s, we deduce at once that
\begin{equation}\label{eq:HH*}
\textstyle
f(y_m)=\sum_{K\in\pa(n),K\subset H}[\mr{J}_K(y_m),F_K(v_n(y_m))]
\end{equation}
for all $m\in\N$, and
\begin{equation}\label{eq:HH**}
\textstyle
f(x)=\sum_{K\in\pa(n),K\subset H\setminus H^*}[\mr{J}_K(x),F_K(v_n(x))].
\end{equation}
Let $K\in\pa(n)$ with $K\subset H$ and $K\cap H^*\neq\emptyset$. Choose $\nu\in K\cap H^*$ and observe that $x_\nu\in\R$. Since $v_n$ and $F_K$ are continuous, the sequence $\{F_K(v_n(y_m))\}_{m\in\N}$ converges to $F_K(v_n(x))$. If we write $v_n(x)=(z_1,\ldots,z_n)\in\C^n$, then $z_\nu=x_\nu\in\R$ and $F_K(v_n(x))=0$. On the other hand, $(\cS_A)^n$ is compact in $A^n$ and hence it is bounded. It follows that the sequence $\{[\mr{J}_K(y_m),F_K(v_n(y_m))]\}_{m\in\N}$ converges to zero. Consequently, by \eqref{eq:HH*}, $\{f(y_m)\}_{m\in\N}$ converges to $\sum_{K\in\pa(n),K\subset H\setminus H^*}[\mr{J}_K(x),F_K(v_n(x))]$, which is equal to $f(x)$ by \eqref{eq:HH**}.

It remains to show point $(\mr{ii})$.

Suppose that $D$ is open in $\C^n$. Assume that $F$ is of class $\mscr{C}^r$ for $r\in\N$ with $r\geq2^n-1$. Let $\rho:\Z\to\Z$ be the function $\rho(s):=\big\lfloor\frac{s-1}{2}\big\rfloor$ and, for each $k\geq 1$, let $\rho^k:\Z\to\Z$ be the $k^{\mr{th}}$-iterated composition of $\rho$ with itself. Since $\rho$ is non-decreasing and $r\geq2^n\mb{w}_n(r)+2^n-1$, we have that $\rho^n$ is non-decreasing and
\begin{equation}\label{eq:rho-n}
\rho^n(r)\geq\mb{w}_n(r);
\end{equation}
indeed, it holds:
\begin{align*}
\rho^n(r)&\geq\rho^n(2^n\mb{w}_n(r)+2^n-1)=\rho^{n-1}(2^{n-1}\mb{w}_n(r)+2^{n-1}-1)=\ldots=\rho(2\mb{w}_n+1)=\mb{w}_n(r).
\end{align*} 
Consider the component $F_K:D\to A$ of $F$ for some fixed $K\in\pa(n)$. If $z=(z_1,\ldots,z_n)=(\alpha_1+i\beta_1,\ldots,\alpha_n+i\beta_n)$ are the coordinates of $\C^n$, then $F_K$ is even w.r.t.\ $z_h$ if $h\not\in K$ and it is odd w.r.t.\ $z_h$ if $h\in K$. By \eqref{eq:rho-n}, $\rho^n(r)$ is non-negative, because $\rho^n(r)\geq\mb{w}_n(r)\geq0$. Since $F_K$ is of class $\mscr{C}^r$ and $\rho^n(r)\geq0$, we can apply to $F_K$ the representation results of Whitney for even-odd function along each variables $z_1,\ldots,z_n$, see \cite{Whitney1943} especially Remark at page 160. 
In this way, if $W_n:\C^n\to\C^n$ is the $\mscr{C}^\omega$ map given by $W_n(\alpha_1+i\beta_1,\ldots,\alpha_n+i\beta_n):=(\alpha_1+i\beta_1^2,\ldots,\alpha_n+i\beta_n^2)$, then we obtain an open neighborhood $U$ of $W_n(D)$ in $\C^n$ and a $\mscr{C}^{\rho^n(r)}$ map $F'_K:U\to A$ such that $F_K(z)=\beta_KF'_K(W_n(z))$ for all $z\in D$, where $\beta_\emptyset:=1$. Using \eqref{eq:rho-n} again, we know that each function $F'_k$ is also of class $\mscr{C}^{\mb{w}_n(r)}$. Note that $w_n(\OO_D)\subset W_n(D)\subset U$. Consequently, the set $V:=(w_n)^{-1}(U)$ is an open neighborhood of $\OO_D$ in $A^n$. Define the function $\hat{f}:V\to A$ by setting
\[
\textstyle
\hat{f}(x):=\sum_{K\in\pa(n)}[\IM_K(x),F'_K(w_n(x))].
\]
The function $\hat{f}$ is of class $\mscr{C}^{\mb{w}_n(r)}$ and extends $f$ to the whole $V$. As a consequence, $f$ belongs to $\mscr{C}^{\mb{w}_n(r)}$, as desired. The proof in the case $r\in\{\infty,\omega\}$ is similar, but easier because the $F'_K$'s have the same $\mscr{C}^r$ regularity of $F_K$.
\end{proof}

\begin{remark}
(i) In the statement of point (i) of the preceding result, we cannot omit the compactness condition on $\cS_A$, also in the one variable case. In Proposition 7(1) of \cite{AIM2011} we forgot to add the compactness hypothesis. Let $A$ be the Clifford algebra $\mi{C}\ell_{1,1}=\cS\HH$ of split-quaternions, equipped with the Clifford conjugation (see Sections 3.2.1 and 3.2.2 of \cite{GHS}). Given any element  $x=x_0+x_1e_1+x_2e_2+x_{12}e_{12}$ of $\cS\HH$ with $x_0,x_1,x_2,x_{12}\in\R$, we have that $t(x)=2x_0$ and $n(x)=x_0^2-x_1^2+x_2^2-x_{12}^2$. It follows that $\cS_A$ is the $2$-hyperboloid of $A\simeq\R^4$ given by the equations $x_0=0=x_2^2-x_1^2-x_{12}^2-1$ and $Q_A$ is the union of $\R$ and the open cone $x_2^2-x_1^2-x_{12}^2>0$. Note that $\cS_A$ is not compact. Consider the continuous stem function $F:\C\to A\otimes\C$ defined by $F(\alpha+i\beta):=i|\beta|^{\frac{1}{2}}\mr{sgn}(\beta)$, where $\mr{sgn}(\beta)$ is equal to $1$ if $\beta>0$, $-1$ if $\beta<0$ and $0$ if $\beta=0$. If $f:Q_A\to A$ is the one variable slice function induced by $F$, then $f(x)=0$ for all $x\in\R$, and $f(x)=(x_2^2-x_1^2-x_{12}^2)^{-\frac{1}{4}}(x_1e_1+x_2e_2+x_{12}e_{12})$ for all $x\in Q_A\setminus\R$. Let $\alpha\in\R$ and, for each $t>0$, let $y_t$ be the point of $Q_A\setminus\R$ defined by $y_t:=\alpha+te_1+(t+t^4)e_2$. Since $f(y_t)=t^{-\frac{1}{4}}(2+t^3)^{-\frac{1}{4}}(e_1+(1+t^3)e_2)$ for all $t>0$, we have that $\lim_{t\to 0^+}y_t=\alpha$ and $\lim_{t\to 0^+}t^{\frac{1}{4}}f(y_t)=2^{-\frac{1}{4}}(e_1+e_2)\neq0$. This proves that $f$ is not continuous at $\alpha$.

(ii) A by-product of the preceding proof is that, if $D$ is open in $\C^n$, then $\mc{S}^r(\OO_D,A)\subset\mscr{C}^{\rho^n(r)}(\OO_D,A)$ for all $r\in\N$ with $\rho^n(r)\geq0$.  

(iii) Thanks to the preceding proof, it is also quite evident that, if $\OO_D\cap\R_\bullet=\emptyset$, then $\mc{S}^r(\OO_D,A)\subset\mscr{C}^r(\OO_D,A)$ for all $r\in\N\cup\{\infty,\omega\}$. \bs
\end{remark}


\subsection{Multiplicative structures on slice functions and polynomials}\label{sec:m-s-slice}

Let us introduce the concept of symmetric difference algebra, or $\dsim$-algebra for short.  

\begin{definition}\label{def:symm-diff-prod}
Given a bilinear map $\mr{b}:\R^{2^n}\times\R^{2^n}\to\R^{2^n}$, we say that $\mr{b}$ is a \emph{symmetric difference product} on $\R^{2^n}$, or a \emph{$\dsim$-product} on $\R^{2^n}$ for short, if there exists a function $\sigma:\pa(n)\times\pa(n)\to\R$ such that
\begin{equation}\label{eq:empty}
\sigma(K,\emptyset)=\sigma(\emptyset,K)=1\, \text{ for all $K\in\pa(n)$}
\end{equation}
and
\begin{equation}\label{eq:K,H}
\mr{b}(e_K,e_H)=e_{K\dsim H}\sigma(K,H)\, \text{ for all $K,H\in\pa(n)$.}
\end{equation}
If this is the case, we say that the $\dsim$-product $\mr{b}$ is \emph{induced by $\sigma$}, and we write $\mr{b}=\EuScript{B}(\sigma)$. For simplicity, we use also the symbol $v\cdot_\sigma w$ in place of $\mr{b}(v,w)$ and, if there is no possibility of confusion, we omit `$\,\cdot_\sigma$' writing simply $vw$.

Let $P$ be a real vector space, equipped with a product $p:P\times P\to P$ making $P$ a real algebra. We say that the real algebra $(P,p)$ is a \emph{symmetric difference algebra}, or a \emph{$\dsim$-algebra} for short, if it is isomorphic to some $\R^{2^n}$ equipped with a $\dsim$-product. \bs
\end{definition}

Evidently, each $\dsim$-product is induced by a unique function $\sigma$, and each function $\sigma:\pa(n)\times\pa(n)\to\R$ satisfying \eqref{eq:empty} defines a $\dsim$-product on $\R^{2^n}$. By Assumption \ref{assumption1}, \eqref{eq:empty} and \eqref{eq:K,H}, we have that $e_\emptyset=1$ is the unity of $\R^{2^n}$, and $e_K^2=\sigma(K,K)\in\R$ for all $K\in\pa(n)$; consequently, a necessary condition for a $2^{n}$-dimensional real algebra to be a $\dsim$-algebra is that it has a unity $e$ and a vector basis $\{v_K\}_{K\in\pa(n)}$ such that, for each $K\in\pa(n)$, $v_K^2$ belongs to the vector subspace of $A$ generated by~$e$.

As we will see in the next remark, the notion of $\dsim$-product includes several important classical products on $\R^{2^n}$. Moreover, all the real algebras with unity of dimension $1$ and $2$ are $\dsim$-algebras. On the contrary, for each $n\geq 2$, there exist $2^n$-dimensional real algebras with unity which are not $\dsim$-algebras.

\begin{examples}\label{ex:symm-diff-prod}
$(1)$ Each real Clifford algebra $\mi{C}\ell(p,q)$ is a $\dsim$-algebra, including quaternions $\hh=\mi{C}\ell(0,2)$. When $n=3$, another example of associative $\dsim$-algebra is the one of dual quaternions, see \cite{AlgebraSliceFunctions} for the definition. The algebra $\oo$ of octonions and the algebra $\sq\oo$ of split-octonions are examples of non-associative $\dsim$-algebras, see \cite{AlgebraSliceFunctions}.

$(2)$ Up to isomorphism, the unique real algebra with unity of dimension $1$ is $\R$, which is a $\dsim$-algebra. All the real algebras with unity of dimension $2$ are $\dsim$-algebras as well. Suppose that $\R^2$ is equipped with a product such that $1$ is its neutral element and $e_1^2=\alpha+\beta e_1$ for some $\alpha,\beta\in\R$. Define $v:=\beta-2e_1$. Since $v^2=4\alpha+\beta^2$ belongs to $\R$ and $\{1,v\}$ is a vector basis of $\R^2$, it follows that $\R^2$ equipped with such a product is isomorphic to $\R^2$ equipped with the $\dsim$-product induced by the function $\sigma:\pa(1)\times\pa(1)\to\R$ such that $\sigma(\{1\},\{1\})=4\alpha+\beta^2$.

Let $n\geq2$. Consider the product on $\R^{2^n}$ such that $e_\emptyset=1$ is its neutral element, $e_Ke_H=0$ if $K,H\in\pa(n)\setminus\{\emptyset\}$ with $K\neq H$, and $e_K^2=-\frac{2}{2^n-1}\sum_{H\in\pa(n)\setminus\{\emptyset\}}e_H$ for all $K\in\pa(n)\setminus\{\emptyset\}$. If $v=\sum_{K\in\pa(n)}e_Ka_K$ is a generic element of $\R^{2^n}$ then
\[
\textstyle
v^2=a_\emptyset^2+\sum_{K\in\pa(n)\setminus\{\emptyset\}}e_K\big(2a_\emptyset a_K-\frac{2}{2^n-1}\sum_{H\in\pa(n)\setminus\{\emptyset\}}a_H^2\big).
\]
By simple computations, we see that $v^2\in\R$ if and only if either $v\in\R$ or $v=\lambda\sum_{K\in\pa(n)}e_K$ for some $\lambda\in\R$. It follows that there exist at most two linearly independent vectors of $\R^{2^n}$ whose squares are real. Consequently, $\R^{2^n}$ equipped with the mentioned product is not a $\dsim$-algebra. \bs
\end{examples}

Other very interesting examples of $\dsim$-algebras can be constructed via tensor products.

\begin{examples}
Let $n.m\in\N^*$. Denote $\{e'_K\}_{K\in\pa(n)}$ the fixed real vector basis of $\R^{2^n}$, and $\{e''_H\}_{H\in\pa(m)}$ the fixed real vector basis of $\R^{2^m}$. Recall that $e'_\emptyset=1\in\R^{2^n}$ and $e''_\emptyset=1\in\R^{2^m}$. Given any $L\in\pa(n+m)$, define $L_m\in\pa(m)$ and $L_m^*\in\pa(n)$ by setting $L_m:=L\cap\{1,\ldots,m\}$ and $L_m^*:=\{l\in\N^*:l+m\in L\}$. Write the elements $x$ of $\R^{2^n}\otimes\R^{2^m}$ as follows:
\[
\textstyle
x=\sum_{H\in\pa(m)}e''_H(\sum_{K\in\pa(n)}e'_Kr_{H,K})=\sum_{H\in\pa(m),K\in\pa(n)}e''_He'_Kr_{H,K} 
\]
for $r_{H,K}\in\R$, where $e''_He'_K:=e'_K\otimes e''_H$. Identify $\R^{2^{n+m}}$ with $\R^{2^n}\otimes\R^{2^m}$ and define the real vector basis $\{e_L\}_{L\in\pa(n+m)}$ of $\R^{2^{n+m}}$ by $e_L:=e''_{L_m}e'_{L_m^*}$. In this way, we can write $x=\sum_{L\in\pa(n+m)}e_Lr_L$, where $r_L:=r_{L_m,L_m^*}$.

Choose a $\dsim$-product $\mr{b}=\EuScript{B}(\sigma)$ on $\R^{2^n}$ and a $\dsim$-product $\mr{c}=\EuScript{B}(\tau)$ on $\R^{2^m}$. Define the function $\sigma\otimes\tau:\pa(n+m)\times\pa(n+m)\to\R$ and the $\dsim$-product $\mr{b}\otimes\mr{c}$ on $\R^{2^{n+m}}$ as follows:
\begin{align}
(\sigma\otimes\tau)(L,M)&:=\sigma(L_m^*,M_m^*)\tau(L_m,M_m),\label{eq:LM}\\
\mr{b}\otimes\mr{c}&:=\EuScript{B}(\sigma\otimes\tau).\label{eq:bc}
\end{align}
We call $(\R^{2^{n+m}},\mr{b}\otimes\mr{c})$ tensor product of the $\dsim$-algebras $(\R^{2^n},\mr{b})$ and $(\R^{2^m},\mr{c})$. Note that, given $L,M\in\pa(n+m)$, it holds:
\begin{align*}
(e'_{L_m^*}\otimes e''_{L_m})\cdot_{\sigma\otimes\tau}(e'_{M_m^*}\otimes e''_{M_m})&=e_L\cdot_{\sigma\otimes\tau}e_M=e_{L\dsim M}\sigma(L_m^*,M_m^*)\tau(L_m,M_m)=\\
&=(e'_{(L\dsim M)^*_m}\otimes e''_{(L\dsim M)_m})\sigma(L_m^*,M_m^*)\tau(L_m,M_m)=\\
&=(e'_{L_m^*\dsim M_m^*}\otimes e''_{L_m\dsim M_m})\sigma(L_m^*,M_m^*)\tau(L_m,M_m)=\\
&=(e'_{L_m^*\dsim M_m^*}\sigma(L_m^*,M_m^*))\otimes(e''_{L_m\dsim M_m}\tau(L_m,M_m))=\\
&=(e'_{L^*_m}\cdot_{\sigma}e'_{M^*_m})\otimes(e''_{L_m}\cdot_{\tau}e''_{M_m}).\, \text{ \bs}
\end{align*}
\end{examples}

Note that the real algebra $\C$ of complex numbers coincides with $\R^2$ equipped with the $\dsim$-product $\eta:\pa(1)\times\pa(1)\to\R$ such that $\eta(\{1\},\{1\}):=-1$.

\begin{definition}\label{def:tp}
Given any $n\in\N^*$, we denote $\sigma_\otimes^n:\pa(n)\times\pa(n)\to\R$ the $n$-times iterated tensor product of $\eta:\pa(1)\times\pa(1)\to\R$ with itself, i.e. $\sigma_\otimes^1:=\eta$ and $\sigma_\otimes^n:=\sigma_\otimes^{n-1}\otimes\eta$ if $n\geq2$. We say that $\mr{b}_\otimes^n:=\EuScript{B}(\sigma_\otimes^n)$ is the \emph{tensor product} on $\C^{\otimes n}=\R^{2^n}$, and $\C^{\otimes n}$ equipped with $\mr{b}_\otimes^n$ is the \emph{$n^{\mr{th}}$-tensor power} of $\C$. \bs
\end{definition}

\begin{lemma}\label{lem:tensor}
For all $n\in\N^*$ and for all $K,H\in\pa(n)$, it holds $\sigma_\otimes^n(K,H)=(-1)^{|K\cap H|}$. In particular, the $n^{\mr{th}}$-tensor power $\C^{\otimes n}$ of $\C$ is commutative and associative.
\end{lemma}
\begin{proof}
Let us prove this assertion by induction on $n\in\N^*$. The case $n=1$ is evident, because $\sigma_\otimes^1=\eta$ and $\eta$ has the required property. Let $n\geq2$. By induction, there exists a real vector basis $\{e_{H'}\}_{H'\in\pa(n-1)}$ of $\R^{2^{n-1}}$ such that $\sigma_\otimes^{n-1}(K',H')=(-1)^{|K'\cap H'|}$ for all $K',H'\in\pa(n-1)$. By \eqref{eq:LM}, we have that $\sigma_\otimes^n(K,H)=\sigma_\otimes^{n-1}(K_1^*,H_1^*)\eta(K_1,L_1)=(-1)^{|K_1^*\cap H_1^*|+|K_1\cap L_1|}$. 
Since $K_1^*\cap H_1^*=(K\cap H)^*_1$ and $K_1\cap H_1=(K\cap H)_1$, we easily deduce that $|K_1^*\cap H_1^*|+|K_1\cap L_1|=|K\cap H|$, as desired.
\end{proof}

\begin{assumption}
Throughout the remaining part of this section, we equip $\R^{2^n}$ with a $\dsim$-product $\mr{b}=\EuScript{B}(\sigma)$, and the tensor product $A\otimes\R^{2^n}$ with the following product extending $\mr{b}$:
\begin{equation}\label{eq:extended-product}
\textstyle
\big(\sum_{H\in\pa(n)}e_Ha_H\big)\cdot_\sigma\big(\sum_{L\in\pa(n)}e_Lb_L\big):=\sum_{H,L\in\pa(n)}(e_H\cdot_\sigma e_L)(a_Hb_L),
\end{equation}
where $a_Hb_L$ is the product of $a_H$ and $b_L$ in $A$. For simplicity, for each $\xi,\eta\in A\otimes\R^{2^n}$, we also write $\xi\eta$ in place of $\xi\cdot_\sigma\eta$. \bs
\end{assumption}

Note that if, for each $K\in\pa(n)$, $\mscr{D}(K)$ denotes the set
\[
\mscr{D}(K):=\big\{(K_1,K_2,K_3) \in \pa(n)^3 \,:\, K_1 \cap K_2=\emptyset, K_1 \cup K_2=K, K \cap K_3=\emptyset\big\},
\]
then $\big(\sum_{H\in\pa(n)}e_Ha_H\big)\big(\sum_{L\in\pa(n)}e_Lb_L\big)=\sum_{K\in\pa(n)}e_Kc_K$, where
\begin{equation}\label{eq:K}
\textstyle
c_K=\sum_{(K_1,K_2,K_3) \in \mscr{D}(K)}a_{K_1 \cup K_3}b_{K_2 \cup K_3}\sigma(K_1 \cup K_3,K_2 \cup K_3).
\end{equation}
Indeed, we have:
\begin{align*}
&\textstyle\sum_{H,L \in \pa(n)}(e_He_L)(a_Hb_L)=\sum_{H,L \in \pa(n)}e_{H \dsim L}a_Hb_L\sigma(H,L)=\\
&\textstyle=\sum_{K \in \pa(n)}e_K\sum_{(K_1,K_2,K_3) \in \mscr{D}(K)}a_{K_1 \cup K_3}b_{K_2 \cup K_3}\sigma(K_1 \cup K_3,K_2 \cup K_3).
\end{align*}

In general the pointwise product of two slice functions is not a slice function. For instance, if $n=1$ and $f,g:\hh\to\hh$ are the slice functions defined by $f\equiv i$ and $g(x)=x$, then $(fg)(x)=ix$ is not slice. Otherwise, by Proposition \ref{prop:polynomials} and Corollary \ref{cor:ip},   it would follow that $(fg)(x)=xi$ for all $x\in\hh$, which is impossible being $(fg)(j)=ji\neq ij$. On the contrary, the pointwise product of two stem functions is still a stem function.

\begin{lemma} \label{lem:stem-product}
Let $F,G:D\to A \otimes\R^{2^n}$ be stem functions and let $F\cdot_\sigma G:D \to A \otimes\R^{2^n}$ be the pointwise product of $F$ and $G$ w.r.t.\ $\mr{b}=\EuScript{B}(\sigma)$, that is $(F\cdot_\sigma G)(z):=F(z)\cdot_\sigma G(z)$ for all $z\in D$. Then $F\cdot_\sigma G$ is still a stem function.
\end{lemma}
\begin{proof}
Write $FG$ in place of $F\cdot_\sigma G$, for short. By \eqref{eq:K}, if $F=\sum_{H \in \pa(n)}e_HF_H$, $G=\sum_{L \in \pa(n)}e_LG_L$ and $FG=\sum_{K\in\pa(n)}e_K(FG)_K$, then 
\[
\textstyle
(FG)_K=\sum_{(K_1,K_2,K_3) \in \mscr{D}(K)}F_{K_1 \cup K_3}G_{K_2 \cup K_3}\sigma(K_1 \cup K_3,K_2 \cup K_3).
\]

Choose $K \in \pa(n)$, $h \in \{1,\ldots,n\}$ and $z \in D$. Note that, for each $(K_1,K_2,K_3) \in \mscr{D}(K)$, the integers $|(K_1 \cup K_3) \cap \{h\}|+|(K_2 \cup K_3) \cap \{h\}|$ and $|K \cap \{h\}|$ have the same parity. Consequently, we have
\begin{align*}
(FG)_K(\overline{z}^h)&=\textstyle\sum_{(K_1,K_2,K_3) \in \mscr{D}(K)}F_{K_1 \cup K_3}(\overline{z}^h) \, G_{K_2 \cup K_3}(\overline{z}^h) \, \sigma(K_1 \cup K_3,K_2 \cup K_3)=\\
&=\textstyle (-1)^{|K \cap \{h\}|}\sum_{(K_1,K_2,K_3) \in \mscr{D}(K)}F_{K_1 \cup K_3}(z) \, G_{K_2 \cup K_3}(z) \, \sigma(K_1 \cup K_3,K_2 \cup K_3)=\\
&=\textstyle(-1)^{|K \cap \{h\}|}(FG)_K(z),
\end{align*}
so $FG$ is a stem function, as desired.
\end{proof}

Thanks to the latter lemma, we can define a product on the class of slice functions.

\begin{definition} \label{def:slice-product}
Let $f,g:\OO_D \to A$ be slice functions with $f=\I(F)$ and $g=\I(G)$. We define the \emph{slice product $f\cdot_\sigma g:\OO_D\to A$} of $f$ and $g$ by $f\cdot_\sigma g:=\I(F\cdot_\sigma G)$. Moreover, we say that the slice product $f\cdot g=f\cdot_\sigma g$ is \emph{induced by $\mr{b}$}, or \emph{by $\sigma$}. If there is no possibility of confusion, we simply write $FG$ and $f\cdot g$ in place of $F\cdot_\sigma G$ and $f\cdot_\sigma g$, respectively. \bs
\end{definition}

We specialize the preceding definition as follows.

\begin{definition}\label{def:stp}
We call \emph{slice tensor product} on $\mc{S}(\OO_D,A)$ the product on $\mc{S}(\OO_D,A)$ induced by the tensor product $\mr{b}_\otimes^n=\EuScript{B}(\sigma_\otimes^n)$. 
Given $f,g\in\mc{S}(\OO_D,A)$, we say that $f\cdot_{\sigma_\otimes^n} g$ is the \emph{slice tensor product} of $f$ and $g$. \bs
\end{definition}

\begin{assumption}
In what follows, we use the symbol `$\,\tenso$' to denote `$\,\cdot_{\sigma_\otimes^n}\,$'.
\end{assumption}

Corollary \ref{cor:I} imply at once the following fact. 

\begin{corollary}
The pairs $(\mr{Stem}(D,A\otimes\R^{2^n}),\cdot_\sigma)$ and $(\mc{S}(\OO_D,A),\cdot_\sigma)$ are real algebras, and $\I$ is a real algebra isomorphism between them.
\end{corollary}

Let us introduce the concepts of slice polynomial functions associated with $\mr{b}=\EuScript{B}(\sigma)$, and of hypercomplex $\dsim$-product.

\begin{definition}\label{def:slice-polynomials}
Let $\mr{b}=\EuScript{B}(\sigma)$ be a $\triangle$-product on $\R^{2^n}$. Given any $k\in\{1,\ldots,n\}$ and $m\in\N$, we define the slice function $x_k^{{\sss\bullet} m}:(Q_A)^n\to A$ as the function constantly equal to $1$ if $m=0$, as the $k^{\text{th}}$-coordinate function $x_k:(Q_A)^n\to A$ if $m=1$ and as the $m$-times iterated slice product of $x_k:(Q_A)^n\to A$ with itself w.r.t.\ $\mr{b}$ if $m\geq2$, i.e. $x_k^{{\sss\bullet} 0}:\equiv1$ and $x_k^{{\sss\bullet} m}:=x_k^{{\sss\bullet} m-1}\cdot_\sigma x_k$ if $m\geq1$. We say that a function $P:(Q_A)^n\to A$ is \emph{slice monomial w.r.t.\ $\mr{b}$}, or \emph{w.r.t.\ $\sigma$}, if there exist $\ell=(\ell_1,\ldots,\ell_n)\in\N$ and $a\in A$ such that
\[
P=x_1^{{\sss\bullet}\ell_1}\cdot_\sigma (x_2^{{\sss\bullet}\ell_2}\cdots (x_{m-1}^{{\sss\bullet}\ell_{m-1}}\cdot_{\sigma} (x_m^{{\sss\bullet}\ell_m} \cdot_{\sigma} a)) \ldots),
\]
where $a\in A$ is identified with the slice function from $(Q_A)^n$ to $A$ constantly equal to $a$. If $P$ has this form and there is no possibility of confusion, then we denote $P$ as $x^{{\sss\bullet}\ell}\cdot a$. We call $P:(Q_A)^n\to A$ \emph{slice polynomial function w.r.t.\ $\mr{b}$}, or \emph{w.r.t.\ $\sigma$} if it is the finite sum of slice monomial functions w.r.t.\ $\mr{b}$, or w.r.t.\ $\mr{\sigma}$.

The restriction of a slice monomial (respectively polynomial) function w.r.t.\ $\mr{b}$, or w.r.t.~$\mr{\sigma}$, on $\OO_D$ is said to be \emph{slice monomial} (respectively \emph{polynomial}) \emph{on $\OO_D$ w.r.t.\ $\mr{b}$}, or \emph{w.r.t.\ $\mr{\sigma}$.} \bs
\end{definition} 

\begin{definition}
We say that the $\dsim$-product $\mr{b}=\EuScript{B}(\sigma)$ on $\R^{2^n}$ is 
\emph{hypercomplex}
 if it satisfies the following two conditions:
\begin{equation}\label{eq:units}
e_k^2=-1
\end{equation}
for all $k\in\{1,\ldots,n\}$, and
\begin{equation}\label{eq:eK}
e_K=e_{k_1}(e_{k_2}\cdots (e_{k_{s-1}}e_{k_s})\ldots)
\end{equation}
for all $K\in\pa(n)\setminus\{\emptyset\}$ with $K=\{k_1,\ldots,k_s\}$ and $k_1<\ldots<k_s$, which are equivalent to $\sigma(\{k\},\{k\})=-1$ and $\sigma(\{k_1\},\{k_2,\ldots,k_s\})\cdots\sigma(\{k_{s-2}\},\{k_{s-1},k_s\})\sigma(\{k_{s-1}\},\{k_s\})=1$, respe\-ctively. We say that a real algebra is a \emph{hypercomplex $\dsim$-algebra} if it is isomorphic to some $\R^{2^n}$ equipped with a hypercomplex $\dsim$-product. \bs
\end{definition}

\begin{examples}\label{ex:symm-diff-prod2}
The real algebras $\C^{\otimes n}$, $\R_q=\mi{C}\ell(0,q)$ and all their finite tensor products are hypercomplex $\dsim$-algebras. The real algebra $\C=\C^{\otimes 1}=\R_1$ of complex numbers is the unique hypercomplex $\dsim$-algebra of dimension $2$. The Clifford algebra $\mi{C}\ell(1,0)$ is an example of $\dsim$-algebra, which is not hypercomplex; indeed it has no imaginary units. A natural question is to understand whether the $\dsim$-algebra $\mi{C}\ell(p,q)$ is hypercomplex when $p\geq1$. This problem seems to be not so easy to settle. The reader bears in mind the several relations existing between Clifford algebras; for instance, $\mi{C}\ell(4,q-4)$ and $\R_q$ are isomorphic if $q\geq4$, see \cite[\S16.4]{Lounesto}. \bs 
\end{examples}

\begin{lemma}\label{lem:tensor-uniqueness-1}
The unique hypercomplex, commutative and associative $\dsim$-product on $\R^{2^n}$ is the tensor product $\mr{b}_\otimes^n$. 
\end{lemma}
\begin{proof}
By Lemma \ref{lem:tensor}, we know that $\mr{b}_\otimes^n$ is hypercomplex, commutative and associative. Suppose $\mr{b}=\EuScript{B}(\sigma)$ is a $\dsim$-product on $\R^{2^n}$ hypercomplex, commutative and associative. By \eqref{eq:eK}, the commutativity and the associativity, we have that $e_K\cdot_\sigma e_H=e_{K\dsim H}\cdot_\sigma e_{\ell_1}^2\cdot_\sigma \cdots\cdot_\sigma e_{\ell_q}^2$, where $\ell_1,\ldots,\ell_q$ are the elements of $K\cap H$ if $K\cap H\neq\emptyset$, and `$e_{\ell_1}^2\cdot_\sigma \cdots\cdot_\sigma e_{\ell_q}^2$' is omitted if $K\cap H=\emptyset$. Using \eqref{eq:units} and Lemma \ref{lem:tensor} again, we deduce that $e_K\cdot_\sigma e_H=e_{K\dsim H}(-1)^{|K\cap H|}$ and $\mr{b}=\mr{b}_\otimes^n$.
\end{proof}

The next result describes the `algebraic relevance' of hypercomplex $\dsim$-algebras in the context of slice functions. It asserts that, if the $\dsim$-product $\mr{b}=\EuScript{B}(\sigma)$ is hypercomplex, then the notions of polynomial function, pointwise defined in Definition \ref{def:polynomials}, and of slice polynomial function defined in Definition \ref{def:slice-polynomials} coincide.

\begin{lemma}\label{lem:polynomials}
If the $\dsim$-product $\mr{b}$ on $\R^{2^n}$ is hypercomplex (for instance, $\mr{b}=\mr{b}_\otimes^n$), then a function $f:\OO_D\to A$ is polynomial if and only if it is slice polynomial w.r.t.\ $\mr{b}$. More precisely, if $f(x)=\sum_{\ell\in L}x^\ell a_\ell$ for some finite subset $L$ of $\N^n$ and $a_\ell\in A$, then $f=\sum_{\ell\in L}x^{{\sss\bullet}\ell}\cdot a_\ell$ on $\OO_D$.
\end{lemma}
\begin{proof}
Let $\ell=(\ell_1,\ldots,\ell_n)\in\N^n$ and let $F^{(\ell)}:D\to A\otimes\R^{2^n}$ be the stem function inducing $x^{{\sss\bullet}\ell}\cdot a_\ell$. Given $h\in\{1,\ldots,n\}$, denote
 $G^{(h)}:D\to A\otimes\R^{2^n}$ the stem function inducing $x_h^{\bullet \ell_h}:\OO_D\to A$. Let $z=(\alpha_1+i\beta_1,\ldots,\alpha_n+i\beta_n)\in D$, and let $p_{\ell_h}$ and $q_{\ell_h}$ be  the real polynomials defined in Definition~\ref{def:pkqk}. Since $\mr{b}$ satisfies \eqref{eq:units} and \eqref{eq:eK}, we have that $G^{(h)}(z)=p_{\ell_h}(\alpha_h,\beta_h)+e_hq_{\ell_h}(\alpha_h,\beta_h)$ and 
\begin{align*}
F^{(\ell)}(z)&\textstyle=[(G^{(h)}(z))_{h=1}^n,a_\ell]=[(p_{\ell_h}(\alpha_h,\beta_h)+e_hq_{\ell_h}(\alpha_h,\beta_h))_{h=1}^n,a_\ell]=\\
&\textstyle=\sum_{K\in\pa(n)}\big(\prod_{h\in\{1,\ldots,n\}\setminus K}p_{\ell_h}(\alpha_h,\beta_h)\big)[(e_hq_{\ell_h}(\alpha_h,\beta_h))_{h\in K},a_\ell]=\\
&\textstyle=\sum_{K\in\pa(n)}e_K\big(\big(\prod_{h\in\{1,\ldots,n\}\setminus K}p_{\ell_h}(\alpha_h,\beta_h)\big)\big(\prod_{h\in K}q_{\ell_h}(\alpha_h,\beta_h)\big)a_\ell\big).
\end{align*}
By Proposition \ref{prop:polynomials}, it follows that $x^{{\sss\bullet}\ell}\cdot a_\ell=\I(F^{(\ell)})=x^\ell a_\ell$. Consequently, $f=\sum_{\ell\in L}x^{{\sss\bullet}\ell}\cdot a_\ell$.
\end{proof}


\subsection{$\C_J$-slice preserving, slice preserving and circular functions}\label{sec:C_Jslicepreserving}

\begin{definition}\label{def:OmegaDI}
Given any function $f:\OO_D \to A$ and 
 $J \in \cS_A$, we denote $\OO_D(J)$ the intersection $\OO_D \cap (\C_J)^n$, and $f_J:\OO_D(J)\to A$ the restriction of $f$ on $\OO_D(J)$. \bs
\end{definition}

\begin{definition}
Let $f:\OO_D\to A$ be a function. Given $J\in\cS_A$, we say that $f$ is a \emph{$\C_J$-slice preserving function} if $f$ is a slice function and $f(\OO_D(J))\subset\C_J$. We denote $\mc{S}_{\C_J}(\OO_D,A)$ the subset of $\mc{S}(\OO_D,A)$ of all $\C_J$-slice preserving functions from $\OO_D$ to $A$.

We say that $f$ is \emph{slice preserving} if it is a slice function and the stem function $F=\sum_{K\in\pa(n)}e_KF_K$ inducing $f$ has the following property: each component $F_K$ of $F$ is real va\-lued, that is $F_K(D)\subset\R$ for all $K\in\pa(n)$. We denote $\mc{S}_\R(\OO_D,A)$ the subset of $\mc{S}(\OO_D,A)$ of all slice preserving functions from $\OO_D$ to $A$. \bs
\end{definition}

\begin{lemma}\label{lem:slice-preserving}
Let $F:\OO_D\to A\otimes\R^{2^n}$ be a stem function with $F=\sum_{K\in\pa(n)}e_KF_K$ and let $f:=\I(F):\OO_D\to A$ be the corresponding slice function. The following assertions hold.
\begin{itemize}
 \item[$(\mr{i})$] Given $J	\in\cS_A$, $f$ belongs to $\mc{S}_{\C_J}(\OO_D,A)$ if and only if $F_K(D)\subset\C_J$ for all $K\in\pa(n)$.
 \item[$(\mr{ii})$] Suppose that there exist $I,J\in\cS_A$ such that $I\neq\pm J$. Then $f$ is slice preserving if and only if it is $\C_K$-slice preserving for $K\in\{I,J\}$, or equivalently $\mc{S}_\R(\OO_D,A)=\bigcap_{K\in\cS_A}\mc{S}_{\C_K}(\OO_D,A)$.
\end{itemize}
\end{lemma}
\begin{proof}
Since $\C_J$ is a real subalgebra of $A$, if $F_K(D)\subset\C_J$ for all $K\in\pa(n)$, then Definition \ref{def:slice-function} implies at once that $f(\OO_D(J))\subset\C_J$. Suppose now $f(\OO_D(J))\subset\C_J$ and apply formula \eqref{eq:components-F} to $f$ with $I_1=\ldots=I_n=J$. We obtain immediately that $F_K(D)\subset\C_J$ for all $K\in\pa(n)$. This proves $(\mr{i})$. Let us show $(\mr{ii})$. Recall that $\R\subset\C_J$ for all $J\in\cS_A$. As a consequence, preceding point $(\mr{i})$ implies that $\mc{S}_\R(\OO_D,A)\subset\bigcap_{K\in\cS_A}\mc{S}_{\C_K}(\OO_D,A)$. Finally, if $f\in\mc{S}_{\C_I}(\OO_D,A)\cap\mc{S}_{\C_J}(\OO_D,A)$, then using again above point $(\mr{i})$ we deduce that $F_K(D)\subset\C_I\cap\C_J$. Thanks to the Independence Lemma \cite[p. 224]{Numbers}, we know that $\C_I\cap\C_J=\R$, and we are done.
\end{proof}

Since each plane $\C_J$ is a real subalgebra of $A$, it follows at once:

\begin{lemma}\label{lem:subalgebras}
Let $\mr{b}=\EuScript{B}(\sigma)$ be any $\dsim$-product on $\R^{2^n}$ and let $J\in\cS_A$. The sets $\mc{S}_\R(\OO_D,A)$ and $\mc{S}_{\C_J}(\OO_D,A)$ are real subalgebras of $(\mc{S}(\OO_D,A),\cdot_\sigma)$. 
\end{lemma}

The next result concerns the relation between slice tensor and pointwise products. Given two functions $f,g:\OO_D\to A$, we indicate $fg:\OO_D\to A$ the pointwise product of $f$ and $g$, i.e. $(fg)(x):=f(x)g(x)$ for all $x\in\OO_D$, where $f(x)g(x)$ is the product of $f(x)$ and $g(x)$ in $A$. 

\begin{proposition}\label{prop:slice-pointwise-products}
Let $f,g\in\mc{S}_{\C_J}(\OO_D,A)$ for some $J\in\cS_A$, and let $a\in A$. Identify $a$ with the function from $\OO_D$ (or from $\OO_D(J)$) to $A$ constantly equal to $a$. The following holds:
\begin{itemize}
 \item[$(\mr{i})$] $(f\tenso(g\tenso a))(x)=f(x)g(x)a$ for all $x\in\OO_D(J)$. Equivalently, $(f\tenso(g\tenso a))_J=f_Jg_Ja$.
 \item[$(\mr{ii})$] $f\tenso(g\tenso a)=(f\tenso g)\tenso a\;$ on the whole $\OO_D$.
\end{itemize}
\end{proposition}
\begin{proof}
First note that $(\mr{i})$ implies $(\mr{ii})$. Indeed, applying $(\mr{i})$ twice (with $a=1$ and $g\equiv1$), Lemma \ref{lem:subalgebras} and Artin's theorem, one obtains $(f\tenso g)_J=f_Jg_J$ and $((f\tenso g)\tenso a)_J=(f\tenso g)_Ja=f_Jg_Ja=(f\tenso(g\tenso a))_J$. Corollary \ref{cor:ip} implies $(\mr{ii})$. Let us prove $(\mr{i})$. Let $F=\sum_{K\in\pa(n)}e_KF_K$ and $G=\sum_{H\in\pa(n)}e_HG_H$ be the stem functions inducing $f$ and~$g$, respectively. Let $C=\sum_{L\in\pa(n)}e_LC_L:D\to A\otimes\R^{2^n}$ be the stem function constantly equal to $a$, i.e. $C_\emptyset=a$ on $D$ and $C_L=0$ on $D$ for all $L\in\pa(n)\setminus\{\emptyset\}$. Evidently, $\I(C)=a$. Consider $z=(\alpha_1+i\beta_1,\ldots,\alpha_n+i\beta_n)\in D$ and $x=(\alpha_1+J\beta_1,\ldots,\alpha_n+J\beta_n)\in\OO_D(J)$. Let $\mr{J}:=(J,\ldots,J)\in(\cS_A)^n$. By Lemma \ref{lem:tensor}, we have that
\[
\textstyle
F\tenso(G\tenso C)=\sum_{K,H\in\pa(n)}e_{K \dsim H}(-1)^{|K\cap H|}F_K(G_Ha).
\]
Lemma \ref{lem:slice-preserving}$(\mr{i})$ implies that $F_K(z),G_H(z)\in\C_J$, so the elements $F_K(z)$, $G_H(z)$ and $J$ of $A$ commute and associate. Bearing in mind Artin's theorem, it follows that
\begin{align*}
(f\tenso(g\tenso a))(x)&\textstyle=\sum_{K,H \in \pa(n)}[\mr{J}_{K \dsim H},(-1)^{|K\cap H|}F_K(z)G_H(z)a]=\\
&\textstyle=\sum_{K,H \in \pa(n)}J^{|K \dsim H|}(-1)^{|K\cap H|}F_K(z)G_H(z)a=\\
&\textstyle=\sum_{K,H \in \pa(n)}J^{|K \dsim H|}J^{2|K\cap H|}F_K(z)G_H(z)a=\\
&\textstyle=\sum_{K,H \in \pa(n)}J^{|K \dsim H|+2|K\cap H|}F_K(z)G_H(z)a=\\
&\textstyle=\sum_{K,H \in \pa(n)}J^{|K|+|H|}F_K(z)G_H(z)a=\\
&\textstyle=(\sum_{K\in \pa(n)}J^{|K|}F_K(z))(\sum_{H\in \pa(n)}J^{|H|}G_H(z))a=f(x)g(x)a.
\end{align*}
The proof is complete.
\end{proof}

\begin{lemma}\label{lem:tensor-center}
Let $\mr{b}=\EuScript{B}(\sigma)$ be a commutative and associative $\triangle$-product on $\R^{2^n}$ (for instance $\mr{b}=\mr{b}^n_\otimes$). Then the set $\mc{S}_\R(\OO_D,A)$ is contained in the center of $(\mc{S}(\OO_D,A),\cdot_\sigma)$.
\end{lemma}
\begin{proof}
Let $F=\sum_{K\in\pa(n)}e_KF_K$, $G=\sum_{L\in\pa(n)}e_LG_L$ and $H=\sum_{M\in\pa(n)}e_MH_M$ be stem functions on $D$ such that each $F_K$ is real-valued. Bearing in mind the latter condition and the fact that `$\,\cdot_\sigma\,$' is commutative and associative, we obtain:
\begin{align*}
FG&=\textstyle\sum_{K,L\in\pa(n)}e_K\cdot_\sigma e_LF_KG_L=\sum_{L,K\in\pa(n)}e_L\cdot_\sigma e_KG_LF_K=GF,\\
(FG)H&=\textstyle\sum_{K,L,M\in\pa(n)}e_K\cdot_\sigma e_L\cdot_\sigma e_M(F_KG_L)H_M=\\
&=\textstyle\sum_{K,L,M\in\pa(n)}e_K\cdot_\sigma e_L\cdot_\sigma e_MF_K(G_LH_M)=F(GH).
\end{align*}
Similar considerations prove also that $(GF)H=G(FH)$ and $(GH)F=G(HF)$.
\end{proof}

In \cite[Remark 7]{AIM2011} we proved that, if $n=1$, $f\in\mc{S}_\R(\OO_D,A)$ and $g\in\mc{S}(\OO_D,A)$, then $f\tenso g=fg$ on the whole $\OO_D$. Remark \ref{rem:order} shows that in general the latter equality is false for $n\geq2$; indeed, the coordinate functions $x_2$ and $x_1$ belong to $\mc{S}_\R(\hh^n,\hh)$, but $x_2x_1$ is not slice on $\hh^n$.

Our next results give some generalizations to several variables of the mentioned result contained in \cite[Remark 7]{AIM2011}. First, we need a definition.

\begin{definition}
Let $F=\sum_{K\in\pa(n)}e_KF_K:D\to A\otimes\R^{2^n}$ be a stem function, let $f=\I(F):\OO_D\to A$ be the slice function induced by $F$ and let $H\in\pa(n)$. We say that $F$ is \emph{$H$-reduced} if $F_K=0$ on $D$ for all $K\in\pa(n)$ with $K\not\subset H$. If $F$ is $H$-reduced then we say also that $f$ is \emph{$H$-reduced}. If $H=\{h\}$ for some $h\in\{1,\ldots,n\}$, then we use the term \emph{$h$-reduced} meaning $\{h\}$-reduced. We say that $f$ is \emph{circular} if it is $\emptyset$-reduced, namely if $F$ is $A$-valued. Denote $\mc{S}_c(\OO_D,A)$ the subset of $\mc{S}(\OO_D,A)$ formed by all circular functions. \bs
\end{definition}

\begin{lemma}
Let $\mr{b}=\EuScript{B}(\sigma)$ be any $\dsim$-product on $\R^{2^n}$. Then, it holds:
\begin{itemize}
 \item[$(\mr{i})$] The set $\mc{S}_c(\OO_D,A)$ is a real subalgebra of $(\mc{S}(\OO_D,A),\cdot_\sigma)$. Furthermore, if $f,g\in\mc{S}_c(\OO_D,A)$, then $f\cdot_\sigma g=fg$ on $\OO_D$.
 \item[$(\mr{ii})$] If $f\in\mc{S}_\R(\OO_D,A)\cap\mc{S}_c(\OO_D,A)$ and $g\in\mc{S}(\OO_D,A)$, then $f\cdot_\sigma g=fg$ on $\OO_D$.
\end{itemize}
\end{lemma}
\begin{proof}
Let $F$ and $G$ be the stem functions inducing $f$ and $g$, respectively. Point $(\mr{i})$ follows immediately from the fact that, for all $z\in D$, $(F\cdot_\sigma G)(z)$ is equal to the product of  $F_\emptyset(z)$ and $G_\emptyset(z)$ in $A$. Let us prove $(\mr{ii})$. In this case $F=F_\emptyset$ is real-valued, $G=\sum_{K\in\pa(n)}e_KG_K$ is generic and
\begin{align*}
(f\cdot_\sigma g)(x)&\textstyle=\sum_{K\in\pa(n)}[J_K,F_\emptyset(z)G_K(z)]=F_\emptyset(z)\sum_{K\in\pa(n)}[J_K,G_K(z)]=f(x)g(x),
\end{align*}
as desired.
\end{proof}

\begin{proposition}\label{prop:reduced-prod}
Let $\mr{b}=\EuScript{B}(\sigma)$ be an associative and hypercomplex $\triangle$-product on $\R^{2^n}$ (for instance $\mr{b}=\mr{b}^n_\otimes$), let $f\in\mc{S}_\R(\OO_D,A)$ and let $g\in\mc{S}(\OO_D,A)$. Suppose that there exist $\ell\in\{1,\ldots,n\}$ and $H\in\pa(n)$ such that $f$ is $\ell$-reduced, $g$ is $H$-reduced and $\ell\leq h$ for all $h\in H$. Then $f\cdot_\sigma g=fg$ on $\OO_D$.
\end{proposition}
\begin{proof}
Denote $F,G\in\mr{Stem}(D,A\otimes\R^{2^n})$ the stem functions inducing $f$ and $g$, respectively. Write $F=F_\emptyset+e_\ell F_\ell$ and $G=\sum_{K\in\pa(n),K\subset H}e_KG_K$. Here $F_\ell$ denotes $F_{\{\ell\}}$. Let $z=(\alpha_1+i\beta_1,\ldots,\alpha_n+i\beta_n)\in D$ and let $x=(\alpha_1+J_1\beta_1,\ldots,\alpha_n+J_n\beta_n)$ for some $J=(J_1,\ldots,J_n)\in(\cS_A)^n$. By hypothesis, we know that $F_\emptyset(z),F_\ell(z)\in\R$. Since $\ell\leq h$ for all $h\in H$, given any $K\in\pa(n)$ with $K\subset H$, we have that:
\begin{itemize}
 \item $e_\ell\cdot_\sigma e_K=e_{K\cup\{\ell\}}$ and $[J_{K\cup\{\ell\}},G_K(z)]=J_\ell[J_K,G_K(z)]$ if $\ell\not\in K$, 
 \item $e_\ell\cdot_\sigma e_K=-e_{K\setminus\{\ell\}}$ and, thanks to Artin's theorem, $[J_{K\setminus\{\ell\}},G_K(z)]=-J_\ell[J_K,G_K(z)]$ if $\ell\in K$.
\end{itemize}
In particular, it holds:
\begin{align*}
(f\cdot_\sigma g)(x)=\,&\textstyle\sum_{K\in\pa(n),K\subset H}[J_K,F_\emptyset(z)G_K(z)]+\sum_{K\in\pa(n),K\subset H,\ell\not\in K}[J_{K\cup\{\ell\}},F_\ell(z)G_K(z)]+\\
&\textstyle+\sum_{K\in\pa(n),K\subset H,\ell\in K}(-[J_{K\setminus\{\ell\}},F_\ell(z)G_K(z)])=\\
=\,&\textstyle F_\emptyset(z)\sum_{K\in\pa(n),K\subset H}[J_K,G_K(z)]+J_\ell F_\ell(z)\sum_{K\in\pa(n),K\subset H,\ell\not\in K}[J_K,G_K(z)]+\\
&\textstyle+J_\ell F_\ell(z)\sum_{K\in\pa(n),K\subset H,\ell\in K}[J_K,G_K(z)]=\\
=\,&\textstyle(F_\emptyset(z)+J_\ell F_\ell(z))(\sum_{K\in\pa(n),K\subset H}[J_K,G_K(z)])=f(x)g(x),
\end{align*}
as desired.
\end{proof}


\section{Slice regular functions}\label{sec:Slice regular functions}


\begin{assumption}\label{assumption:openess}
Throughout this section, we assume that $D$ is open in $\C^n$.
\end{assumption}


\subsection{Complex structures on $A \otimes \R^{2^n}$}

Let $F=\sum_{K\in\pa(n)}e_KF_K:D\to A\otimes\R^{2^n}$ be a $\mscr{C}^1$ function, i.e. each $F_K:D\to A$ is of class $\mscr{C}^1$ in the usual real sense. Given a family $\J=\{\J_h\}_{h=1}^n$ consisting of $n$ complex structures on $A\otimes\R^{2^n}$, we say that a function $F$ is \emph{holomorphic w.r.t.\ $\J$} if, for all $z\in D$, it holds:
\begin{equation}\label{eq:J}
\frac{\partial F}{\partial\alpha_h}(z)+\J_h\left(\frac{\partial F}{\partial\beta_h}(z)\right)=0\; \text{ for all $h\in\{1,\ldots,n\}$,}
\end{equation}
where $z=(\alpha_1+i\beta_1,\ldots,\alpha_n+i\beta_n)$ are the coordinates of $\C^n$, $\frac{\partial F}{\partial\alpha_h}=\sum_{K\in\pa(n)}e_K\frac{\partial F_K}{\partial\alpha_h}$ and $\frac{\partial F}{\partial\beta_h}=\sum_{K\in\pa(n)}e_K\frac{\partial F_K}{\partial\beta_h}$. For short, in what follows, we will often denote $\partial_{\alpha_h}$ and $\partial_{\beta_h}$ the partial derivatives $\frac{\partial}{\partial\alpha_h}$ and $\frac{\partial}{\partial\beta_h}$, respectively.

We are interested in finding all the families $\J=\{\J_h\}_{h=1}^n$ having the following two universal/algebraic properties:
\begin{property}\label{property:universal}
Each complex structure $\J_h$ of $A\otimes\R^{2^n}$ is the extension of a complex structure of $\R^{2^n}$, we call again $\J_h$, via tensor product in the sense that $\J_h(a\otimes x)=a \otimes \J_h(x)\,$ for all $a\in A$ and $x\in\R^{2^n}$. Equivalently, given any $K\in\pa(n)$ and $a\in A$, if $\J_h(e_K)=\sum_{H\in\pa(n)}e_Hj^{\sss(h,K)}_H\in\R^{2^n}$, then $\J_h(e_Ka)=\sum_{H\in\pa(n)}e_H(j^{\sss(h,K)}_Ha)\in A\otimes\R^{2^n}$. 
\end{property}
\begin{property}\label{property:algebraic}
All polynomial stem functions $F:\C^n\to A\otimes\R^{2^n}$ are holomorphic w.r.t.\ $\J$.
\end{property}
Note that, if we apply the latter property to the polynomial stem functions $\{(\alpha_1+i\beta_1,\ldots,\alpha_n+i\beta_n)\mapsto \alpha_h+e_h\beta_h\}_{h=1}^n$, the ones inducing the coordinate monomials $x_1,\ldots,x_n$, then we deduce that $1+\J_h(e_h)=0$ for all $h\in\{1,\ldots,n\}$, which is equivalent to say that $\J_h(1)=e_h$ or $\J_h(e_h)=-1$.

Our next result shows that there exists a unique family $\J$ with the mentioned two properties.

\begin{proposition}\label{prop:complex-structures}
There exists a unique family $\{\J_h:A\otimes\R^{2^n}\to A\otimes\R^{2^n}\}_{h=1}^n$ of complex structures satisfying Properties \ref{property:universal} and \ref{property:algebraic}. Each endomorphism $\J_h$ is characterized by Property \ref{property:universal} and the following condition: for each $K\in\pa(n)$, it holds
\begin{equation}\label{eq:c-s1}
\J_h(e_K)=(-1)^{|K \cap \{h\}|}e_{K \dsim \{h\}}
\end{equation}
or, equivalently,
\begin{equation}\label{eq:c-s2}
\J_h(e_K)=
\left\{
\begin{array}{ll}
-e_{K \setminus \{h\}} & \text{\quad if }h\in K\vspace{.3em}\\
e_{K \cup \{h\}} & \text{\quad if }h\not\in K
\end{array}
\right..
\end{equation}
\end{proposition}
\begin{proof}
Let us prove \eqref{eq:c-s2} by induction on the cardinality $|K|$ of $K\in\pa(n)$. If $|K|=0$, then $K=\emptyset$ and we just know that $\J_h(e_\emptyset)=\J_h(1)=e_h$ for all $h\in\{1,\ldots,n\}$, so \eqref{eq:c-s2} is verified.

Suppose $|K|\geq1$. Let $K=\{k_1,\ldots,k_s\}$ with $k_1<\ldots<k_s$ and let $F:\C^n\to A\otimes\R^{2^n}$ be the monomial stem function inducing the monomial function $x_K:=[(x_k)_{k\in K}]$. We have
\[\textstyle
F(z)=\sum_{L\in\pa(K)}e_L(\alpha_{K\setminus L}\beta_L),
\]
for all $z=(\alpha_1+i\beta_1,\ldots,\alpha_n+i\beta_n)\in\C^n$, where $\pa(K)$ denotes the set $\{L\in\pa(n):L\subset K\}$, $\alpha_{K\setminus L}:=\prd_{h\in K\setminus L}\alpha_h$ and $\beta_L:=\prd_{h\in L}\beta_h$, where $\alpha_\emptyset=\beta_\emptyset:=1$. Choose $h\in K$. It holds
\begin{align*}
&\textstyle\partial_{\alpha_h}F=\sum_{L\in\pa(K),h\not\in L}e_L(\alpha_{K\setminus (L\cup\{h\})}\beta_L);
\end{align*}
moreover, thanks to Property \ref{property:universal} and the induction hypothesis, we have
\begin{align*}
\textstyle\J_h(\partial_{\beta_h}F)&\textstyle=\sum_{L\in\pa(K),h\in L}\J_h(e_L)(\alpha_{K\setminus L}\beta_{L\setminus\{h\}})=\\
&\textstyle=\J_h(e_K)\beta_{K\setminus\{h\}}+\sum_{L\in\pa(K)\setminus\{K\},h\in L}(-e_{L\setminus\{h\}})(\alpha_{K\setminus L}\beta_{L\setminus\{h\}})=\\
&\textstyle=\J_h(e_K)\beta_{K\setminus\{h\}}-\sum_{L\in\pa(K),h\not\in L,L\neq K\setminus\{h\}}e_L(\alpha_{K\setminus (L\cup\{h\})}\beta_L)=\\
&\textstyle=\J_h(e_K)\beta_{K\setminus\{h\}}+e_{K\setminus\{h\}}\beta_{K\setminus\{h\}}-\partial_{\alpha_h}F.
\end{align*}
By Property \ref{property:algebraic}, $F$ is holomorphic w.r.t.\ $\J$, so
\[\textstyle
0=\textstyle\partial_{\alpha_h}F+\J_h(\partial_{\beta_h}F)=\J_h(e_K)\beta_{K\setminus\{h\}}+e_{K\setminus\{h\}}\beta_{K\setminus\{h\}}.
\]
Consequently, $\J_h(e_K)=-e_{K\setminus\{h\}}$ as desired. If $h\not\in K$, then the preceding equality implies that $\J_h(e_{K\cup\{h\}})=-e_K$, so $\J_h(e_K)=\J_h(-\J_h(e_{K\cup\{h\}}))=e_{K\cup\{h\}}$. Equality  \eqref{eq:c-s2} is proven.

It remains to show that the complex structures $\J=\{\J_1,\ldots,\J_n\}$ on $A\otimes\R^{2^n}$ defined by Property \ref{property:universal} and equality \eqref{eq:c-s2} satisfies also Property \ref{property:algebraic}. Equivalently, we have to prove that, given any $\ell=(\ell_1,\ldots,\ell_n)\in\N^n$, the monomial stem function $F$ inducing $x^\ell$ is holomorphic w.r.t.~$\J$. Given $k\in\N$, let $p_k,q_k\in\R[X,Y]$ such that $(\alpha+i\beta)^k=p_k(\alpha,\beta)+iq_k(\alpha,\beta)$ for all $\alpha,\beta\in\R$, as in Definition \ref{def:pkqk}. By the Cauchy-Riemann equations, we have that $\partial_\alpha p_k=\partial_\beta q_k$ and $\partial_\beta p_k=-\partial_\alpha q_k$. For each $K=\{k_1,\ldots,k_s\}\in\pa(n)\setminus\{\emptyset\}$ with $k_1<\ldots<k_s$, define the functions $p_K,q_K:\C^n\to\R$ by $p_K(z):=\prd_{k\in K}p_{\ell_k}(\alpha_k,\beta_k)$ and $q_K(z):=\prd_{k\in K}q_{\ell_k}(\alpha_k,\beta_k)$, where $z=(\alpha_1+i\beta_1,\ldots,\alpha_n+i\beta_n)\in\C^n$. Define also $p_\emptyset=q_\emptyset:\C^n\to\R$ as the function constantly equal to $1$. Note that $F=\sum_{K\in\pa(n)}e_Kp_{\{1,\ldots,n\}\setminus K}q_K$. Consequently, if $h\in\{1,\ldots,n\}$, it holds:
\begin{align*}
\partial_{\alpha_h}F=&\textstyle\sum_{K\in\pa(n),h\in K}e_Kp_{\{1,\ldots,n\}\setminus K}(\partial_{\alpha_h}q_{\ell_h})q_{K\setminus\{h\}}+\\
&\textstyle+\sum_{K\in\pa(n),h\not\in K}e_K(\partial_{\alpha_h}p_{\ell_h})p_{\{1,\ldots,n\}\setminus(K\cup\{h\})}q_K=\\
=&\textstyle-\sum_{K\in\pa(n),h\in K}e_Kp_{\{1,\ldots,n\}\setminus K}(\partial_{\beta_h}p_{\ell_h})q_{K\setminus\{h\}}+\\
&\textstyle+\sum_{K\in\pa(n),h\not\in K}e_K(\partial_{\beta_h}q_{\ell_h})p_{\{1,\ldots,n\}\setminus(K\cup\{h\})}q_K=\\
=&\textstyle-\sum_{K\in\pa(n),h\not\in K}e_{K\cup\{h\}}p_{\{1,\ldots,n\}\setminus(K\cup\{h\})}(\partial_{\beta_h}p_{\ell_h})q_K+\\
&\textstyle+\sum_{K\in\pa(n),h\in K}e_{K\setminus\{h\}}(\partial_{\beta_h}q_{\ell_h})p_{\{1,\ldots,n\}\setminus K}q_{K\setminus\{h\}}
\end{align*}
and
\begin{align*}
\textstyle\J_h\left(\partial_{\beta_h}F\right)=&\textstyle\sum_{K\in\pa(n),h\in K}\J_h(e_K)p_{\{1,\ldots,n\}\setminus K}(\partial_{\beta_h}q_{\ell_h})q_{K\setminus\{h\}}+\\
&\textstyle+\sum_{K\in\pa(n),h\not\in K}\J_h(e_K)(\partial_{\beta_h}p_{\ell_h})p_{\{1,\ldots,n\}\setminus(K\cup\{h\})}q_K=\\
=&\textstyle-\sum_{K\in\pa(n),h\in K}e_{K\setminus\{h\}}p_{\{1,\ldots,n\}\setminus K}(\partial_{\beta_h}q_{\ell_h})q_{K\setminus\{h\}}+\\
&\textstyle+\sum_{K\in\pa(n),h\not\in K}e_{K\cup\{h\}}(\partial_{\beta_h}p_{\ell_h})p_{\{1,\ldots,n\}\setminus(K\cup\{h\})}q_K.
\end{align*}
Consequently, $\partial_{\alpha_h}F+\J_h\left(\partial_{\beta_h}F\right)=0$, as desired.
\end{proof}

The above complex structures $\J_1,\ldots,\J_n$ commute.

\begin{lemma}\label{prop:Jh}
The complex structures $\J_1,\ldots,\J_n$ on $\R^{2^n}$ defined in Proposition \ref{prop:complex-structures} commute, that is $\J_h\J_k=\J_k\J_h$ for all $h,k\in\{1,\ldots,n\}$.
\end{lemma}
\begin{proof}
Let $K \in \pa(n)$ and let $h,k \in \{1,\ldots,n\}$ with $h\neq k$. We have
\[
(\J_h\J_k)(e_K)=
\left\{
\begin{array}{ll}
-\J_h(e_{K \setminus \{k\}}) & \text{\quad if $k\in K$}
\vspace{.3em}\\
\J_h(e_{K \cup \{k\}}) & \text{\quad if $k\not\in K$.}
\end{array}
\right.
\]
Therefore
\[
(\J_h\J_k)(e_K)=
\left\{
\begin{array}{ll}
e_{K \setminus\{h,k\}} & \text{\quad if $h,k \in K$}
\vspace{.3em}\\
-e_{(K \cup \{h\}) \setminus \{k\}} & \text{\quad if $h \notin K$, $k \in K$}
\vspace{.3em}\\
-e_{(K \cup \{k\}) \setminus \{h\}} & \text{\quad if $h \in K$, $k \notin K$}
\vspace{.3em}\\
e_{K \cup \{h,k\}} & \text{\quad if $h \notin K$, $k \notin K$}
\end{array}
\right.
\]
is symmetric in $h$ and $k$.
\end{proof}

\begin{remark}
The mentioned complex structures $\J_1,\ldots,\J_n$ on $\R^{2^n}$ can also be characterized as the unique complex structures on $\R^{2^n}$ satisfying the following two conditions:
\begin{itemize}
 \item[$(\mr{i})$] $\J_h\J_k=\J_k\J_h$ for all $h,k\in\{1,\ldots,n\}$.
 \item[$(\mr{ii})$] $\J_h(e_K)=e_{K\cup\{h\}}$ for all $h \in \{1,\ldots,n\}$ and $K \in \pa(n)$ such that $h<k$ for all $k \in K$.
\end{itemize}
We have only to verify the uniqueness of $\J_1,\ldots,\J_n$. Let $\J_1^{\pr},\ldots,\J_n^{\pr}$ be complex structures on $\R^{2^n}$ satisfying conditions $(\mr{i})$ and $(\mr{ii})$. Applying $(\mr{ii})$ with $K=\emptyset$, we have that $\J_h(1)=e_h$ and hence $\J_h(e_h)=-1$ for all $h\in\{1,\ldots,n\}$. Let $K=\{k_1,\ldots,k_s\}\in\pa(n)\setminus\{\emptyset\}$ with $k_1<\ldots<k_s$, and let $h\in\{1,\ldots,n\}$. We have to prove that $\J_h^{\pr}(e_K)=-e_{K \setminus \{h\}}$ if $h \in K$ and $\J_h^{\pr}(e_K)=e_{K \cup \{h\}}$ if $h \not\in K$. First, suppose $h\not\in K$. If $h<k$ for all $k\in K$, then we are done by $(\mr{ii})$. If $h>k$ for some $k\in K$, then there exists a unique $t\in\{1,\ldots,s\}$ such that $k_t<h<k_{t+1}$, where $k_{s+1}:=n+1$. Define $K':=\{k_{t+1},\ldots,k_s\}$ if $t+1\leq s$, and $K':=\emptyset$ otherwise. By  $(\mr{i})$ and  $(\mr{ii})$, we have:
\[
\J_h^{\pr}(e_K)=(\J_h^{\pr}\J_{k_1}^{\pr}\cdots\J_{k_t}^{\pr})(e_{K^{\pr}})=(\J_{k_1}^{\pr}\cdots\J_{k_t}^{\pr})\big(\J_h^{\pr}(e_{K^{\pr}})\big)=(\J_{k_1}^{\pr}\cdots\J_{k_t}^{\pr})(e_{K'\cup\{h\}})=e_{K\cup\{h\}}.
\]
Finally, if $h\in K$, then $\J_h(e_{K\setminus\{h\}})=e_K$, so $\J_h(e_K)=\J_h(\J_h(e_{K\setminus\{h\}}))=-e_{K\setminus\{h\}}$. \bs
\end{remark}

\begin{assumption}\label{notation:J}
In what follows, we denote $\J=\{\J_1,\ldots,\J_n\}$ the family of complex stru\-ctures on $A\otimes\R^{2^n}$ defined in Proposition \ref{prop:complex-structures}.
\end{assumption}

\begin{definition} \label{def:partial_h,dibar_h}
Let $F:D \to A \otimes \R^{2^n}$ be a $\mscr{C}^1$ stem function. For each $h \in \{1,\ldots,n\}$, we denote $\partial_h$ and $\dibar_h$ the \emph{Cauchy-Riemann operators} w.r.t.\ the complex structures $i$ on $D$ and  $\J_h$ on $A \otimes \R^{2^n}$, that is
\[
\partial_hF=\frac{1}{2}\left(\dd{F}{\alpha_h}-\J_h\left(\dd{F}{\beta_h}\right)\right)
\;\;\text{ and }\;\;
\quad \dibar_hF=\frac{1}{2}\left(\dd{F}{\alpha_h}+\J_h\left(\dd{F}{\beta_h}\right)\right),
\]
where $(\alpha_1+i\beta_1,\ldots,\alpha_n+i\beta_n)$ are the coordinates of $D$. \bs
\end{definition}

As a consequence of Proposition \ref{prop:Jh}, each operator of the type $\partial_h$ or $\dibar_h$ commutes with each other:
\begin{equation} \label{eq:dd}
\partial_h\partial_k=\partial_k\partial_h,\quad \partial_h\dibar_k=\dibar_k\partial_h
\;\;\text{ and }\;\;
\dibar_h\dibar_k=\dibar_k\dibar_h
\end{equation}
for all $h,k \in \{1,\ldots,n\}$.

\begin{lemma} \label{lem:K}
Let $F:D \to A \otimes \R^{2^n}$ be a $\mscr{C}^1$ stem function and let $h \in \{1,\ldots,n\}$. For each $K \in \pa(n)$, denote $(\partial_hF)_K$ and $(\dibar_hF)_K$ the $K$-components of $\partial_hF$ and $\dibar_hF$, respectively. Then, for all $K \in \pa(n)$, it holds:
\begin{equation} \label{eq:partial_h}
(\partial_hF)_K=\frac{1}{2}\left(\dd{F_K}{\alpha_h}+\dd{F_{K \dsim \{h\}}}{\beta_h}(-1)^{|K \cap \{h\}|}\right)
\end{equation}
and
\begin{equation} \label{eq:dibar_h}
(\dibar_hF)_K=\frac{1}{2}\left(\dd{F_K}{\alpha_h}-\dd{F_{K \dsim \{h\}}}{\beta_h}(-1)^{|K \cap \{h\}|}\right).
\end{equation}
In particular, $\partial_hF:D \to A \otimes \R^{2^n}$ and $\dibar_hF:D \to A \otimes \R^{2^n}$ are stem functions.
\end{lemma}
\begin{proof} Let $K \in \pa(n)$. Equation \eqref{eq:partial_h} follows immediately from the following computation:
\begin{align*}
2 \, \partial_hF=&\textstyle\sum_{K \in \pa(n)} \big(e_K \partial_{\alpha_h}F_K-\J_h(e_K) \partial_{\beta_h}F_K\big)=\\
=&\textstyle\sum_{K\in\pa(n),h\in K} \big(e_K \partial_{\alpha_h}F_K+e_{K  \setminus\{h\}}\partial_{\beta_h}F_K\big)+\\
&\textstyle+\sum_{K\in\pa(n),h\not\in K} \big(e_K \partial_{\alpha_h}F_K-e_{K \cup\{h\}} \partial_{\beta_h}F_K\big)=\\
=&\textstyle\sum_{K\in\pa(n),h\in K} e_K\big(\partial_{\alpha_h}F_K- \partial_{\beta_h}F_{K \setminus\{h\}}\big)+
\sum_{K\in\pa(n),h\not\in K} e_K\big(\partial_{\alpha_h}F_K+ \partial_{\beta_h}F_{K\cup\{h\}}\big).
\end{align*}
Similarly, we have that
\[
\textstyle
2 \, \dibar_hF=\sum_{K\in\pa(n),h\in K} e_K\big(\partial_{\alpha_h}F_K+ \partial_{\beta_h}F_{K \setminus\{h\}}\big)+
\sum_{K\in\pa(n),h\not\in K} e_K\big(\partial_{\alpha_h}F_K- \partial_{\beta_h}F_{K\cup\{h\}}\big).
\]
Consequently, \eqref{eq:dibar_h} holds.

It remains to show that $\partial_hF$ and $\dibar_hF$ are stem functions. Let $j \in \{1,\ldots,n\}$. Since $F$ is a stem function, we know that $F_K(\overline{z}^j)=(-1)^{|K \cap \{j\}|}F_K(z)$ for all $z \in D$. Fix $z \in D$. Differentiating both members of the latter equality w.r.t.\ $\alpha_h$ and $\beta_h$, we obtain:
\begin{equation} \label{eq:j1}
\partial_{\alpha_h}F_K(\overline{z}^j)=(-1)^{|K \cap \{j\}|}\partial_{\alpha_h}F_K(z)
\end{equation}
and
\[
\partial_{\beta_h}F_K(\overline{z}^j)=(-1)^{|K \cap \{j\}|+|\{j\} \cap \{h\}|} \partial_{\beta_h}F_K(z).
\]
Since the integers $|(K \dsim \{h\}) \cap \{j\}|+|\{j\} \cap \{h\}|$ and $|K \cap \{j\}|$ have the same parity, the latter equality implies the next one:
\begin{equation} \label{eq:j2}
\partial_{\beta_h}F_{K \dsim \{h\}}(\overline{z}^j)=(-1)^{|K \cap \{j\}|} \partial_{\beta_h}F_{K\dsim\{h\}}(z).
\end{equation}
By combining (\ref{eq:partial_h}), (\ref{eq:j1}) and (\ref{eq:j2}), we obtain that $(\partial_hF)_K(\overline{z}^j)=(-1)^{|K \cap \{j\}|}(\partial_hF)_K(z)$ and hence $\partial_hF$ is a stem function. Similarly, by using (\ref{eq:dibar_h}) instead of (\ref{eq:partial_h}), we infer that $\dibar_hF$ is a stem function as well.
\end{proof}

The last part of Lemma \ref{lem:K} allows to give the following definition.

\begin{definition} \label{def:df}
Let $F:D \to A\otimes\R^{2^n}$ be a $\mscr{C}^1$ stem function, let $f=\I(F):\OO_D \to A$ be the corresponding slice function and let $h \in \{1,\ldots,n\}$. We define the \emph{slice partial derivatives} $\frac{\partial f}{\partial x_h}$ and $\frac{\partial f}{\partial x_h^c}$ of $f$ as the following slice function in $\mc{S}^0(\OO_D,A)$:
\begin{equation}
\dd{f}{x_h}:=\I(\partial_hF)
\;\;\text{ and }\;\;
\dd{f}{x_h^c}:=\I(\dibar_hF). \; \text{ \bs}
\end{equation}
\end{definition}


\subsection{Holomorphic stem functions, slice regular functions and polynomials}

Let us introduce the concepts of holomorphic stem and slice regular functions. The reader bears in mind Assumptions \ref{assumption:openess} and \ref{notation:J}, and Definition \ref{def:partial_h,dibar_h}.

\begin{definition}
Let $F:D \to A \otimes \R^{2^n}$ be a $\mscr{C}^1$ stem function. Given $h \in \{1,\ldots,n\}$, we say that $F$ is \emph{$h$-holomorphic} if $\dibar_hF=0$ on $D$. We call $F$ \emph{holomorphic} if it is $h$-holomorphic for all $h \in \{1,\ldots,n\}$. If $F$ is holomorphic, then we say that $\I(F):\OO_D \to A$ is a \emph{slice regular function}. We denote  $\mc{SR}(\OO_D,A)$ the real vector subspace of $\mc{S}(\OO_D,A)$ of all slice regular functions. \bs
\end{definition}

Note that $F$ is $h$-holomorphic if and only if, for each $z=(z_1,\ldots,z_n)\in D$, the restriction function $F^{\sss(z)}:(D_h(z),i)\to(A \otimes \R^{2^n},\J_h)$, sending $w$ into $F(z_1,\ldots,z_{h-1},w,z_{h+1},\ldots,z_n)$, is holomorphic in the usual sense. Here $D_h(z)$ is the non-empty open subset of $\C$ defined in \eqref{def:Dhz}.

A useful characterization of the $h$-holomorphy of a stem function is as follows.

\begin{lemma} \label{lem:h-holomorphy}
Let $F:D \to A \otimes \R^{2^n}$ be a $\mathscr C^1$ stem function with $F=\sum_{K \in \pa(n)}e_KF_K$ and let $h \in \{1,\ldots,n\}$.
The following assertions are equivalent:
\begin{itemize}
  \item[$(\mr{i})$] $F$ is $h$-holomorphic.
  \item[$(\mr{ii})$] For each $K \in \pa(n)$ with $h\not\in K$, the functions $F_K$ and $F_{K \cup \{h\}}$ satisfy the following Cauchy-Riemann equations
\begin{equation} \label{eq:K notni h}
\dd{F_K}{\alpha_h}=\dd{F_{K \cup \{h\}}}{\beta_h}
\;\; \text{ and } \;\;
\dd{F_K}{\beta_h}=-\dd{F_{K \cup \{h\}}}{\alpha_h}.
\end{equation}
\end{itemize}
\end{lemma}
\begin{proof} By (\ref{eq:dibar_h}), we know that
\[
\textstyle
2 \, \dibar_hF=\sum_{K\in\pa(n),h\not\in K} \big(e_{K \cup\{h\}} (\partial_{\alpha_h}F_{K \cup\{h\}}+\partial_{\beta_h}F_K)+
e_K(\partial_{\alpha_h}F_K-\partial_{\beta_h}F_{K \cup \{h\}})\big).
\]
As a consequence, $\dibar_hF=0$ if and only if \eqref{eq:K notni h} is satisfied.
\end{proof}

By Definition \ref{def:OmegaDI}, given any $J\in\cS_A$, we have that $\OO_D(J)=\OO_D\cap(\C_J)^n$ and $f_J:\OO_D(J)\to A$ is the restriction of $f$ on $\OO_D(J)$. By Assumptions \ref{assumption1} and \ref{assumption:openess}, it follows that $\OO_D(J)$ is a non-empty open subset of $(\C_J)^n$.

The next result contains some characterizations of slice regularity.

\begin{proposition} \label{prop:slice-regularity}
Let $f\in\mc{S}^1(\OO_D,A)$ and let $F=\sum_{K \in \pa(n)}e_KF_K:D \to A \otimes \R^{2^n}$ be the $\mathscr C^1$ stem function inducing $f$. The following assertions are equivalent:
\begin{itemize}
  \item[$(\mr{i})$] $f$ is slice regular.
  \item[$(\mr{i}')$] $\displaystyle\frac{\partial f}{\partial x_h^c}=0$ on $\OO_D$ for all $h\in\{1,\ldots,n\}$.
  \item[$(\mr{ii})$] For each $K \in \pa(n)$ and for each $h \in \{1,\ldots,n\}$ with $h\not\in K$, it holds:
\[
\dd{F_K}{\alpha_h}=\dd{F_{K \cup \{h\}}}{\beta_h}
\;\;\text{ and }\;\;
\dd{F_K}{\beta_h}=-\dd{F_{K \cup \{h\}}}{\alpha_h}.
\]
  \item[$(\mr{iii})$] There exists $J \in \cS_A$ such that $f_J:\OO_D(J)\to A$ is holomorphic w.r.t.\ the complex structures on $\OO_D(J)$ and on $A$ defined by the left multiplication by $J$; that is,
\begin{equation}\label{eq:holom-f_I}
\dd{f_J}{\alpha_h}(z)+J\dd{f_J}{\beta_h}(z)=0\;\text{ for all $z\in\OO_D(J)$ and for all $h \in \{1,\ldots,n\}$,}
\end{equation}
where $z=(\alpha_1+J\beta_1,\ldots,\alpha_n+J\beta_n)$ are the coordinates of points $z\in(\C_J)^n$.
  \item[$(\mr{iii}^{\pr})$] For each $J \in \cS_A$, $f_J$ is holomorphic in the sense of \eqref{eq:holom-f_I}.
 \end{itemize}
\end{proposition}
\begin{proof}
Equivalence $(\mr{i}) \Leftrightarrow (\mr{i}')$ is an immediate consequence of the fact that, by Proposition~\ref{prop:representation}, a slice function is null if and only if its stem function is. Equivalence $(\mr{i}) \Leftrightarrow (\mr{ii})$ follows immediately from Lemma \ref{lem:h-holomorphy}. The implication $(\mr{iii}^{\pr}) \Rightarrow (\mr{iii})$ is evident. We will show that $(\mr{ii})\Rightarrow (\mr{iii}^{\pr})$ and $(\mr{iii})\Rightarrow (\mr{ii})$ completing the proof.

Let $J \in \cS_A$, let $x=(\alpha_1+J\beta_1,\ldots,\alpha_n+J\beta_n) \in \OO_D(J)$, let $z=(\alpha_1+\ui \beta_1,\ldots,\alpha_n+\ui \beta_n) \in D$ and let $h \in \{1,\ldots,n\}$. By Artin's theorem, we have that 
\begin{equation} \label{eq:f_I}
\textstyle
f_J(x)=\sum_{H \in \pa(n)}J^{|H|}F_H(z).
\end{equation}
Consequently, $f_J(x)=\sum_{H\in\pa(n),h\not\in H}J^{|H|}\big(F_H(z)+JF_{H \cup \{h\}}(z)\big)$ and hence
\begin{align*}
\left(\partial_{\alpha_h}f_J+J\partial_{\beta_h}f_J\right)(x)
=&\textstyle\sum_{H\in\pa(n),h\not\in H}J^{|H|}
\big(\partial_{\alpha_h}F_H-\partial_{\beta_h}F_{H \cup\{h\}}\big)(z)+\\
&\textstyle+\sum_{H\in\pa(n),h\not\in H}J^{|H|+1}\big(\partial_{\alpha_h}F_{H \cup\{h\}}+ \partial_{\beta_h}F_H\big)(z)=0.
\end{align*}
This proves implication $(\mr{ii})\Rightarrow (\mr{iii}^{\pr})$. Finally, suppose that $(\mr{iii})$ holds. Let us prove $(\mr{ii})$. Let $z \in D$ and $x \in \OO_D$ be as above and let $L \in \pa(n)$. Recall that $\overline{z}^L \in D$ and $x^{\, c,L} \in \OO_D$. For each $H \in \pa(n)$, define the function $F_{H,L}:D \to A$ by setting $F_{H,L}(z):=F_H(\overline{z}^L)$.  Thanks to (\ref{eq:stem2}) and (\ref{eq:f_I}), we obtain that
\[
\textstyle
f_J(x)=\sum_{H \in \pa(n)}J^{|H|}(-1)^{|H \cap L|}F_{H,L}(z).
\] 
On the other hand, it is immediate to verify that $\partial_{\alpha_h}F_{H,L}(z)=\partial_{\alpha_h}F_H(\overline{z}^L)$ and $\partial_{\beta_h}F_{H,L}(z)=(-1)^{|L \cap \{h\}|}\partial_{\beta_h}F_H(\overline{z}^L)$. Using again Artin's theorem, it follows that
\begin{align*}
0=\left(\partial_{\alpha_h}f_J+J\partial_{\beta_h}f_J\right)(x^{\, c,L})=&
\textstyle
\sum_{H \in \pa(n)}J^{|H|}(-1)^{|H \cap L|}\left(\partial_{\alpha_h}F_{H,L}+J\partial_{\beta_h}F_{H,L}\right)(\overline{z}^L)=\\
=&
\textstyle
\sum_{H \in \pa(n)}J^{|H|}(-1)^{|H \cap L|}\partial_{\alpha_h}F_H(z)+\\
&
\textstyle
+\sum_{H \in \pa(n)}J^{|H|+1}(-1)^{|H \cap L|+|L \cap \{h\}|}\partial_{\beta_h}F_H(z).
\end{align*}
We have just proven that
\begin{equation}\label{eq:trucco}
0=\textstyle
\sum_{H \in \pa(n)}J^{|H|}(-1)^{|H \cap L|}\partial_{\alpha_h}F_H(z)+\sum_{H \in \pa(n)}J^{|H|+1}(-1)^{|H \cap L|+|L \cap \{h\}|}\partial_{\beta_h}F_H(z).
\end{equation}
Fix $K \in \pa(n)$ and $h \in \{1,\ldots,n\}$ with $h\not\in K$. Multiply both members of (\ref{eq:trucco}) by $(-1)^{|L \cap K|}$ and sum over all $L \in \pa(n)$. Bearing in mind Lemma \ref{lem:combinatorial}, we obtain 
\begin{align*}
0=&
\textstyle
\sum_{H \in \pa(n)}J^{|H|}\partial_{\alpha_h}F_H(z)\left(\sum_{L \in \pa(n)}(-1)^{|H \cap L|+|L \cap K|}\right)+\\
&
\textstyle
+\sum_{H \in \pa(n)}J^{|H|+1}\partial_{\beta_h}F_H(z)\left(\sum_{L \in \pa(n)}(-1)^{|H \cap L|+|L \cap \{h\}|+|L \cap K|}\right)=\\
=&\,
\textstyle
2^nJ^{|K|}\partial_{\alpha_h}F_K(z)+\sum_{H \in \pa(n)}J^{|H|+1}\partial_{\beta_h}F_H(z)\left(\sum_{L \in \pa(n)}(-1)^{|H \cap L|+|L \cap (K \cup \{h\})|}\right)=\\
=&\,
\textstyle
2^nJ^{|K|}\partial_{\alpha_h}F_K(z)+2^nJ^{|K|+2}\partial_{\beta_h}F_{K \cup \{h\}}(z)=2^nJ^{|K|}\left(\partial_{\alpha_h}F_K(z)-\partial_{\beta_h}F_{K \cup \{h\}}(z)\right),
\end{align*}
and hence $\partial_{\alpha_h}F_K(z)=\partial_{\beta_h}F_{K \cup \{h\}}(z)$. Similarly, multiplying both members of (\ref{eq:trucco}) by $(-1)^{|L \cap (K \cup \{h\})|}$ and summing over all $L \in \pa(n)$, we obtain
\[
0=2^nJ^{|K|+1}\left(\partial_{\alpha_h}F_{K \cup \{h\}}(z)+\partial_{\beta_h}F_K(z)\right),
\]
and hence $\partial_{\beta_h}F_K(z)=-\partial_{\alpha_h}F_{K \cup \{h\}}(z)$. This proves $(\mr{ii})$.
\end{proof}

Slice regular functions include polynomial functions.

\begin{proposition}\label{prop:polynomials-sr}
All polynomial functions from $\OO_D$ to $A$ are slice regular.
\end{proposition}
\begin{proof}
For each $k\in\{1,\ldots,n\}$, it is immediate to see that the coordinate function $x_k:\OO_D\to A$ is slice regular; indeed $\frac{\partial x_k}{\partial x_h^c}=0$ on $\OO_D$ for all $h\in\{1,\ldots,n\}$. Combining Lemma \ref{lem:polynomials}, equation \eqref{eq:L2-tensor} and equivalence $(\mr{i})\Leftrightarrow(\mr{i}')$ of Proposition \ref{prop:slice-regularity}, we deduce at once that all polynomials functions from $\OO_D$ to $A$ are slice regular.
\end{proof}


\subsection{On the zeros of polynomials}

This section deals with the zero set of polynomials over $A$ in the case $A$ is the division algebra $\HH$ of quaternions, the one $\oo$ of octonions or the Clifford algebra $\R_m=\mi{C}\ell(0,m)$ for $m\geq3$.   

Given any function $f:(Q_A)^n\to A$, we denote $V(f)$ its zero set, i.e.
\[
V(f):=\{x\in(Q_A)^n:f(x)=0\}.
\]

If $A=\HH$ and $f$ is a polynomial in the sense of Definition \ref{def:polynomials}, then the next result gives some properties of $V(f)$. Identify $\HH^n$ with $\R^{4n}$ by choosing one of its real vector basis.

\begin{proposition}\label{prop:dim-HH}
Let $f:\HH^n\to\HH$ be a nonconstant polynomial. Then $V(f)$ is a nonempty real algebraic subset of $\R^{4n}$ and its dimension $\dim_\R(V(f))$, as a real algebraic set, satisfies the following estimates:
\begin{equation}\label{eq:estimate-HH}
4n-4\leq\dim_\R(V(f))\leq4n-2.
\end{equation}
\end{proposition}
\begin{proof}
The equation $f=0$ in $n$ quaternionic variables is equivalent to a system of four real polynomial equations in $4n$ real variables, so $V(f)$ is a real algebraic subset of $\R^{4n}$.

Let us prove \eqref{eq:estimate-HH} by induction on $n\geq1$. If $n=1$, this is an immediate consequence of the Fundamental Theorem of Algebra for quaternions \cite{Ni1941}, FTA for short.
Let $n\geq2$ and let $f:\HH^n\to\HH$ be a nonconstant polynomial function. We can assume that $f$ has the following form: $f(x)=\sum_{h=0}^dx_1^hp_h(x')$, where $d\geq0$, each $p_h:\HH^{n-1}\to\HH$ is a polynomial function in the variables $x':=(x_2,\ldots,x_n)$ and $p_d$ does not vanish identically on $\HH^{n-1}$. Note that the zero set $W$ of $p_d$ in $\HH^{n-1}$ is either empty if $p_d$ is constant or, by induction, $\dim_\R(W)\leq4(n-1)-2=4n-6$ if $p_d$ is not constant; in particular, $W\neq\HH^{n-1}$. 

Assume that $d\ge1$ and let $y'\not\in W$. Then the FTA implies that $V(f)\cap(\HH\times\{y'\})$ is either a nonempty finite set $F$ or a nonempty finite union $S$ of $2$-spheres of $\HH=\R^4$ or a nonempty set of the form $F\cup S$. In particular, $V(f)\neq\emptyset$. Let $\pi:V(f)\to\HH^{n-1}$ be the projection $\pi(x_1,x'):=x'$. We have just proven that $\pi^{-1}(y')$ is nonempty and $\dim_\R(\pi^{-1}(y'))\leq2$ for all $y'\not\in W$. Evidently, if $y'\in W$, then $\pi^{-1}(y')\subset\HH\times\{y'\}$ and hence $\dim_\R(\pi^{-1}(y'))\leq4$. By the version of Sard's theorem in real algebraic geometry, we deduce at once that
\[
4n-4=\dim_\R(\HH^{n-1})\leq\dim_\R(\pi^{-1}(\HH^{n-1}\setminus W))\leq\dim_\R(\HH^{n-1})+2=4n-2
\]
and
\[
\dim_\R(\pi^{-1}(W))\leq\dim_\R(W)+4\leq(4n-6)+4=4n-2.
\]
Since $\dim_\R(V(f))=\max\{\dim_\R(\pi^{-1}(\HH^{n-1}\setminus W)),\dim_\R(\pi^{-1}(W))\}$, we are done.

Assume now that $d=0$. In this case it must be $n>1$, since $f$ is nonconstant, and $f$ coincides with the polynomial $p_0(x')$ in the variables $x_2,\ldots, x_n$. From the previous argument applied to $p_0$, we obtain that $V(f)$ is nonempty and 
$4(n-1)-4\leq\dim_\R(V(p_0)\cap\HH^{n-1})\leq4(n-1)-2$,
from which it follows again that
\[
4(n-1)\leq\dim_\R(V(f))\leq4(n-1)+2.
\]
\end{proof}

\begin{example}\label{ex:rag-HH}
Estimates \eqref{eq:estimate-HH} are sharp. Suppose that  $n\geq2$. Consider the polynomial functions $f_1,f_2,f_3:\HH^n\to\HH$ defined by setting
\[
f_1(x):=x_1, \qquad f_2(x):=x_1^2+x_2^2+1, \qquad f_3(x):=x_1^2+1, 
\]
for all $x=(x_1,x_2,\ldots,x_n)\in\HH^n$. Evidently, $\dim_\R(V(f_1))=4n-4$ and $\dim_\R(V(f_3))=4n-2$. Let us study $V(f_2)$. Note that $x_2^2+1$ is a positive real number if and only if $x_2$ belongs to $3$-dimensional semi-algebraic subset $S$ of $\HH$ defined by
\[
S:=\R\cup\{x_2\in\HH:\mr{Re}(x_2)=0,n(x_2)<1\}
\]
Let $\pi:V(f_2)\to\HH^{n-1}$ be the projection $\pi(x,x'):=x'$, where $x':=(x_2,\ldots,x_n)$. Note that $\pi^{-1}(y')$ consists of a single point if $y'\in\cS_A\times\HH^{n-2}$, a $2$-sphere if $y'\in S\times\HH^{n-2}$ and two distinct points if $y'\in\HH^{n-1}\setminus((\cS_A\cup S)\times\HH^{n-2})$. As a consequence, we have:
\begin{itemize}
 \item $\dim_\R(\pi^{-1}(\cS_A\times\HH^{n-2}))=\dim_\R(\cS_A\times\HH^{n-2})+0=2+4(n-2)=4n-6$,
 \item $\dim_\R(\pi^{-1}(S\times\HH^{n-2}))=\dim_\R(S\times\HH^{n-2})+2=3+4(n-2)+2=4n-3$,
 \item $\dim_\R(\pi^{-1}(\HH^{n-1}\setminus((\cS_A\cup S)\times\HH^{n-2})))=\dim_\R(\HH^{n-1})+0=4(n-1)=4n-4$.
 \end{itemize}
It follows that $\dim_\R(V(f_2))=\max\{4n-6,4n-3,4n-4\}=4n-3$. \bs
\end{example}

Proposition \ref{prop:dim-HH} remains valid over the octonion algebra $\oo$. The proof is identical,
thanks to the fact that the Fundamental Theorem of Algebra still holds in this case (see \cite{Jou,Serodio2007}).

\begin{proposition}\label{prop:dim-oo}
Let $f:\oo^n\to\oo$ be a nonconstant polynomial. Then $V(f)$ is a nonempty real algebraic subset of $\R^{8n}$ and it holds
\begin{equation}\label{eq:estimate-oo}
8n-8\leq\dim_\R(V(f))\leq8n-2.
\end{equation}
\end{proposition}

\begin{remark}
Suppose that $n\geq2$. Let $f_1,f_2,f_3:\oo^n\to\oo$ be polynomial functions defined as in Example \ref{ex:rag-HH}. Repeating the same considerations we made in the mentioned example, we obtain that $\dim_\R(V(f_1))=8n-8$, $\dim_\R(V(f_2))=8n-3$ and $\dim_\R(V(f_3))=8n-2$. \bs 
\end{remark}

A weaker version of Proposition \ref{prop:dim-HH} is valid also over all Clifford algebra $\R_m$ with $m\geq3$.

Choose $m\geq3$ and equip $\R_m$ with the Clifford conjugation. Given any $a\in\R_m$, we write $a=\sum_{K\in\pa(m)}a_Ke_K$ for $a_K\in\R$. Recall that $t(a)=2\sum'_Ka_Ke_K$ and $n(a)=\sum'_K\langle a,ae_K\rangle e_K$, where $\sum'_K=\sum_{K\in\pa(m),|K|\equiv0,3\;(\text{mod } 4)}$ and $\langle\cdot,\cdot\rangle$ is the standard Euclidean scalar product on $\R_m=\R^{2^m}$, see Section 3.2 of \cite{GHS}. As a consequence, $a=\sum_{K\in\pa(m)}a_Ke_K$ belongs to $Q_{\R_m}$ if and only if $a_K=0$ and $\langle a,ae_K\rangle=0$ for all $K\in\pa(m)\setminus\{\emptyset\}$ such that $|K|\equiv0,3\;(\text{mod }4)$. Moreover, $a=\sum_{K\in\pa(m)}a_Ke_K$ belongs to $\cS_{\R_m}$ if and only if $a\in Q_{\R_m}$, $a_\emptyset=0$ and $\langle a,a\rangle=1$. It turns out that $Q_{\R_m}$ and $\cS_{\R_m}$ are real algebraic subsets of $\R_m=\R^{2^m}$. Since $Q_{\R_m}\setminus\R$ is semi-algebraically homeomorphic to $\cS_{\R_m}\times\{(\alpha,\beta)\in\R^2:\beta>0\}$, we have that
\begin{equation}\label{eq:qmsm} 
\dim_\R(\cS_{\R_m})=\dim_\R(Q_{\R_m})-2.
\end{equation}
Note that $Q_{\R_3}\subset Q_{\R_m}$; moreover, $\dim_\R(Q_{\R_3})=6$, see Examples 1(3) of \cite{GhPe_Trends} 
and \cite[Example 1.15]{AlgebraSliceFunctions}. It follows that $\dim_\R(Q_{\R_m})\geq6$. Denote $\R^{m+1}$ the real vector subspace of $\R_m$ of paravectors $a=a_0+\sum_{h=1}^ma_he_h$, where $a_0:=a_\emptyset$ and $a_h:=a_{\{h\}}$ (and $e_h=e_{\{h\}}$). It is easy to see that $\R^{m+1}$ is contained in $Q_{\R_m}$.

The Fundamental Theorem of Algebra still holds for one variable polynomials with paravector Clifford coefficients (see \cite[Examples 9(1)]{AIM2011}). This fact allows to generalize the first statement of Proposition \ref{prop:dim-HH} to this setting.

\begin{proposition}
Let $f:(Q_{\R_m})^n\to\R_m$ be a nonconstant polynomial with paravector coefficients. Then $V(f)$ is a nonempty real algebraic subset of $(Q_{\R_m})^n\subset\R^{n2^m}$.
\end{proposition}
\begin{proof}
We proceed by induction on $n$. If $n=1$, the FTA \cite[Examples 9(1)]{AIM2011} implies the thesis. Assume $n\ge2$ and write $f(x)=\sum_{h=0}^dx_1^hp_h(x')$, where $d\geq0$, each $p_h:(Q_{\R_m})^{n-1}\to\R_m$ is a polynomial function in the variables $x':=(x_2,\ldots,x_n)$ and $p_d$ does not vanish identically on $(Q_{\R_m})^{n-1}$. If $d=0$, then $f(x)=p_0(x')$ is nonconstant and therefore $V(f)\ne\emptyset$ by the induction assumption. If $d\ge1$, then for every fixed point $y'=(t_2,\ldots,t_n)\in\R\times\cdots\times\R\subset(Q_{\R_m})^{n-1}$, the function $f(x_1,y')=\sum_{h=0}^dx_1^hp_h(y')$ is a nonconstant polynomial in the variable $x_1$ with paravector coefficients. Therefore $V(f)$ is nonempty.
\end{proof}

When $n>1$, the zero set $V(f)$ is infinite for every nonconstant polynomial with paravector coefficients. It is difficult to obtain general estimates for the real dimension of $V(f)$. The next proposition suggests the types of results one can expect.

\begin{proposition}\label{prop:clifford_m}
Let $f:Q_{\R_m}\times(\R^{m+1})^{n-1}\to\R_m$ be a function of the form
\begin{equation}\label{eq:fx1x'}
\textstyle
f(x_1,x')=x_1^d+\sum_{h=0}^{d-1}x_1^hp_h(x')
\end{equation}
for all $(x_1,x')\in Q_{\R_m}\times(\R^{m+1})^{n-1}$, where $d\geq1$ and each $p_h:(Q_{\R_m})^{n-1}\to\R_m$ is a polynomial function in the variables $x'=(x_2,\ldots,x_n)$ such that 
\begin{equation}\label{eq:p_h}
\text{$p_h((\R^{m+1})^{n-1})\subset\R^{m+1}\;$ for all $h\in\{0,\ldots,d-1\}$;} 
\end{equation}
in the case $n=1$, \eqref{eq:p_h} means that the $p_h$'s are elements of $\R^{m+1}$. Then $f^{-1}(0)$ is a nonempty real algebraic subset of $\R^{2^m+(m+1)(n-1)}$ contained in $Q_{\R_m}\times(\R^{m+1})^{n-1}$, and it holds
\begin{equation}\label{eq:estimate-R_n}
(m+1)(n-1)\leq\dim_\R(f^{-1}(0))\leq(m+1)(n-1)+\dim_\R(Q_{\R_m})-2.
\end{equation} 
\end{proposition}
\begin{proof}
If $n=1$, the statement follows immediately from \eqref{eq:qmsm}, \eqref{eq:p_h} and Examples 9(1) of \cite{AIM2011}. Suppose that $n\geq2$. Let $y'\in(\R^{m+1})^{n-1}$. Using \eqref{eq:qmsm}, \eqref{eq:p_h} and Examples 9(1) of \cite{AIM2011} again, we have that $f^{-1}(0)\cap(Q_{\R_m}\times\{y'\})$ is either a nonempty finite set $F$ or a nonempty finite union $S$ of `spheres' semi-algebraically homeomorphic to $\cS_{\R_m}$ or a nonempty set of the form $F\cup S$. In particular, $f^{-1}(0)\neq\emptyset$ and, if $\pi:f^{-1}(0)\to(\R^{m+1})^{n-1}$ is the projection $\pi(x_1,x'):=x'$, then the fibers of $\pi$ are nonempty real algebraic sets of dimension ranging from $0$ to $\dim_\R(\cS_{\R_m})$. This fact and \eqref{eq:qmsm} completes the proof.
\end{proof}

\begin{example}
Suppose that $n\geq2$. Choose $m\geq 3$. Set $N:=(m+1)(n-1)$ and $M:=\dim_\R(Q_{\R_m})$. Define the functions  $f_1,f_2,f_3:Q_{\R_m}\times(\R^{m+1})^{n-1}\to\R_m$ as in Example \ref{ex:rag-HH}. Note that each of these functions satisfies \eqref{eq:fx1x'} and \eqref{eq:p_h}. Evidently, $\dim_\R(f_1^{-1}(0))=N$ and $\dim_\R(f_3^{-1}(0))=N+M-2$.
Let us study $f_2^{-1}(0)$ following the strategy used in Example \ref{ex:rag-HH} to investigate $V(f_2)$. Note that, if $x_2\in\R^{m+1}$, then $x_2^2+1$ is a positive real number if and only if $x_2\in S$, where $S$ is the $m$-dimensional semi-algebraic set defined by
\[\textstyle
S:=\R\cup\big\{\sum_{h=0}^ma_he_h\in\R^{m+1}:a_0=0,\,\sum_{h=1}^ma_h^2<1\big\}.
\]
Let $\partial S$ be the $(m-1)$-sphere $\{\sum_{h=0}^ma_he_h\in\R^{m+1}:a_0=0,\,\sum_{h=1}^ma_h^2=1\}$ and let $\pi:f_2^{-1}(0)\to(\R^{m+1})^{n-1}$ be the projection $\pi(x,x'):=x'$, where $x':=(x_2,\ldots,x_n)$. Note that $\pi^{-1}(y')$ consists of a single point if $y'\in\partial S\times(\R^{m+1})^{n-2}$, a `sphere' semi-algebraically homeomorphic to $\cS_{\R_m}$ if $y'\in  S\times(\R^{m+1})^{n-2}$ and two distinct points if $y'\in(\R^{m+1})^{n-1}\setminus((S\cup\partial S)\times(\R^{m+1})^{n-2})$. As a consequence, we have:
\begin{itemize}
 \item $\dim_\R(\pi^{-1}(\partial S\times(\R^{m+1})^{n-2}))=(m-1)+(m+1)(n-2)<N$,
 \item $\dim_\R(\pi^{-1}(S\times(\R^{m+1})^{n-2}))=m+(m+1)(n-2)+\dim_\R(\cS_{\R_m})=N-1+\dim_\R(\cS_{\R_m})$,
 \item $\dim_\R(\pi^{-1}((\R^{m+1})^{n-1}\setminus((S\cup\partial S)\times(\R^{m+1})^{n-2})))=N$.
\end{itemize}
Thanks to \eqref{eq:qmsm} and the fact that $M=\dim_\R(Q_{\R_m})\geq6$, it follows that
\[
N-1+\dim_\R(\cS_{\R_m})=N-1+M-2=N+M-3\geq N+3>N.
\]
As a consequence, we have that $\dim_\R(f_2^{-1}(0))=\dim_\R(\pi^{-1}(S\times(\R^{m+1})^{n-2}))=N+M-3$. \bs
\end{example}


\subsection{One variable interpretation of slice regularity}

We now show that the condition of slice regularity in several variables has an interpretation in terms of slice regularity in one variable. More precisely, we show that the slice regularity of a slice function $f:\OO_D\to A$ is equivalent to the slice regularity of all its $2^n-1$ truncated spherical derivatives $\SD_\epsilon f$ w.r.t.\ the single variable $\mr{x}_h$, where $h-1$ is the order of $\SD_\epsilon$. The notion of truncated spherical $\epsilon$-derivatives was introduced in above Definition \ref{def:D_epsilon}.

Let $g:\OO_D\to A$ be a function and let $h\in\{1,\ldots,n\}$. Recall that, by Definition \ref{def:gy}, $g$ is a slice function w.r.t.\ $\mr{x}_h$ if, for each $y=(y_1,\ldots,y_n)\in\OO_D$, the restriction function $g_h^{\sss(y)}:\OO_{D,h}(y)\to A$, sending $x_h$ into $g_h^{\sss(y)}(x_h):=g(y_1,\ldots,y_{h-1},x_h,y_{h+1},\ldots,y_n)$, is a slice function. Let us specialize this definition to the slice regular case.

\begin{definition}\label{def:gy-sr}
Let $g:\OO_D\to A$ be a function and let $h\in\{1,\ldots,n\}$. We say that $g$ is a \emph{slice regular function w.r.t.\ $\mr{x}_h$} if, for each $y\in\OO_D$, the function $g_h^{\sss(y)}:\OO_{D,h}(y)\to A$ is a slice regular function. \bs
\end{definition}

\begin{theorem}\label{prop:regularwrtxh}
Assume that $n\geq2$. Let $f:\OO_D\to A$ be a slice function. Then $f$ is slice regular if and only if $f$ is a slice regular function w.r.t.\ $\mr{x}_1$ and, for each $h\in\{2,\ldots,n\}$ and each function $\epsilon:\{1,\ldots,h-1\}\to\{0,1\}$, the truncated spherical $\epsilon$-derivative $\SD_\epsilon f$ of $f$ is a slice regular function w.r.t.\ $\mr{x}_h$.
\end{theorem}

\begin{proof}
We use the notation introduced in \eqref{def:Dhz} and \eqref{def:OmegaDhx}. If $x=(x_1,\ldots,x_n)\in \OO_D$, then $\OO_{D,h}(x)$ is the subset of $Q_A$ defined by $\OO_{D,h}(x)=\{a\in A:(x_1,\ldots,x_{h-1},a,x_{h+1},\ldots,x_n)\in\OO_D\}$. It holds $\OO_{D,h}(x)=\OO_{D_h(z)}$, where $D_h(z)=\{w\in\C:(z_1,\ldots,z_{h-1},w,z_{h+1},\ldots,z_n)\in D\}$.

Assume that $f$ is slice regular. 
Let $F=\sum_{H\in\pa(n)}e_HF_H:D\to A\otimes\R^{2^n}$ be the stem function inducing $f$, let $y=(y_1,\ldots,y_n)=(\alpha_1+I_1\beta_1,\ldots,\alpha_n+I_n\beta_n)\in\OO_D$, let $I:=(I_1,\ldots,I_n)$ and let $w=(w_1,\ldots,w_n):=(\alpha_1+i\beta_1,\ldots,\alpha_n+i\beta_n)\in D$. 
Let us prove  that $f$  is slice regular w.r.t.\ $\mr{x}_1$. As seen in formula \eqref{eq:fx1}, it holds $f(x_1,y')=F_1(z_1)+J_1F_2(z_1)$,  where $x_1=\alpha_1+J_1\beta_1\in\OO_{D,1}(y)$, $y'=(y_2,\ldots,y_n)$, $z_1=\alpha_1+i\beta_1\in D_1(w)$, and  
\[
F_1(z_1)\textstyle=\sum_{H\in\pa(n),1\not\in H}[J_H,F_H(z_1,w')]
\;\;\text{ and }\;\;
F_2(z_1)\textstyle=\sum_{H\in\pa(n),1\not\in H}[J_H,F_{H\cup\{1\}}(z_1,w')],
\]
with $w'=(w_2,\ldots,w_n)$ and $J=(J_1,I_2,\ldots,I_n)$. From implication $(\mr{i})\Rightarrow (\mr{ii})$ of Proposition \ref{prop:slice-regularity}, it follows that $\partial_{\alpha_1}F_1(z_1)=\partial_{\beta_1}F_2(z_1)$ and $\partial_{\beta_1}F_1(z_1)=-\partial_{\alpha_1}F_2(z_1)$, that is $f$ is slice regular w.r.t.\ $\mr{x}_1$ on $\OO_{D,1}(y)$. 

Now let $h\in\{2,\ldots,n\}$ and let $\epsilon:\{1,\ldots,h-1\}\to\{0,1\}$ be any function. Set $K_{h-1}:=\epsilon^{-1}(1)$. Let us prove that $\SD_\epsilon f$ is slice regular w.r.t.\  $\mr{x}_h$. From formula \eqref{eq:DD}, it follows that $\SD_\epsilon f(y'',x_{h},\hat{y})=F_{h,1}(z_h)+J_hF_{h,2}(z_h)$, where
\begin{align*}
F_{h,1}(z_h)&\textstyle:=(\beta_{K_{h-1}})^{-1}\sum_{H\in\pa(n),H_{h}=\emptyset}[L_H,F_{H\cup K_{h-1}}(z'',z_{h},\hat{z})],\\
F_{h,2}(z_h)&\textstyle:=(\beta_{K_{h-1}})^{-1}\sum_{H\in\pa(n),H_{h}=\emptyset}[L_H, F_{H\cup K_{h-1}\cup\{h\}}(z'',z_{h},\hat{z})],
\end{align*}
$x_{h}\in\OO_{D,h}(y)$, $y''=(y_1,\ldots,y_{h-1})$, $\hat{y}=(y_{h+1},\ldots,y_n)$, $z_{h}=\alpha_{h}+i\beta_{h}\in D_{h}(z)$, $z''=(z_1,\ldots,z_{h-1})$, $\hat{z}=(z_{h+1},\ldots,z_n)$ and $L=(I_1,\ldots,I_{h-1},J_{h},I_{h+1},\ldots,I_n)$. Again from implication $(\mr{i})\Rightarrow (\mr{ii})$ of Proposition \ref{prop:slice-regularity}, 
it follows that $\partial_{\alpha_h}F_{h,1}(z_h)=\partial_{\beta_h}F_{h,2}(z_h)$ and $\partial_{\beta_h}F_{h,1}(z_h)=-\partial_{\alpha_h}F_{h,2}(z_h)$, i.e. $\SD_\epsilon f$ is slice regular w.r.t.\ $\mr{x}_h$ on $\OO_{D,h}(y)$. 

Conversely, assume that $f$ is slice regular w.r.t.\ $\mr{x}_1$ and that the functions $\SD^{\varepsilon(h-1)}_{\mr{x}_{h-1}}\cdots\SD^{\varepsilon(1)}_{\mr{x}_1}f$ are slice regular w.r.t.\ $\mr{x}_h$ for all $K\in\pa(n)$ and all $h\in\{2,\ldots,n\}$, where $\varepsilon:\{1,\ldots,n\}\to\{0,1\}$ is the characteristic function of $K$.

Let $K\in\pa(n)$ and let $h\in\{1,\ldots,n\}$ with $h\not\in K$.  Consider any $z\in D_{K_{h-1}}$, where $D_{K_{h-1}}=\bigcap_{k\in K_{h-1}}\{(z_1,\ldots,z_n)\in D:z_k\not\in\R\}$. From formulas \eqref{eq:fx1} and \eqref{eq:DD}, we obtain that it holds:
\begin{align}\label{eq:sr_one}
&\textstyle\sum_{H\in\pa(n),H_{h}=\emptyset}[L_H,\partial_{\alpha_h}F_{H\cup K_{h-1}}(z)-\partial_{\beta_h}F_{H\cup K_{h-1}\cup\{h\}}(z)]=0\\\label{eq:sr_one1}
&\textstyle\sum_{H\in\pa(n),H_{h}=\emptyset}[L_H,\partial_{\beta_h}F_{H\cup K_{h-1}}(z)+\partial_{\alpha_h}F_{H\cup K_{h-1}\cup\{h\}}(z)]=0,
\end{align}
where $L=(I_1,\ldots,I_{h-1},J_{h},I_{h+1},\ldots,I_n)$ and $H_0=K_0=\emptyset$.

Let $M\in\pa(n)$. Thanks to \eqref{eq:stem2}, for each $H\in\pa(n)$ such that $H_h=\emptyset$, it holds
\begin{equation}\label{eq:11}
\partial_{\alpha_h}F_{H\cup K_{h-1}}(\overline{z}^M)=(-1)^{|(H\cup K_{h-1})\cap M|}\partial_{\alpha_h}F_{H\cup K_{h-1}}(z);
\end{equation}
moreover, being $|(H\cup K_{h-1}\cup\{h\})\cap M|=|(H\cup K_{h-1})\cap M|+|M\cap\{h\}|$, it holds
\begin{equation}\label{eq:12}
\partial_{\beta_h}F_{H\cup K_{h-1}{\cup\{h\}}}(\overline{z}^M)=(-1)^{|(H\cup K_{h-1})\cap M|}\partial_{\beta_h}F_{H\cup K_{h-1}}(z).
\end{equation}
Thanks to \eqref{eq:11}, \eqref{eq:12} and the validity of \eqref{eq:sr_one} at the point $\overline{z}^M$, we get that
\begin{align}\label{eq:sr_one_M}
&\textstyle\sum_{H\in\pa(n),H_{h}=\emptyset}(-1)^{|(H\cup K_{h-1})\cap M|}[L_H,\partial_{\alpha_h}F_{H\cup K_{h-1}}(z)-\partial_{\beta_h}F_{H\cup K_{h-1}\cup\{h\}}(z)]=0
\end{align}
for all $M\in\pa(n)$. Multiply both members of $\eqref{eq:sr_one_M}$ by $(-1)^{|K\cap M|}$ and sum over all $M\in\pa(n)$. Using the combinatorial Lemma \ref{lem:combinatorial}, we get
\begin{align*}
0&\textstyle=\sum_{M\in\pa(n)}(-1)^{|K\cap M|}
\sum_{H\in\pa(n),H_{h}=\emptyset}(-1)^{|(H\cup K_{h-1})\cap M|}[L_H,\partial_{\alpha_h}F_{H\cup K_{h-1}}-\partial_{\beta_h}F_{H\cup K_{h-1}\cup\{h\}}]=\\
&\textstyle=\sum_{H\in\pa(n),H_{h}=\emptyset}
\sum_{M\in\pa(n)}(-1)^{|K\cap M|+|(H\cup K_{h-1})\cap M|}[L_H,\partial_{\alpha_h}F_{H\cup K_{h-1}}-\partial_{\beta_h}F_{H\cup K_{h-1}\cup\{h\}}]=\\
&\textstyle=\sum_{H\in\pa(n),H_{h}=\emptyset}
2^n\delta_{K,H\cup K_{h-1}}[L_H,\partial_{\alpha_h}F_{H\cup K_{h-1}}-\partial_{\beta_h}F_{H\cup K_{h-1}\cup\{h\}}]=\\
&\textstyle=2^n[L_{K\setminus K_h},\partial_{\alpha_h}F_{K}-\partial_{\beta_h}F_{K\cup\{h\}}]
\end{align*}
on the whole $D_{K_{h-1}}$. Therefore $\partial_{\alpha_h}F_{K}(z)=\partial_{\beta_h}F_{K\cup\{h\}}(z)$ on $D_{K_{h-1}}$. By Assumption \ref{assumption:openess}, we know that $D$ is open in $\C^n$. As a consequence, $D_{K_{h-1}}$ is dense in $D$. Since $F$ is of class $\mscr{C}^1$ on $D$, the partial derivatives $\partial_{\alpha_h}F_{K}$ and $\partial_{\beta_h}F_{K\cup\{h\}}(z)$ are continuous on $D$. It follows that $\partial_{\alpha_h}F_{K}(z)=\partial_{\beta_h}F_{K\cup\{h\}}(z)$ on the whole $D$. In a similar way, we deduce from \eqref{eq:sr_one1} that $\partial_{\beta_h}F_{K}=-\partial_{\alpha_h}F_{K\cup\{h\}}$ on $D$. From implication $(\mr{ii})\Rightarrow (\mr{i})$ of Proposition \ref{prop:slice-regularity}, it follows that $f$ is slice regular on $\OO_D$.
\end{proof}


\subsection{Leibniz's rule}

The next lemma gives sufficient conditions for Leibniz's rule to be valid for each $\partial_h$ and $\dibar_h$.

\begin{lemma}\label{lem:technical}
Let $\mr{b}=\EuScript{B}(\sigma)$ be a $\dsim$-product of $\R^{2^n}$, let $h\in\{1,\ldots,n\}$ and let $F,G:D\to A\otimes\R^{2^n}$ be $\mscr{C}^1$ stem functions. Write $F=\sum_{K\in\pa(n)}e_KF_K$ and $G=\sum_{H\in\pa(n)}e_HG_H$. Define
\begin{align*}
P'_{1,h}&\textstyle:=\big\{(K,H)\in\pa(n)\times\pa(n):\frac{\partial F_K}{\partial\beta_h}\cdot_\sigma G_H=0 \text{ on $D$}\big\},\\
P_{1,h}&:=(\pa(n)\times\pa(n))\setminus P'_{1,h},\\
P'_{2,h}&\textstyle:=\big\{(K,H)\in\pa(n)\times\pa(n):F_K\cdot_\sigma \frac{\partial G_H}{\partial\beta_h}=0 \text{ on $D$}\big\},\\
P_{2,h}&:=(\pa(n)\times\pa(n))\setminus P'_{2,h}.
\end{align*}
Suppose that the following two conditions hold:
\begin{equation}\label{eq:KHh1}
(-1)^{|(K\dsim H)\cap\{h\}|}\sigma(K,H)=(-1)^{|K\cap\{h\}|}\sigma(K\dsim\{h\},H)\;\text{ for all $(K,H)\in P_{1,h}$}
\end{equation}
and
\begin{equation}\label{eq:KHh2}
(-1)^{|(K\dsim H)\cap\{h\})}\sigma(K,H)=(-1)^{|H\cap\{h\}|}\sigma(K,H\dsim\{h\})\;\text{ for all $(K,H)\in P_{2,h}$}.
\end{equation}
Then it holds:
\begin{align}\label{eq:L1}
\partial_h(F\cdot_\sigma G)&=(\partial_hF)\cdot_\sigma G+F\cdot_\sigma (\partial_hG),\\\label{eq:L2}
\dibar_h(F\cdot_\sigma G)&=(\dibar_hF)\cdot_\sigma G+F\cdot_\sigma (\dibar_hG).
\end{align}
\end{lemma}
\begin{proof}
We prove only \eqref{eq:L1}, the proof of \eqref{eq:L2} being similar. For simplicity, we omit `$\,\cdot_\sigma$' in each product w.r.t\ $\mr{b}$. We have:
\begin{align} \label{eq:FG-regular}
2 \, \partial_h(FG)=& \, \partial_{\alpha_h}(FG)-\J_h\big(\partial_{\beta_h}(FG)\big)=(\partial_{\alpha_h}F)G+F(\partial_{\alpha_h}G)+ \nonumber\\
&
-\J_h\big((\partial_{\beta_h}F)G\big)-\J_h\big(F(\partial_{\beta_h}G)\big).\end{align}
Bearing in mind \eqref{eq:c-s1}, we have also:
\begin{align*}
\J_h\big((\partial_{\beta_h}F)G\big)&=\textstyle\J_h\big(\sum_{K,H\in\pa(n)}e_{K\dsim H}\sigma(K,H)(\partial_{\beta_h}F_K)G_H\big)=\\
&=\textstyle\sum_{(K,H)\in{P_{1,h}}}\J_h(e_{K\dsim H})\sigma(K,H)(\partial_{\beta_h}F_K)G_H=\\
&=\textstyle\sum_{(K,H)\in{P_{1,h}}}e_{K\dsim H\dsim\{h\}}(-1)^{|(K\dsim H)\cap\{h\}|}\sigma(K,H)(\partial_{\beta_h}F_K)G_H.
\end{align*}
Similarly, we deduce:
\begin{align*}
\big(\J_h(\partial_{\beta_h}F)\big)G&=\textstyle\sum_{K,H\in\pa(n)}\J_h(e_K)e_H(\partial_{\beta_h}F_K)G_H=\\
&=\textstyle\sum_{(K,H)\in{P_{1,h}}}e_{K\dsim\{h\}}e_H(-1)^{|K\cap\{h\}|}(\partial_{\beta_h}F_K)G_H=\\
&=\textstyle\sum_{(K,H)\in{P_{1,h}}}e_{K\dsim H\dsim\{h\}}(-1)^{|K\cap\{h\}|}\sigma(K\dsim\{h\},H)(\partial_{\beta_h}F_K)G_H.
\end{align*}
Consequently, \eqref{eq:KHh1} implies that $\J_h\big((\partial_{\beta_h}F)G\big)=\big(\J_h(\partial_{\beta_h}F)\big)G$. Similar computations and \eqref{eq:KHh2} ensure also that $\J_h\big(F(\partial_{\beta_h}G)\big)=F\big(\J_h(\partial_{\beta_h}G)\big)$. Combining the latter two equalities with \eqref{eq:FG-regular}, we obtain:
\begin{align*}
\partial_h(FG)=&\,\textstyle\frac{1}{2}\big((\partial_{\alpha_h}F)G+F(\partial_{\alpha_h}G)
-\big(\J_h(\partial_{\beta_h}F)\big)G-F\big(\J_h(\partial_{\beta_h}G)\big)\big)=\\
=& \,
(\partial_hF)G+F(\partial_hG),
\end{align*}
as desired.
\end{proof}

The above lemma suffices to prove that Leibniz's rule works for the slice tensor product.

\begin{proposition}\label{prop:leibniz}
For each $f,g\in\mc{S}^1(\OO_D,A)$ and for each $h\in\{1,\ldots,n\}$, it holds:
\begin{align}\label{eq:L1-tensor}
\frac{\partial}{\partial x_h}(f\tenso g)&=\frac{\partial f}{\partial x_h}\tenso g+f\tenso \frac{\partial g}{\partial x_h},\\
\vspace{.3em}\label{eq:L2-tensor}
\frac{\partial}{\partial x_h^c}(f\tenso g)&=\frac{\partial f}{\partial x_h^c}\tenso g+f\tenso \frac{\partial g}{\partial x_h^c}.
\end{align}
\end{proposition}
\begin{proof}
Given any pair $(K,H)\in\pa(n)\times\pa(n)$ and any $h\in\{1,\ldots,n\}$, it is immediate to check that $|(K\dsim H)\cap\{h\}|+|K\cap H|=|K\cap\{h\}|+|(K\dsim\{h\})\cap H|$; indeed, the preceding equality becomes $|K\cap H|=|K\cap H|$ if $h\not\in K\dsim H$, and $|K\cap H|+1=|K\cap H|+1$ if $h\in K\dsim H$. The mentioned equality is equivalent to \eqref{eq:KHh1} and \eqref{eq:KHh2} if $\mr{b}=\EuScript{B}(\sigma_\otimes^n)$, because $\sigma_\otimes^n(K,H)=(-1)^{|K\cap H|}$ by Lemma \ref{lem:tensor}. Lemma \ref{lem:technical} concludes the proof.
\end{proof}


\subsection{Multiplication of slice regular functions}

\begin{proposition}\label{prop:sr-prod}
The set $\mc{SR}(\OO_D,A)$ is a real subalgebra of $(\mc{S}(\OO_D,A),\tenso)$. Moreover, the set $\mc{SR}_\R(\OO_D,A):=\mc{S}_\R(\OO_D,A)\cap\mc{SR}(\OO_D,A)$ of all slice preserving slice regular functions from $\OO_D$ to $A$ is contained in the center of $(\mc{S}(\OO_D,A),\tenso)$. 
\end{proposition}
\begin{proof}
The first part is a direct consequence of equation \eqref{eq:L2-tensor} and equivalence $(\mr{i})\Leftrightarrow(\mr{i}')$ of Proposition \ref{prop:slice-regularity}; the second follows immediately from Lemma \ref{lem:tensor-center}. 
\end{proof}

In the next result we see that the slice tensor product is the unique associative and hypercomplex $\triangle$-product on $\R^{2^n}$, which preserves slice regularity.

\begin{proposition}\label{prop:tensor}
The tensor product $\mr{b}_\otimes^n$ is the unique associative and hypercomplex $\dsim$-product $\mr{b}=\EuScript{B}(\sigma)$ of $\R^{2^n}$ such that $\mc{SR}(\OO_D,A)$ is a real subalgebra of $(\mc{S}(\OO_D,A),\cdot_\sigma)$.
\end{proposition}

\begin{proof}
Let $\mr{b}=\EuScript{B}(\sigma)$ be a $\dsim$-product of $\R^{2^n}$, let $f=\I(F)$ and $g=\I(G)$ be functions in $\mc{SR}(\OO_D,A)$, and let $(FG)_K$ be the $K$-component of $FG=F\cdot_\sigma G$ for all $K\in\pa(n)$.

Choose $K\in\pa(n)$ and $h\in\{1,\ldots,n\}$ with $h\not\in K$. For short, during the remaining part of this proof, we will use the symbol $\sum^*$ in place of $\sum_{(K_1,K_2,K_3)\in\mscr{D}(K)}$. Bearing in mind equality \eqref{eq:K} and implication $(\mr{i})\Rightarrow(\mr{ii})$ in Proposition \ref{prop:slice-regularity}, we have:
\begin{align}
\partial_{\alpha_h}(FG)_K=&\textstyle\sum^*\left((\partial_{\alpha_h}F_{K_1\cup K_3})G_{K_2\cup K_3}+F_{K_1\cup K_3}(\partial_{\alpha_h}G_{K_2\cup K_3})\right)\sigma(K_1\cup K_3,K_2\cup K_3)=\nonumber\\
=&\textstyle\sum^*_{h\in K_3}(-\partial_{\beta_h}F_{(K_1\cup K_3)\setminus\{h\}})G_{K_2\cup K_3}\sigma(K_1\cup K_3,K_2\cup K_3)+\nonumber\\
&\textstyle+\sum^*_{h\in K_3}F_{K_1\cup K_3}(-\partial_{\beta_h}G_{(K_2\cup K_3)\setminus\{h\}})\sigma(K_1\cup K_3,K_2\cup K_3)+\nonumber\\
&\textstyle+\sum^*_{h\not\in K_3}(\partial_{\beta_h}F_{K_1\cup K_3\cup\{h\}})G_{K_2\cup K_3}\sigma(K_1\cup K_3,K_2\cup K_3)+\nonumber\\
&\textstyle+\sum^*_{h\not\in K_3}F_{K_1\cup K_3}(\partial_{\beta_h}G_{K_2\cup K_3\cup\{h\}})\sigma(K_1\cup K_3,K_2\cup K_3)=\nonumber\\
=&\textstyle\sum^*_{h\not\in K_3}(-\partial_{\beta_h}F_{K_1\cup K_3})G_{K_2\cup K_3\cup\{h\}}\sigma(K_1\cup K_3\cup\{h\},K_2\cup K_3\cup\{h\})+\label{eq:chain1}\\
&\textstyle+\sum^*_{h\not\in K_3}F_{K_1\cup K_3\cup\{h\}}(-\partial_{\beta_h}G_{K_2\cup K_3})\sigma(K_1\cup K_3\cup\{h\},K_2\cup K_3\cup\{h\})+\nonumber\\
&\textstyle+\sum^*_{h\not\in K_3}(\partial_{\beta_h}F_{K_1\cup K_3\cup\{h\}})G_{K_2\cup K_3}\sigma(K_1\cup K_3,K_2\cup K_3)+\nonumber\\
&\textstyle+\sum^*_{h\not\in K_3}F_{K_1\cup K_3}(\partial_{\beta_h}G_{K_2\cup K_3\cup\{h\}})\sigma(K_1\cup K_3,K_2\cup K_3)\nonumber
\end{align}
and
\begin{align}
\partial_{\beta_h}(FG)_{K\cup\{h\}}=&\textstyle\sum^*_{h\not\in K_3}(\partial_{\beta_h}F_{(K_1\cup\{h\})\cup K_3})G_{K_2\cup K_3}\sigma((K_1\cup\{h\})\cup K_3,K_2\cup K_3)+\nonumber\\
&\textstyle+\sum^*_{h\not\in K_3}F_{(K_1\cup\{h\})\cup K_3}(\partial_{\beta_h}G_{K_2\cup K_3})\sigma((K_1\cup\{h\})\cup K_3,K_2\cup K_3)+\nonumber\\
&\textstyle+\sum^*_{h\not\in K_3}(\partial_{\beta_h}F_{K_1\cup K_3})G_{(K_2\cup\{h\})\cup K_3}\sigma(K_1\cup K_3,(K_2\cup\{h\})\cup K_3)+\nonumber\\
&\textstyle+\sum^*_{h\not\in K_3}F_{K_1\cup K_3}(\partial_{\beta_h}G_{(K_2\cup\{h\})\cup K_3})\sigma(K_1\cup K_3,(K_2\cup\{h\})\cup K_3)=\nonumber\\
=&\textstyle\sum^*_{h\not\in K_3}(\partial_{\beta_h}F_{K_1\cup K_3\cup\{h\}})G_{K_2\cup K_3}\sigma(K_1\cup K_3\cup\{h\},K_2\cup K_3)+\label{eq:chain2}\\
&\textstyle+\sum^*_{h\not\in K_3}F_{K_1\cup K_3\cup\{h\}}(\partial_{\beta_h}G_{K_2\cup K_3})\sigma(K_1\cup K_3\cup\{h\},K_2\cup K_3)+\nonumber\\
&\textstyle+\sum^*_{h\not\in K_3}(\partial_{\beta_h}F_{K_1\cup K_3})G_{K_2\cup K_3\cup\{h\}}\sigma(K_1\cup K_3,K_2\cup K_3\cup\{h\})+\nonumber\\
&\textstyle+\sum^*_{h\not\in K_3}F_{K_1\cup K_3}(\partial_{\beta_h}G_{K_2\cup K_3\cup\{h\}})\sigma(K_1\cup K_3,K_2\cup K_3\cup\{h\}).\nonumber
\end{align}
It follows that, if the following chain of equalities
\begin{align}
\sigma(K_1\cup K_3\cup\{h\},K_2\cup K_3\cup\{h\})&=-\sigma(K_1\cup K_3,K_2\cup K_3\cup\{h\})=\nonumber\\
&=-\sigma(K_1\cup K_3\cup\{h\},K_2\cup K_3)=\label{eq:2}\\
&=-\sigma(K_1\cup K_3,K_2\cup K_3)\nonumber
\end{align}
holds for all $(K_1,K_2,K_3)\in\mscr{D}(K)$ with $h\not\in K_1\cup K_2\cup K_3$, then $\partial_{\alpha_h}(FG)_K=\partial_{\beta_h}(FG)_{K\cup\{h\}}$. Similar computations ensures that, if equalities \eqref{eq:2} hold, then $\partial_{\beta_h}(FG)_K=-\partial_{\alpha_h}(FG)_{K\cup\{h\}}$ as well; consequently, by implication $(\mr{ii})\Rightarrow(\mr{i})$ in Proposition \ref{prop:slice-regularity}, $f\cdot_\sigma g$ is slice regular. Note that if $\sigma=\sigma_\otimes^n$, then equalities \eqref{eq:2} are trivially verified; indeed,
\begin{align*}
\sigma_\otimes^n(K_1\cup K_3\cup\{h\},K_2\cup K_3\cup\{h\})&=(-1)^{|K_3\cup\{h\}|}=-(-1)^{|K_3|},\\
\sigma_\otimes^n(K_1\cup K_3,K_2\cup K_3\cup\{h\})&=\sigma_\otimes^n(K_1\cup K_3\cup\{h\},K_2\cup K_3)=\\
&=\sigma_\otimes^n(K_1\cup K_3,K_2\cup K_3)=(-1)^{|K_3|}.
\end{align*}
This gives another proof of the fact that $f\tenso g$ is slice regular. 

Suppose $\mr{b}=\EuScript{B}(\sigma)$ is a hypercomplex $\dsim$-product of $\R^{2^n}$ such that $f\cdot_\sigma g\in\mc{SR}(\OO_D,A)$ for all $f,g\in\mc{SR}(\OO_D,A)$. Let $K,H\in\pa(n)$. We have to prove that $\sigma(K,H)=(-1)^{|K\cap H|}$.

First, we show that $\sigma(K,H)=\sigma(K\cap H,K\cap H)$. Suppose that $K\setminus H\neq\emptyset$, and choose $h\in K\setminus H$. Denote $x_K$ and $x_H$ the monomial functions from $\OO_D$ to $A$ defined as $x_K:=[(x_k)_{k\in K}]$ and $x_H:=[(x_h)_{h\in H}]$. Note that, if $\mc{K}=\sum_{L\in\pa(n)}e_L\mc{K}_L$ and $\mc{H}=\sum_{M\in\pa(n)}e_M\mc{H}_M$ denote the stem functions such that $x_K=\I(\mc{K})$ and $x_H=\I(\mc{H})$, then $\mc{K}_L(z)=\alpha_{K\setminus L}\beta_L$ if $L\subset K$, $\mc{K}_L(z)=0$ if $L\not\subset K$, $\mc{H}_M(z)=\alpha_{H\setminus M}\beta_M$ if $M\subset H$ and $\mc{H}_M(z)=0$ if $M\not\subset H$, where $z=(\alpha_1+i\beta_1,\ldots,\alpha_n+i\beta_n)\in D$, $\alpha_P:=\prd_{p\in P}\alpha_p$ and $\beta_P:=\prd_{p\in P}\beta_p$ for all $P\in\pa(n)\setminus\{\emptyset\}$, and $\alpha_\emptyset=\beta_\emptyset:=1$. Moreover, by \eqref{eq:chain1} and \eqref{eq:chain2}, if we set
\begin{align*}
\textstyle\sum^\bullet&:=\textstyle\sum_{(K_1,K_2,K_3)\in\mscr{D}((K\setminus\{h\})\dsim H),h\not\in K_3},\\
\textstyle\sum'&:=\textstyle\sum_{K_3\in\pa(n),K_3\subset K\cap H},\\
\sigma^{(1)}(K_3)&:=\sigma\big(((K\setminus H)\setminus\{h\})\cup K_3,(H\setminus K)\cup K_3\big),\\
\sigma^{(2)}(K_3)&:=\sigma\big((K\setminus H)\cup K_3,(H\setminus K)\cup K_3\big),
\end{align*}
then it holds
\begin{align*}
\partial_{\alpha_h}(\mc{K}\mc{H})_{(K\setminus\{h\})\dsim H}(z)=&\textstyle\sum^\bullet(\partial_{\beta_h}\mc{K}_{K_1\cup K_3\cup\{h\}})\mc{H}_{K_2\cup K_3}\sigma(K_1\cup K_3,K_2\cup K_3)=\\
=&\textstyle\sum'(\alpha_{(K\cap H)\setminus K_3}\beta_{((K\setminus H)\setminus\{h\})\cup K_3})\alpha_{(K\cap H)\setminus K_3}\beta_{(H\setminus K)\cup K_3}\sigma^{(1)}(K_3)=\\
=&\textstyle\,\beta_{(K\dsim H)\setminus\{h\}}\sum'(\alpha_{(K\cap H)\setminus K_3})^2(\beta_{K_3})^2\sigma^{(1)}(K_3)
\end{align*}
and, similarly,
\begin{align*}
\partial_{\beta_h}(\mc{K}\mc{H})_{K\dsim H}(z)=&\textstyle\textstyle\sum^\bullet(\partial_{\beta_h}\mc{K}_{K_1\cup K_3\cup\{h\}})\mc{H}_{K_2\cup K_3}\sigma(K_1\cup K_3\cup\{h\},K_2\cup K_3)=\\
=&\textstyle\,\beta_{(K\dsim H)\setminus\{h\}}\sum'(\alpha_{(K\cap H)\setminus K_3})^2(\beta_{K_3})^2\sigma^{(2)}(K_3).
\end{align*}
By hypothesis, the polynomial functions $\partial_{\alpha_h}(\mc{K}\mc{H})_{(K\setminus\{h\})\dsim H}$ and $\partial_{\beta_h}(\mc{K}\mc{H})_{K\dsim H}$ in the variable $z=(\alpha_1+i\beta_1,\ldots,\alpha_n+i\beta_n)$ are equal on the non-empty open subset $D$ of $\C^n$. Consequently,  the coefficients $\sigma^{(1)}(K_3)$ and $\sigma^{(2)}(K_3)$ are equal for all $K_3\in\pa(n)$ with $K_3\subset K\cap H$. In particular, the case $K_3=K\cap H$ implies that $\sigma(K\setminus\{h\},H)=\sigma(K,H)$ for all $h\in K\setminus H$. Since by hypothesis $\partial_{\beta_h}(\mc{K}\mc{H})_{(K\setminus\{h\})\dsim H}$ is equal to $-\partial_{\alpha_h}(\mc{K}\mc{H})_{K\dsim H}$ as well, similar computations show that $\sigma(K,H\setminus\{h\})=\sigma(K,H)\,$ for all $h\in H\setminus K$. This proves that
\begin{equation}\label{eq:cap}
\text{$\sigma(K,H)=\sigma(K\cap H,K\cap H)\,$ for all $K,H\in\pa(n)$,}
\end{equation}
as desired. In particular, $\mr{b}$ turns out to be commutative. Since it is also associative and hypercomplex by hypothesis, Lemma \ref{lem:tensor-uniqueness-1} implies that $\mr{b}=\mr{b}_\otimes^n$.
\end{proof}

We conclude this section with a result regarding slice regularity of pointwise products. For each $\ell\in\{1,\ldots,n\}$, we denote $\pi_\ell:A^n\to A$ the natural projection $\pi_\ell(x_1,\ldots,x_n):=x_\ell$.

\begin{lemma}\label{lem:sr-prod}
Let $\ell\in\{1,\ldots,n\}$, let $E_\ell$ be an open subset of $\C$ invariant under the complex conjugation of $\C$ and such that $\pi_\ell(\OO_D)\subset E_\ell$, and let $f:\OO_{E_\ell}\to A$ be a slice preserving slice regular function (in one variable), that is $f\in\mc{SR}_\R(\OO_{E_\ell},A)$. Let $H\in\pa(n)$ be such that $\ell\leq h$ for all $h\in H$, and let $g:\OO_D\to A$ be a $H$-reduced slice regular function (in $n$ variables). Define the function $p:\OO_D\to A$ by
\[
p(x):=f(x_\ell)g(x)
\]
for all $x= (x_1,\ldots,x_n)\in\OO_D$, where $f(x_\ell)g(x)$ is the product of $f(x_\ell)$ and $g(x)$ in $A$. Then $p$ belongs to $\mc{SR}(\OO_D,A)$.
\end{lemma}
\begin{proof}
Let $f_*:\OO_D\to A$ be the function $f_*(x_1,\ldots,x_n):=f(x_h)$. Note that $f_*$ is a $\ell$-reduced slice regular function. By Proposition \ref{prop:reduced-prod}, we have that $p=f_*\tenso g$. By \eqref{eq:L2-tensor}, $p$ is slice regular.
\end{proof}


\subsection{Splitting decomposition of slice regular functions}\label{subsection:split-deco}

By Assumption \ref{assumption1}, the set $\cS_A$ is non-empty. Thanks to Artin's theorem, the left multiplication by an element of $\cS_A$ induces a complex structure on $A$, so the real dimension of $A$ is even and positive, say $\dim_\R(A)=2u+2$ for some $u\in\N$. More precisely, if $J\in\cS_A$, then the addition of $A$ together with the complex scalar multiplication $\C_J\times A\to A$, sending $(c,a)$ into the product $ca$ in $A$, defines a structure of $\C_J$-vector space. If $\{J_1,\ldots,J_u\}$ is a $\C_J$-vector basis of $A$, then $\{1,J,J_1,JJ_1,\ldots,J_u,JJ_u\}$ is a real vector basis of $A$, see \cite[Lemma 2.3]{PowerSeries}. A real vector basis of $A$ of this form is called \textit{splitting basis} of $A$ associated with $J$.    

In the next result we use Definition \ref{def:OmegaDI}.  

\begin{proposition}\label{prop:splitting}
Let $f\in\mc{S}^1(\OO_D,A)$, let $J\in\cS_A$ and let $\{1,J,J_1,JJ_1,\ldots,J_u,JJ_u\}$ be a splitting basis of $A$ associated with $J$. Denote $\{f_{k,\ell}:\OO_D(J)\to\R\}_{k\in\{1,2\},\ell\in\{1,\ldots,u\}}$ the unique real-valued functions on $\OO_D(J)$ such that $f_J=\sum_{\ell=0}^u(f_{1,\ell}J_\ell+f_{2,\ell}JJ_\ell)$, where $J_0:=1$. Define the $\C_J$-valued functions $\{f_\ell:\OO_D(J)\to\C_J\}_{\ell=0}^u$ by setting $f_\ell:=f_{1,\ell}+f_{2,\ell}J$, in such a way that $f_J=\sum_{\ell=0}^uf_\ell J_\ell$. The following assertions are equivalent:
\begin{itemize}
 \item[$(\mr{i})$] $f$ is slice regular.
 \item[$(\mr{ii})$] For each $\ell\in\{0,1,\ldots,u\}$, we have
\[
\frac{\partial f_{1,\ell}}{\partial\alpha_\ell}=\frac{\partial f_{2,\ell}}{\partial\beta_\ell}\;\;\text{ and }\;\;\frac{\partial f_{1,\ell}}{\partial\beta_\ell}=-\frac{\partial f_{2,\ell}}{\partial\alpha_\ell}\;\;\text{ on $\OO_D(J)$,}
\]
where $(\alpha_1+J\beta_1,\ldots,\alpha_n+J\beta_n)$ are the coordinates of $(\C_J)^n$.
 \item[$(\mr{iii})$] For each $\ell\in\{0,1,\ldots,u\}$, the function $f_\ell:\OO_D(J)\to\C_J$ is holomorphic, where $\OO_D(J)$ and $\C_J$ are equipped with the natural complex structures associated with (left) multiplication by~$J$.
\end{itemize}
\end{proposition}
\begin{proof}
Equivalence $(\mr{ii}) \Leftrightarrow (\mr{iii})$ is evident. Let $h\in\{1,\ldots,n\}$. Bearing in mind Artin's theorem, we have
\[
\textstyle
\partial_{\alpha_h}f_J+J\partial_{\beta_h}f_J=\sum_{\ell=0}^u(\partial_{\alpha_h}f_\ell+J\partial_{\beta_h}f_\ell)J_\ell.
\]
Thanks to equivalence $(\mr{i}) \Leftrightarrow (\mr{iii})$ in Proposition \ref{prop:slice-regularity}, we deduce that $f$ is slice regular if and only if $\partial_{\alpha_h}f_\ell+J\partial_{\beta_h}f_\ell=0$ on $\OO_D$ for all $\ell\in\{0,1,\ldots,u\}$ and $h\in\{1,\ldots,n\}$. The latter assertion is in turn equivalent to say that each $f_\ell$ is holomorphic. This proves equivalence $(\mr{i}) \Leftrightarrow (\mr{iii})$ and completes the proof.
\end{proof}

As a consequence, we deduce:

\begin{corollary} \label{cor:real-analyticity}
Every slice regular function is real analytic, i.e. $\mc{SR}(\OO_D,A) \subset \mscr{C}^{\omega}(\OO_D,A)$.
\end{corollary}
\begin{proof}
By Proposition \ref{prop:splitting}, if $f=\I(F)\in\mc{SR}(\OO_D,A)$ and $J\in\cS_A$, then $f_J\in\mscr{C}^\omega(\OO_D(J),A)$. By formula \eqref{eq:components-F}, it follows that $F\in\mr{Stem}^\omega(D,A\otimes\R^{2^n})$ or, equivalently, $f\in\mc{S}^\omega(\OO_D,A)$. Now Theorem \ref{thm:Cr}$(\mr{ii})$ ensures that $f\in\mscr{C}^\omega(\OO_D,A)$.  
\end{proof}


\subsection{Convergent power series, slice tensor and star products}\label{subsec:cps}

In the theory of slice functions in one variable, the slice product $f\cdot g$ between slice functions $f=\I(F)$ and $g=\I(G)$ is induced by the usual product between complex numbers on $\R^2=\C$. Indeed, the latter product determines a real algebra structure on $A\otimes\R^2$, which coincides with the tensor product $A\otimes\C$. Then such a tensor product is used to compute the pointwise product $FG$ and, finally, one defines $f\cdot g:=\I(FG)$. One of the main features of the slice product $f\cdot g$ is revealed by its algebraic nature: if $f$ and $g$ are polynomials or, more generally, convergent power series in one variable, then $f\cdot g$ coincides with the standard abstract product $f*g$, sometimes called star product of $f$ and $g$ in the noncommutative setting, see \cite[Section 5]{AIM2011}.
 
The real algebra $A\otimes\C^{\otimes n}$ defines the slice tensor product `$\,\tenso\,$' on slice functions in $n$ variables, see Definition \ref{def:stp}. In this way, it is natural to ask whether such a product coincides with the star product on polynomials and on convergent power series in $n$-variables as well.

The aim of this section is to answer affirmatively to this question.

Let us review the concept of convergent power series with coefficients in $A$. A formal power series $s$ in $n$ indeterminates $X=(X_1,\ldots,X_n)$ with coefficients in $A$ is a sequence $\{a_\ell\}_{\ell\in\N^n}$ of elements of $A$. We formally write: $s(X)=\sum_{\ell\in\N^n}X^\ell a_\ell$ and $X^\ell=X_1^{\ell_1}\cdots X_n^{\ell_n}$ if $\ell=(\ell_1,\ldots,\ell_n)$. If the set $\{\ell\in\N^n:a_\ell\neq0\}$ is finite, then $s$ is called a (formal) polynomial. Denote $A[[X]]=A[[X_1,\ldots,X_n]]$ the set of all such formal power series. We can add and multiply formal power series in a standard way: if $t(X)=\sum_{\ell\in\N^n}X^\ell b_\ell\in A[[X]]$, then $s+t$ and $s\ast t$ are the elements of $A[[X]]$ defined by
\begin{align*}
&(s+t)(X)\textstyle:=\sum_{\ell\in\N^n}X^\ell(a_\ell+b_\ell),\\
&(s\ast t)(X)\textstyle:=\sum_{\ell\in\N^n}X^\ell\big(\sum_{p,q\in\N^n,\,p+q=\ell}a_pb_q\big).
\end{align*}
Given $r\in\R$, we can also define the real scalar multiplication $sr$ by $(sr)(X):=\sum_{\ell\in\N^n}X^\ell(a_\ell r)$. These operations make $A[[X]]$ a real algebra.

\begin{assumption}\label{assumption:norm}
In what follows, we assume that $\|\cdot\|:A\to\R$ is a norm of $A$ such that
\begin{equation}\label{eq:norm=nx}
\|x\|=\sqrt{n(x)}\;\;\text{ for all $x\in Q_A$,}
\end{equation}
equivalently, $\|\alpha+J\beta\|=\sqrt{\alpha^2+\beta^2}$ for all $\alpha,\beta\in\R$ and $J\in\cS_A$. 
\end{assumption}

Examples of real alternative $^*$-algebra $A$ with a norm having property \eqref{eq:norm=nx} are as follows: $\hh$ and $\oo$ with the usual conjugation $x^c:=\overline{x}$ and the usual Euclidean norm; the real Clifford algebra $\R_n$ with signature $(0, n)$, with the Clifford conjugation \cite[Definition 3.7]{GHS}, and the usual Euclidean norm of $\R^{2^n}$ or the Clifford operator norm, as defined in \cite[(7.20)]{GilbertMurray}.

Assumption \ref{assumption:norm} implies that
\begin{equation}\label{eq:compactness-sa}
\text{$\cS_A$ is compact in $A$.}
\end{equation}
Indeed, $\cS_A=\{I\in A:t(I)=0,n(I)=1\}$ is closed in $A$ and it is contained in the compact subset $S:=\{a\in A:\|a\|=1\}$ of $A$.

We say that $s(X)=\sum_{\ell\in\N^n}X^\ell a_\ell\in A[[X]]$ is a \textit{convergent power series} if there exists a real number $M>0$ such that
\begin{equation}\label{eq:Mdelta}
\|a_\ell\|\leq M^{|\ell|} \;\; \text{ for all $\ell=(\ell_1,\ldots,\ell_n)\in\N^n$ with $\textstyle|\ell|=\sum_{h=1}^n\ell_h$.}
\end{equation}
Note that such a concept does not depend on the chosen norm of $A$. Indeed, $A$ has finite real dimension so all the norms of $A$ are equivalent. Define the real number $B:=\max_{x,y\in S}\|xy\|$.
Note that $B>0$; indeed, $B\geq\|(1\|1\|^{-1})\cdot(1\|1\|^{-1})\|=\|1\|^{-1}>0$. By the very definition of $B$, it follows at once that
\begin{equation}\label{eq:B}
\text{$\|xy\|\leq B\|x\|\|y\|\;$ for all $x,y\in A$.}
\end{equation}
Thanks to the latter inequality, it follows immediately that the set of all convergent power series is a real subalgebra of $A[[X]]$.

Suppose now that $s(X)=\sum_{\ell\in\N^n}X^\ell a_\ell$ satisfies \eqref{eq:Mdelta}. Let $\rho\in\R$ with $0<B\rho M=:\gamma<1$ and let $\mbb{B}_\rho:=\bigcap_{h=1}^n\{(x_1,\ldots,x_n)\in A^n:\|x_h\|<\rho\}$. Note that, if $x\in\mbb{B}_\rho$ and $x^\ell a_\ell\in A$ is defined as in Definition \ref{def:polynomials}, then $\|x^\ell a_\ell\|\leq B^{|\ell|}\rho^{|\ell|}\|a_\ell\|\leq\gamma^{|\ell|}$; consequently, the series $\sum_{\ell\in\N^n}\|x^\ell a_\ell\|=\sum_{h\in\N}(\sum_{\ell\in\N^n,|\ell|=h}\|x^\ell a_\ell\|)$ is dominated by the positive real term series $\sum_{h\in\N}(h+1)^n\gamma^h$, which converges in $\R$. This proves that, for each $x\in\mbb{B}_\rho$, the series $\sum_{\ell\in\N^n}x^\ell a_\ell$ converges in~$A$ to a point $s(x)$. We abuse notation denoting $s:\mbb{B}_\rho\to A$ the function from $\mbb{B}_\rho$ to $A$, sending $x$ into $s(x)$. We say that such a function $s$ is a \textit{sum} of the convergent power series $s(X)$. Note that 
the series $\sum_{\ell\in\N^n}x^\ell a_\ell$ converges to $s:\mbb{B}_\rho\to A$ uniformly on $\mbb{B}_\rho$.

\begin{proposition}\label{prop:sum}
Let $s:\mbb{B}_\rho\to A$ be a sum of a convergent power series $s(X)$. Then $s$ is a slice regular function and there exists a unique sequence $\{a_\ell\}_{\ell\in\N^n}$ in $A$ such that $s(x)=\sum_{\ell\in\N^n}x^\ell a_\ell$. In particular, this is true if $s$ is a polynomial function.
\end{proposition}
\begin{proof}
Corollary \ref{cor:fl} ensures that $s$ is a slice function. By Proposition \ref{prop:polynomials-sr}, the monomial function $M_\ell:\mbb{B}_\rho\to A$, sending $x$ into $x^\ell a_\ell$, is slice regular for each $\ell\in\N^n$. Choose $J\in\cS_A$. By Proposition \ref{prop:slice-regularity}, we know that $\partial_{\alpha_h}M_{\ell,J}+J\partial_{\beta_h}M_{\ell,J}=0$ on $\mr{B}:=\mbb{B}_\rho\cap(\C_J)^n$, where $M_{\ell,J}$ is the restriction of $M_\ell$ to $\mr{B}$. Since the series $\sum_{\ell\in\N^n}M_{\ell,J}$ converges to $s_J$ uniformly on $\mr{B}$ and both the series $\sum_{\ell\in\N^n}\partial_{\alpha_h}M_{\ell,J}$ and $\sum_{\ell\in\N^n}\partial_{\beta_h}M_{\ell,J}$ converge uniformly on $\mr{B}$ as well, for all $h\in\{1,\ldots,n\}$, we can differentiate $s_J$ term by term obtaining:
\[
\textstyle
\partial_{\alpha_h}s_J+J\partial_{\beta_h}s_J=\sum_{\ell\in\N^n}(\partial_{\alpha_h}M_{\ell,J}+J\partial_{\beta_h}M_{\ell,J})=0
\]
on $\mr{B}$. Using Proposition \ref{prop:slice-regularity} again, we deduce that $s$ is slice regular. Finally, note that $a_\ell=(\ell!)^{-1}D_\ell s(0)$, where $D_\ell$ denotes the partial derivative $\partial^{|\ell|}/\partial\alpha_1^{\ell_1}\cdots\partial\alpha_n^{\ell_n}$.
\end{proof}

\begin{remark}\label{rem:poly}
In the polynomial case, the uniqueness assertion in the statement of the preceding result can be improved. Let $s:\OO_D\to A$ be a polynomial function. \emph{If $D$ is open in $\C^n$, then there exists a unique sequence $\{a_\ell\}_{\ell\in\N^n}$ in $A$ such that the set $L:=\{\ell\in\N^n:a_\ell\neq0\}$ is finite and $s(x)=\sum_{\ell\in L}x^\ell a_\ell$ on $\OO_D$}. As $s$ is a polynomial, $s(x)=\sum_{\ell\in L}x^\ell a_\ell$ for some finite non-empty subset $L$ of $\N^n$ and for some $a_\ell$ in $A$. Assume $s=0$ on $\OO_D$. We have to show that each $a_\ell$ is equal to zero. By Assumption \ref{assumption1}, we know that $D$ and $\cS_A$ are non-empty. Let $J\in\cS_A$. Let $z=(z_1,\ldots,z_n)\in\OO_D(J)$, let $\ell=(\ell_1,\ldots,\ell_n)\in\N^n$ and $m=(m_1,\ldots,m_n)\in\N^n$. Bearing in mind Artin's theorem and the fact that the components $z_h$ of $z$ commute each other, we have that $z^\ell a_\ell=[(z_h^{\ell_h})_{h=1}^n,a_\ell]=(z^\ell)a_\ell$ and the complex derivative $D_ms_J:\OO_D(J)\to A$ of $s_J$ is well-defined and it is equal to $\sum_{\ell\in\N^n}\frac{\ell!}{(\ell-m)!}(z^{\ell-m})a_\ell$, where $D_ms_J:=\partial^{|m|}s_J/\partial z_1^{m_1}\cdots\partial z_n^{m_n}$. Since $s_J=0$ on $\OO_D(J)$, the same is true for each complex derivative $D_ms_J$, i.e. $D_ms_J=0$ on $\OO_D(J)$. We prove by induction on the cardinality of $L$ that $a_\ell=0$ for all $\ell\in L$. If $L$ is the singleton $\{\ell\}$, then $a_\ell=D_\ell s_J=0$. Suppose $L$ contains at least two elements. Choose $m\in L$ such that $|m|\geq|\ell|$ for all $\ell\in L$. Since $\frac{D_ms_J}{m!}=a_m$, we have that $a_m=0$ and $s_J(z)=\sum_{\ell\in L\setminus\{m\}}(z^\ell)a_\ell$. The cardinality of $L\setminus\{m\}$ is $l-1$ so by induction all $a_\ell=0$. \bs
\end{remark}

\begin{definition}
Let $s:\mbb{B}_{\rho_1}\to A$ and $t:\mbb{B}_{\rho_2}\to A$ be sums of  convergent power series $s(X)$ and $t(X)$, respectively. Let $\{a_p\}_{p\in\N^n}$ and $\{b_q\}_{q\in\N^n}$ be the unique sequences in $A$ such that $s(X)=\sum_{p\in\N^n}X^p a_p$ and $t(X)=\sum_{q\in\N^n}X^q b_q$, see Proposition \ref{prop:sum}. Given a sum $s\ast t:\mbb{B}_{\rho_3}\to A$ of the convergent power series $(s\ast t)(X)$, we say that $s\ast t$ is a \emph{star product} of $s$ and $t$.

Suppose that $D$ is open in $\C^n$, and $s,t:\OO_D\to A$ are polynomial functions. Let $\{a_p\}_{p\in\N^n}$ and $\{b_q\}_{q\in\N^n}$ be the unique sequences in $A$ such that $s(X)=\sum_{p\in\N^n}X^p a_p$ and $t(X)=\sum_{q\in\N^n}X^q b_q$, see Remark \ref{rem:poly}. The \emph{star product $s\ast t:\OO_D\to A$} of $s$ and $t$ is defined as for $s(X)$ and $t(X)$, that is $(s\ast t)(x):=\sum_{\ell\in\N^n}x^\ell(\sum_{p,q\in\N^n,\,p+q=\ell}a_pb_q)$. \bs
\end{definition}

\begin{proposition}\label{prop:star-tensor}
The following hold.
\begin{itemize}
 \item[$(\mr{i})$] Let $s,t:\mbb{B}_\rho\to A$ be sums of convergent power series and let $s\ast t:\mbb{B}_{\rho'}\to A$ be a star product of $s$ and $t$ with $0<\rho'\leq\rho$. Then $s\tenso t=s\ast t$ on $\mbb{B}_{\rho'}$.
  \item[$(\mr{ii})$] Let $s,t:\OO_D\to A$ be polynomial functions. Suppose $D$ is open in $\C^n$. Then $s\tenso t=s\ast t$ on $\OO_D$. In particular, given any $\ell,m\in\N^n$ and $a,b\in A$, we have:
\begin{equation}\label{eq:xaxb}
(x^\ell a)\tenso(x^mb)=(x^\ell a)\ast(x^mb)=x^{\ell+m}(ab)\;\; \text{ on $\OO_D$}.
\end{equation}
\end{itemize}
\end{proposition}
\begin{proof}
Let us start proving $(\mr{ii})$. Equip $\mc{S}(\OO_D,A)$ with the slice tensor product `$\,\tenso\,$'. Write `$\,\cdot\,$' in place `$\,\tenso\,$' for short. Let $\ell=(\ell_1,\ldots,\ell_n)$ and $m=(m_1,\ldots,m_n)$. Since the slice tensor product is hypercomplex, we can apply Lemma \ref{lem:polynomials}. As a consequence, in order to complete the proof, it suffices to show that $(x^{\bullet\ell}\cdot a)\cdot(x^{\bullet m}\cdot b)=x^{\bullet\ell+m}\cdot (ab)$. Let $h\in\{1,\ldots,n\}$. Note that the stem function $F=F_\emptyset+e_hF_h$ inducing $x_h^{\bullet\ell_h}$ has real-valued components $F_\emptyset$ and $F_h$. The same is true for $x_h^{\bullet m_h}$. By Lemma \ref{lem:tensor-center}, it follows at once that $x_h^{\bullet\ell_h}$ and $x_h^{\bullet m_h}$ belong to the center of $\mc{S}(\OO_D,A)$. Consequently, it holds $x^{\bullet\ell}\cdot x^{\bullet m}=x^{\bullet\ell+m}$ and
\[
(x^{\bullet\ell}\cdot a)\cdot(x^{\bullet m} \cdot b)=(x^{\bullet\ell}\cdot x^{\bullet m})\cdot(ab)=x^{\bullet\ell+m}\cdot(ab).
\]
This proves \eqref{eq:xaxb} and hence point $(\mr{ii})$. Passing $(\mr{ii})$ to the limit, we obtain $(\mr{i})$.
\end{proof}


\subsection{Slice regular functions and ordered analyticity}\label{sec:ordered analyticity}

Let $\|\cdot\|:A\to\R$ be a norm of $A$ satisfying Assumption \ref{assumption:norm}, let $\rho\in\R$ with $\rho>0$ and let $\mbb{B}_\rho:=\bigcap_{h=1}^n\{(x_1,\ldots,x_n)\in A^n:\|x_h\|<\rho\}$, as in Section \ref{subsec:cps}.

\begin{theorem}
A function $f:\mbb{B}_\rho\to A$ is slice regular if and only if $f$ is a sum of a convergent power series $\sum_{\ell\in\N^n}X^\ell a_\ell$ with coefficients in $A$. Moreover, if $f$ is slice regular, then
\[
a_\ell=(\ell!)^{-1}\partial_\ell f(0)
\]
for all $\ell=(\ell_1,\ldots,\ell_n)\in\N^n$, where $\partial_\ell$ is the derivative $\partial^{|\ell|}/\partial x_1^{\ell_1}\cdots\partial x_n^{\ell_n}$ obtained by composing in any order $\ell_1$-times $\partial/\partial x_1$, $\ell_2$-times $\partial/\partial x_2$, $\ldots$, $\ell_n$-times $\partial/\partial x_n$.
\end{theorem}
\begin{proof}
By Proposition \ref{prop:sum}, if $f(x)=\sum_{\ell\in\N^n}x^\ell a_\ell$ on $\mbb{B}_\rho$, then $f$ is slice regular. Suppose that $f$ is slice regular. Let $J\in\cS_A$, let $\{1,J,J_1,JJ_1,\ldots,J_u,JJ_u\}$ be a splitting basis of $A$ associated with $J$ and let $\{f_h:\mbb{B}_\rho(J)\to\C_J\}_{h=1}^u$ be the family of $\C_J$-complex functions such that $f_J=\sum_{h=0}^uf_hJ_h$, where $J_0:=1$. By Proposition \ref{prop:splitting}, we know that each $f_h$ is holomorphic. As a consequence, each $f_h$ is the sum of a (unique) series $\sum_{\ell\in\N^n}x^\ell a_{\ell,h}$ with coefficients in $\C_J$, converging totally on compact subsets of $\mbb{B}_\rho(J)$ in the sense that $\sum_{\ell\in\N^n}\max_{x\in C}\|x^\ell a_{\ell,h}\|<+\infty$ for all compact subsets $C$ of $\mbb{B}_\rho(J)$. Set $a_\ell:=\sum_{h=0}^ua_{\ell,h}J_h$. Bearing in mind Artin's theorem, we deduce that $f_J$ is the sum of the series $\sum_{\ell_\in\N^n}x^\ell a_\ell$, converging totally on compact subsets of $\mbb{B}_\rho(J)$. In particular, if $M_\ell:\mbb{B}_\rho\to A$ denotes the monomial function $x^\ell a_\ell$ for each $\ell\in\N^n$, then $Q_C:=\sum_{\ell\in\N^n}\max_{x\in C}\|M_\ell(x)\|$ is finite for all compact subsets $C$ of $\mbb{B}_\rho(J)$. Choose arbitrarily a non-empty circular compact subset $C^*$ of $\mbb{B}_\rho$ and let $C:=C^*\cap(\C_J)^n\subset\mbb{B}_\rho(J)$. Let $x=(\alpha_1+J\beta_1,\ldots,\alpha_n+J\beta_n)\in C$ and let $y=(\alpha_1+I_1\beta_1,\ldots,\alpha_n+I_n\beta_n)\in C^*$ for some $I=(I_1,\ldots,I_n)\in(\cS_A)^n$. By \eqref{eq:compactness-sa}, we know that $\cS_A$ is compact in $A$ so $L:=\max\{\|1\|,\max_{a\in\cS_A}\|a\|\}$ is a positive real number. Let $B$ be a positive real number with property \eqref{eq:B}. Let $\ell\in\N^n$. By representation formula~\eqref{eq:f}, we have that $M_\ell(y)=2^{-n}\sum_{K,H \in \pa(n)}(-1)^{|K \cap H|}[I_K,[J^{-|K|}, M_\ell(x^{\, c,H})]]$. Consequently,
\begin{align*}
\|M_\ell(y)\|&\leq\textstyle2^{-n}\sum_{K,H \in \pa(n)}(BL)^{|K|+1}\|M_\ell(x^{\, c,H})\|\leq\\
&\leq\textstyle2^{-n}(BL)^{n+1}\sum_{K\in\pa(n)}\big(\sum_{H \in \pa(n)}\|M_\ell(x^{\, c,H})\|\big)\leq\\
&\leq\textstyle(BL)^{n+1}\sum_{H \in \pa(n)}\|M_\ell(x^{\, c,H})\|
\end{align*}
and hence
\begin{align*}
\textstyle
\max_{y\in C^*}\|M_\ell(y)\|&\leq\textstyle(BL)^{n+1}\sum_{H \in \pa(n)}\max_{x\in C}\|M_\ell(x^{\, c,H})\|=\\
&\textstyle=(BL)^{n+1}\sum_{H \in \pa(n)}\max_{x\in C}\|M_\ell(x)\|=\\
&\textstyle=2^n(BL)^{n+1}\max_{x\in C}\|M_\ell(x)\|.
\end{align*}
As a consequence, we deduce:
\begin{align*}
\textstyle\sum_{\ell\in\N^n}\max_{y\in C^*}\|M_\ell(y)\|&\leq\textstyle2^n(BL)^{n+1}Q_{C}<+\infty.
\end{align*}
This proves that the series $\sum_{\ell\in\N^n}x^\ell a_\ell$ converges to a function $s:\mbb{B}_\rho\to A$, totally on compact subsets of $\mbb{B}_\rho$. Since $f$ and $s$ are slice functions which coincide on $\mbb{B}_\rho(J)$, $f$ and $s$ coincide on the whole $\mbb{B}_\rho$ by Corollary \ref{cor:ip}. Finally, differentiating $s$ term by term, we obtain that $\partial_\ell f(0)=\partial_\ell s(0)=\ell !a_\ell$. We are done.
\end{proof}


\subsection{Cauchy integral formula for slice regular functions}\label{sec:Cauchy}


Throughout this section, we suppose Assumption \ref{assumption:openess} is true, i.e.\ $D$ is open in $\C^n$. Moreover, we fix $J\in\cS_A$.

\subsubsection{Some preparations}

Recall that $\phi_J:\C\to\C_J$ is the real algebra isomorphism $\phi_J(\alpha+i\beta):=\alpha+J\beta$. Choose bounded open subsets $E'_1,\ldots,E'_n$ of $\C$ invariant under complex conjugation and with $\mscr{C}^1$ boundaries $\partial E'_1,\ldots,\partial E'_n$. Let $h\in\{1,\ldots,n\}$. Define $E_h:=\phi_J(E_h')$ and $\partial E_h:=\phi_J(\partial E_h')$. Note that $\partial E_h$ is the boundary of $E_h$ in $\C_J$. Moreover, $\partial E_h$ is the disjoint union of a certain finite number, say $c_h$, of connected components, each homeomorphic to the circumference~$\cS^1$. Choose a $\mscr{C}^1$ parametrization $\xi_h:T_h\to\partial E_h$ of $\partial E_h$. Here $T_h$ is the disjoint union of $c_h$ intervals of the form $\{[a_{h,l},b_{h,l}]\}_{l=1}^{c_h}$, $\{\xi_h([a_{h,l},b_{h,l}])\}_{l=1}^{c_h}$ is the family of all connected components of $\partial E_h$, $\xi_h(a_{h,l})=\xi_h(b_{h,l})$ for all $l\in\{1,\ldots,c_h\}$, and each restriction of $\xi_h$ to $[a_{h,l},b_{h,l})$ is injective.

Consider the open subset $E$ of $(\C_J)^n$ and its distinguished boundary $\partial^*E$ given by
\[
E:=E_1\times\ldots\times E_n\;\;\;\text{ and }\;\;\;\partial^*E:=(\partial E_1)\times\ldots\times(\partial E_n).
\]
Define
\[
\OO(E):=\OO_{E'_1\times\ldots\times E'_n}.
\]
Note that $E=\OO(E)\cap(\C_J)^n$, i.e. $\OO(E)$ is the smallest circular set containing $E$. Denote $\mr{cl}(\OO(E))$ the closure of $\OO(E)$ in $(Q_A)^n$.

Given $v=(v_1,\ldots,v_n)\in(\C_J)^n$ and $K\in\pa(n)$, we define the element $v_K^c$ of $A$ by setting
\[
v_K^c:=
\left\{
 \begin{array}{ll}
1 & \text{ if  $K=\emptyset$,}
\vspace{.3em}
\\
\textstyle\prod_{h\in K}(v_h)^c  & \text{ if $K\neq\emptyset$.}
 \end{array}
\right.
\]

Let $T:=T_1\times\cdots\times T_n$. We define the maps $\xi:T\to\partial^*E$ and $\dot{\xi}:T\to\C_J$ by
\[
\xi(t):=(\xi_1(t_1),\ldots,\xi_n(t_n))
\;\;\;\text{ and }\;\;\;
\dot{\xi}(t):=\dot{\xi}_1(t_1)\cdots\dot{\xi}_n(t_n)
\]
for all $t=(t_1,\ldots,t_n)\in T$, where $\dot{\xi}_h(t_h)\in\C_J$ denotes the derivative of $\xi_h$ at $t_h$ in $\C_J$, and $\dot{\xi}_1(t_1)\cdots\dot{\xi}_n(t_n)$ the product of the $\dot{\xi}_h(t_h)$'s in $\C_J$ or, equivalently, in $A$. Given any $K\in\pa(n)$, we define also the function $\xi_K^c:T\to\C_J$ by $\xi_K^c(t):=(\xi(t))_K^c$, i.e.,
\[
\xi_K^c(t):=
\left\{
 \begin{array}{ll}
1 & \text{ if  $K=\emptyset$,}
\vspace{.3em}
\\
\textstyle\prod_{h\in K}(\xi_h(t_h))^c  & \text{ if $K\neq\emptyset$,}
 \end{array}
\right.
\]
for all $t=(t_1,\ldots,t_n)\in T$.

We need also two variants of Definition \ref{def:[]}. Let $C:\{1,\ldots,n\}\to A$ be any function and let $a\in A$. We define the elements $[C,a]$ and $[C]$ of $A$ by setting
\[
\textstyle
[C,a]:=C(1)\big(C(2)\big(\cdots\big(C(n-1)\big(C(n)a\big)\big)\ldots\big)\big)
\]
and
\[
\textstyle
[C]:=[C,1]=C(1)\big(C(2)\big(\cdots\big(C(n-1)C(n)\big)\ldots\big)\big).
\]
In addition, given any map $\EuScript{C}:\{1,\ldots,n\}\to\mc{S}(\OO(E),A)$, we define the slice functions $[\EuScript{C},a]$ and $\EuScript{C}$ in $\mc{S}(\OO(E),A)$ by setting
\[
\textstyle
[\EuScript{C},a]_\tenso:=\EuScript{C}(1)\tenso\big(\EuScript{C}(2)\tenso\big(\cdots\tenso\big(\EuScript{C}(n-1)\tenso\big(\EuScript{C}(n)\tenso a\big)\big)\ldots\big)\big)
\]
and
\[
\textstyle
[\EuScript{C}]_\tenso:=[\EuScript{C},1]_\tenso=\EuScript{C}(1)\tenso\big(\EuScript{C}(2)\tenso\big(\cdots\tenso\big(\EuScript{C}(n-1)\tenso\EuScript{C}(n)\big)\ldots\big)\big);
\]
here $a$ (in particular $1$) is identified with the function $\OO(E)\to A$ constantly equal to $a$.

Recall that, given any $q\in Q_A$, the function $\Delta_q:Q_A\to A$ is the characteristic polynomial of $q$, i.e.\ $\Delta_q(p):=p^2-2\mr{Re}(q)p+n(q)$. Denote $\Gamma_n$ the non-empty open subset of $(Q_A)^n\times (Q_A)^n$ defined by
\begin{equation}\label{eq:Gamma}
\textstyle
\Gamma_n:=\bigcap_{h=1}^n\{(x,y)=((x_1,\ldots,x_n),(y_1,\ldots,y_n))\in (Q_A)^n\times (Q_A)^n:\Delta_{y_h}(x_h)\neq0\},
\end{equation}
and $\Gamma_n(J)$ the non-empty open subset  of $(Q_A)^n\times(\C_J)^n$ given by
\begin{align}
&\Gamma_n(J):=\Gamma_n\cap((Q_A)^n\times(\C_J)^n).
\end{align}


\subsubsection{The general case}

Let $y=(y_1,\ldots,y_n)\in\partial^*E$. For each $h\in\{1,\ldots,n\}$, let $\EuScript{C}_{y,h}:\OO(E)\to A$ be the Cauchy kernel in the variable $x_h$ w.r.t.\ $y_h$, i.e. the slice function defined as follows:
\begin{equation}\label{eq:cyh}
\EuScript{C}_{y,h}(x):=\Delta_{y_h}(x_h)^{-1}(y_h^c-x_h)\;\;\text{ for all $x=(x_1,\ldots,x_n)\in\OO(E)$,}
\end{equation}
where $y_h^c:=(y_h)^c$ and $\Delta_{y_h}(x_h)^{-1}$ is the inverse of $\Delta_{y_h}(x_h)$ in $\C_J$ or, equivalently, in $A$. We define $\EuScript{C}_y:\{1,\ldots,n\}\to\mc{S}(\OO(E),A)$ by
\begin{equation}\label{def:cy}
\EuScript{C}_y(h):=\EuScript{C}_{y,h}\;\;\text{ for each $h\in\{1,\ldots,n\}$}.
\end{equation}

\begin{definition}\label{def:cg}
Let $S$ be any subset of $A^n$ containing $\partial^*E$ and let $g:S\to A$ be any function. We define the function $C_g:\OO(E)\times T\to A$ by setting
\begin{equation}\label{eq:cg}
C_g(x,t):=[\EuScript{C}_{\xi(t)},{\dot{\xi}(t)}J^{-n}g(\xi(t))]_\tenso(x). \;\text{ \bs}
\end{equation}
\end{definition}

\begin{remark}\label{rem:sr-cauchy}
For each $t\in T$, the function $\OO(E)\to A$, $x\mapsto C_g(x,t)$ is slice regular. This follows immediately from Proposition \ref{prop:sr-prod} and the fact that each function $\EuScript{C}_{\xi(t),h}:\OO(E)\to A$ and the constant function $\OO(E)\to A$, $x\mapsto\dot{\xi}(t)J^{-n}g(\xi(t))$ are slice regular. \bs
\end{remark}

Our general Cauchy integral formula reads as follows. 

\begin{theorem}
Let $f:\OO_D\to A$ be a slice regular function. Suppose that $\mr{cl}(\OO(E))\subset\OO_D$. Then
\begin{equation}\label{eq:slice-cauchy}
f(x)=(2\pi)^{-n}\int_TC_f(x,t)\,dt\quad\text{ for all $x\in\OO(E)$,}
\end{equation}
where $dt=dt_1\cdots dt_n$.
\end{theorem}
\begin{proof}
Let $\{1,J,J_1,JJ_1,\ldots,J_u,JJ_u\}$ be a splitting basis of $A$, see above  Section \ref{subsection:split-deco}. Write $f_J:E\to A$ as follows: $f_J=\sum_{\ell=0}^uf_\ell J_\ell$ for some (unique) functions $f_\ell:E\to\C_J$. By Proposition \ref{prop:splitting}, we know that each $f_\ell$ is holomorphic. In this way, given any $x=(x_1,\ldots,x_n)\in E$, we can apply the classical Cauchy formula to $f_\ell$ obtaining
\begin{align*}
f_\ell(x)&=(2\pi J)^{-n}\int_{\partial^*E}\frac{f_\ell(y)}{(y_1-x_1)\cdots(y_n-x_n)}\,dy_1\cdots dy_n=\\
&=(2\pi)^{-n}\int_{\partial^*E}\EuScript{C}_{y,1}(x)\cdots\EuScript{C}_{y,n}(x)J^{-n}f_\ell(y)\,dy_1\cdots dy_n=\\
&=(2\pi)^{-n}\int_T\EuScript{C}_{\xi(t),1}(x)\cdots\EuScript{C}_{\xi(t),n}(x)J^{-n}f_\ell(\xi(t))\dot{\xi}(t)\,dt=\\
&=(2\pi)^{-n}\int_T\EuScript{C}_{\xi(t),1}(x)\cdots\EuScript{C}_{\xi(t),n}(x)\dot{\xi}(t)J^{-n}f_\ell(\xi(t))\,dt.
\end{align*}
Thanks to Artin's theorem, we deduce
\begin{align*}
f_J(x)&=\sum_{\ell=0}^u\left((2\pi)^{-n}\int_T\EuScript{C}_{\xi(t),1}(x)\cdots\EuScript{C}_{\xi(t),n}(x)\dot{\xi}(t)J^{-n}f_\ell(\xi(t))\,dt\right)J_\ell=\\
&=(2\pi)^{-n}\int_T\EuScript{C}_{\xi(t),1}(x)\cdots\EuScript{C}_{\xi(t),n}(x)\dot{\xi}(t)J^{-n}f(\xi(t))\,dt.
\end{align*}
Choose $t\in T$ and define $\psi_t\in\mc{S}(\OO(E),A)$ by $\psi_t(x):=C_f(x,t)$. Bearing in mind Lemma \ref{lem:subalgebras}, Proposition \ref{prop:slice-pointwise-products} and Definition \ref{def:cg}, we have:
\begin{align*}
(\psi_t)_J&=\big([\EuScript{C}_{\xi(t)},\dot{\xi}(t)J^{-n}f(\xi(t))]_\tenso\big)_J=\\
&=\big((\EuScript{C}_{\xi(t),1}\tenso\cdots\tenso\EuScript{C}_{\xi(t),n})\tenso(\dot{\xi}(t)J^{-n}f(\xi(t)))\big)_J=\\
&=(\EuScript{C}_{\xi(t),1})_J\cdots(\EuScript{C}_{\xi(t),n})_J\dot{\xi}(t)J^{-n}f(\xi(t)).
\end{align*}
This proves that $C_f(x,t)=\psi_t(x)=\EuScript{C}_{\xi(t),1}(x)\cdots\EuScript{C}_{\xi(t),n}(x)\dot{\xi}(t)J^{-n}f(\xi(t))$ for all $(x,t)\in E\times T$. As a consequence, we have:
\[
f_J(x)=(2\pi)^{-n}\int_TC_f(x,t)\,dt\;\;\text{ for all $x\in E$.}
\]
Thanks to Corollary \ref{cor:ip}, in order to complete the proof, it suffices to show that the function $\OO_E\to A$, $x\mapsto\int_T\psi_t(x)\,dt=\int_TC_f(x,t)\,dt$ is a slice function. Define $\mr{J}:=(J,\ldots,J)\in(\cS_A)^n$. Choose $w=(\alpha_1+J\beta_1,\ldots,\alpha_n+J\beta_n) \in E$ and $x=(\alpha_1+L_1\beta_1,\ldots,\alpha_n+L_n\beta_n) \in \OO_D$ for some $L=(L_1,\ldots,L_n)\in(\cS_A)^n$. By representation formula \eqref{eq:f}, we know that
\[
\textstyle
\psi_t(x)=2^{-n}\sum_{K,H \in \pa(n)}(-1)^{|K \cap H|}[\mr{J}_K,[L_K^{-1},\psi_t(w^{\, c,H})]]
\]
for all $t\in T$. Hence, it holds:
\begin{align*}
\int_T\psi_t(x)\,dt&=\int_T\textstyle\big(2^{-n}\sum_{K,H \in \pa(n)}(-1)^{|K \cap H|}[\mr{J}_K,[L_K^{-1},\psi_t(w^{\, c,H})]]\big)\,dt=\\
&=\textstyle2^{-n}\sum_{K,H \in \pa(n)}(-1)^{|K \cap H|}[\mr{J}_K,[L_K^{-1},\int_T\psi_t(w^{\, c,H})\,dt]].
\end{align*}
Corollary \ref{cor:sliceness-intrinsic} implies that the function $\OO_E\to A$, $x\mapsto\int_T\psi_t(x)\,dt=\int_TC_f(x,t)\,dt$ is slice, as desired.
\end{proof}

\begin{remark}\label{rem:cauchy1}
As a byproduct of the preceding proof, we have that
\[
C_f(x,t)=\EuScript{C}_{\xi(t),1}(x)\cdots\EuScript{C}_{\xi(t),n}(x)\dot{\xi}(t)J^{-n}f(\xi(t))
\]
for all $(x,t)\in E\times T$. \bs
\end{remark}

Our next aim is to write $C_f$ in term of pointwise products of $A$-valued functions.  

For each $(x,y)=((x_1,\ldots,x_n),(y_1,\ldots,y_n))\in\Gamma_n(J)$ and for each $K\in\pa(n)$, we denote $\EuScript{C}(x,y,K):\{1,\ldots,n\}\to A$ the function given by
\begin{equation} \label{eq:C_g}
\EuScript{C}(x,y,K)(h):=
\left\{
 \begin{array}{ll}
(\Delta_{y_h}(x_h))^{-1} & \text{ if  $h\in K$,}
\vspace{.5em}
\\
(\Delta_{y_h}(x_h))^{-1}x_h  & \text{ if $h\not\in K$.}
 \end{array}
\right.
\end{equation}

\begin{theorem}
Let $f:\OO_D\to A$ be a slice regular function. Suppose that $\mr{cl}(\OO(E))\subset\OO_D$. Then
\begin{equation}\label{eq:cfxt}
C_f(x,t)=\textstyle\sum_{K\in\pa(n)}(-1)^{n-|K|}\big[\EuScript{C}(x,\xi(t),K),\xi_K^c(t)\dot{\xi}(t)J^{-n}f(\xi(t))\big]
\end{equation}
for each $(x,t)\in\OO(E)\times T$.
\end{theorem}
\begin{proof}
Let $t=(t_1,\ldots,t_n)\in T$, let $\psi_t:\OO(E)\to A$ be the function $\psi_t(x):=C_f(x,t)$, let $y=(y_1,\ldots,y_n):=(\xi_1(t_1),\ldots,\xi_n(t_n))\in\partial^*E$ and let $a:=\dot{\xi}(t)J^{-n}f(\xi(t))$. Thanks to Remark \ref{rem:cauchy1} and Artin's theorem, given any $x\in E$, we have that 
\begin{align}
\psi_t(x)&=\EuScript{C}_{\xi(t),1}(x)\cdots\EuScript{C}_{\xi(t),n}(x)\dot{\xi}(t)J^{-n}f(\xi(t))=\frac{(y_1^c-x_1)\cdots(y_n^c-x_n)}{\Delta_{y_1}(x_1)\cdots\Delta_{y_n}(x_n)}\,a=\nonumber\\
&=\frac{\textstyle\sum_{K\in\pa(n)}y_K^c\prod_{h\in\{1,\ldots,n\}\setminus K}(-x_h)}{\Delta_{y_1}(x_1)\cdots\Delta_{y_n}(x_n)}\,a=\label{eq:psi}\\
&=\textstyle\sum_{K\in\pa(n)}(-1)^{n-|K|}\big(\prod_{h\in K}\Delta_{y_h}(x_h)^{-1}\big)\big(\prod_{h\in\{1,\ldots,n\}\setminus K}\Delta_{y_h}(x_h)^{-1}x_h\big)y_K^ca=\nonumber\\
&=\textstyle\sum_{K\in\pa(n)}(-1)^{n-|K|}\big[\EuScript{C}(x,\xi(t),K),y_K^ca].\nonumber
\end{align}
Since $\psi_t$ is a slice function, in order to complete the proof, it is sufficient to show that the function $\phi_K:\OO(E)\to A$, sending $x$ into $[\EuScript{C}(x,\xi(t),K),y_K^ca]$, is slice for all $K\in\pa(n)$. Indeed, if this is true, then the function $\sum_{K\in\pa(n)}(-1)^{n-|K|}\phi_K$ is slice as well, and by Corollary \ref{cor:ip} we are done. Let $K\in\pa(n)$. 
Note that, for each $h\in\{1,\ldots,n\}$, the functions $\OO(E)\to A$, $x=(x_1,\ldots,x_n)\mapsto\Delta_{y_h}(x_h^c)$ and $\OO(E)\to A$, $x=(x_1,\ldots,x_n)\mapsto\Delta_{y_h}(x_h^c)x_h$ are slice preserving and $h$-reduced. Consequently, an iterated application of Proposition \ref{prop:reduced-prod} implies that $\phi_K$ is the iterated slice tensor product of $n+1$ slice functions, which is a slice function as well.

The proof is complete.
\end{proof}

\begin{remark}\label{rem:psi}
Repeating the chain of equalities \eqref{eq:psi} with $a=1$, we deduce that
\[
\textstyle
\EuScript{C}_{y,1}(x)\cdots\EuScript{C}_{y,n}(x)=\sum_{K\in\pa(n)}(-1)^{n-|K|}\big[\EuScript{C}(x,y,K),y_K^c]
\]
for all $(x,y)\in E\times\partial^*E$. \bs
\end{remark}


\subsubsection{The associative case}

Recall the definition of $\EuScript{C}_y$ for each $y\in\partial^*E$, given in \eqref{def:cy}.

\begin{definition}
We define the \emph{slice Cauchy kernel for $E$} as the function $C:\OO(E)\times\partial^*E\to A$ given by
\begin{equation}
C(x,y):=[\EuScript{C}_y]_\tenso(x). \;\text{ \bs}
\end{equation}
\end{definition}

\begin{remark}
For each $y\in\partial^*E$, the function $\OO(E)\to A$, $x\mapsto C(x,y)$ is slice regular, see Remark \ref{rem:sr-cauchy}. \bs
\end{remark}

Given two continuous functions $p,q:\partial^*E\to A$, if $A$ is associative, then we define
\[
\int_{\partial^*E}p(y)dyq(y):=\int_Tp(\xi(t))\dot{\xi}(t)q(\xi(t))\,dt.
\]

In the associative setting our Cauchy integral formula assumes a quite familiar form.

\begin{theorem}
Let $f:\OO_D\to A$ be a slice regular function. Suppose that $\mr{cl}(\OO(E))\subset\OO_D$. If~$A$ is associative then
\begin{equation}\label{eq:slice-cauchy-2}
f(x)=(2\pi)^{-n}\int_{\partial^*E}C(x,y)J^{-n}dy f(y)\quad\text{ for all $x\in\OO(E)$,}
\end{equation}
and the slice Cauchy kernel $C$ can be expressed in terms of pointwise products as follows:
\begin{equation}\label{eq:c(x,y)}
C(x,y)=\textstyle\sum_{K\in\pa(n)}(-1)^{n-|K|}[\EuScript{C}(x,y,K),y_K^c]
\end{equation}
for all $(x,y)\in\OO(E)\times\partial^*E$.
\end{theorem}
\begin{proof}
Thanks to Remark \ref{rem:psi}, for all $y\in\partial^*E$, the slice functions $\OO(E)\to A$, $x\mapsto C(x,y)$ and $\OO(E)\to A$, $x\mapsto \sum_{K\in\pa(n)}(-1)^{n-|K|}[\EuScript{C}(x,y,K),y_K^c]$ are equal on $E$. By Corollary \ref{cor:ip}, they coincide on the whole $\OO(E)$. This proves \eqref{eq:c(x,y)}. Since $A$ is associative, \eqref{eq:cfxt} and \eqref{eq:c(x,y)} imply that $C_f(x,t)=C(x,\xi(t))\dot{\xi}(t)J^{-n}f(\xi(t))$ for all $(x,t)\in\OO(E)\times T$. Consequently, the right hand sides of formulas \eqref{eq:slice-cauchy} and \eqref{eq:slice-cauchy-2} coincide. 
\end{proof}


\vspace{1em}

\noindent {\bf Acknowledgement.} This work was supported by GNSAGA of INdAM, and by the grants ``Progetto di Ricerca INdAM, Teoria delle funzioni ipercomplesse e applicazioni'', and PRIN ``Real and Complex Manifolds: Topology, Geometry and holomorphic dynamics'' of MUR. 




\end{document}